\documentclass[a4paper, reqno, 11pt]{amsart}

\usepackage[english]{babel}
\usepackage{amsmath}
\usepackage{amssymb}
\usepackage{amsthm}
\usepackage{enumerate}
\usepackage{ifthen}
\usepackage{bbm}
\usepackage{color}
\provideboolean{shownotes} 
\setboolean{shownotes}{true}
\usepackage{hyperref} 
\usepackage{cleveref}
\usepackage{graphicx}
\usepackage{pgfplots}
\usepackage{tikz,tikz-3dplot} 
\usepackage{mathtools}
\usepackage{comment}
\usepackage{float}
\usepackage{geometry}
\usepackage[T1]{fontenc}
\usepackage{esint}
\newcommand{\margnote}[1]{
\ifthenelse{\boolean{shownotes}}%
{\marginpar{\raggedright\tiny\texttt{#1}}}%
{}%
}

\newcommand{\hole}[1]{
\ifthenelse{\boolean{shownotes}}%
{\begin{center} \fbox{ \rule {.25cm}{0cm}
\rule[-.1cm]{0cm}{.4cm} \parbox{.85\textwidth}{\begin{center}
\texttt{#1}\end{center}} \rule {.25cm}{0cm}}\end{center}}
{}
}
\usepackage{mathrsfs}
\newtheorem{thm}{Theorem}[section]

\newtheorem{prop}[thm]{Proposition}
\newtheorem{lem}[thm]{Lemma}
\newtheorem{cor}[thm]{Corollary}
\newtheorem{rem}[thm]{Remark}

\newtheorem{defn}[thm]{Definition}

\newcommand{\e}{\varepsilon}		       
\newcommand{\R}{\mathbb{R}}
\newcommand{\T}{\mathbb{T}}
\newcommand{\N}{\mathbb{N}}
\newcommand{\Z}{\mathbb{Z}}
\newcommand{\dive}{\mathop{\mathrm {div}}}
\newcommand{\Id}{\mathop{\mathrm {Id}}}
\newcommand{\curl}{\mathop{\mathrm {curl}}}

\newcommand{\de}{\mathrm{d}}	
\newcommand{\p}{\mathbb{P}}

\newcommand{\Zeta}{Z^{\ell,\theta}}

\newcommand{\ZetaB}{Z^{\ell',\theta}}
\newcommand{\ZetaZero}{Z^{0,\theta}}

\newcommand{\ZetaL}{Z^{\ell,\theta}}
\newcommand{\vn}{v_{n+1}}
\newcommand{\mn}{m_{n+1}}
\newcommand{\Ph}{\Phi^n}
\newcommand{\PhB}{\Phi^{n-1}}
\newcommand{\DL}{\langle D \rangle^{\ell}}

\numberwithin{equation}{section}



\begin{document}

\title[Magnetic relaxation via the stable manifold method]{Magnetic relaxation for the MHD equations via the stable manifold method}

\author[G. Ciampa]{Gennaro Ciampa}
\address[G.\ Ciampa]{DISIM - Dipartimento di Ingegneria e Scienze dell'Informazione e Matematica\\ Universit\`a  degli Studi dell'Aquila \\Via Vetoio \\ 67100 L'Aquila \\ Italy}
\email[]{\href{gciampa@}{gennaro.ciampa@univaq.it}}

\author[R. Luc\`a]{Renato Luc\`a}
\address[R. Luc\`a]{Université d’Orléans, Université de Tours, CNRS, IDP, UMR 7013, 45067 Orléans, France.}
\email[]{\href{rluca@}{renato.luca@univ-orleans.fr}}

\maketitle

\begin{abstract}
We prove that given any sufficiently small and regular solution $B$ of the stationary Euler equations there exists an infinite dimensional family of solutions $(u,b)$ of the non-resistive magnetohydrodynamics equations (MHD) that relax to $(0, B)$. More precisely, $(u,b) \to (0, B)$ exponentially fast as $t \to +\infty$. This family may be viewed as lying in the stable manifold of the non-resistive MHD equations around the equilibrium state $(0, B)$. The problem whether it actually coincides with the stable manifold remains open. As a byproduct of our result, we provide a large class of global regular solutions of the non-resistive MHD equations. Another consequence is that any sufficiently small and regular solution of the stationary Euler 
equation is (non-trivially) topologically {\it accessible} via MHD from a large class of magnetic fields according to the 
definition of Moffatt \cite{Moff85}
and, in this scenario, the topology of the magnetic lines is (entirely) preserved in the limit $t \to + \infty$. 
\end{abstract}

\tableofcontents

\section{Introduction and main result}

We consider the three-dimensional non-resistive magnetohydrodynamic (MHD) equations
\begin{equation}\label{eq:mhdMoreStandard}\tag{MHD}
\begin{cases}
\partial_t u + \mathbb{P} (u \cdot \nabla ) u - \Delta u =  \mathbb{P} (b \cdot \nabla ) b,\\
\partial_t b + (u \cdot \nabla ) b  = (b \cdot \nabla ) u,  \\
\dive u=\dive b =0,\\
u(0,\cdot)=u^\mathrm{in},\hspace{0.3cm} b(0,\cdot)= b^\mathrm{in}.
\end{cases}
\end{equation}
where
$u^\mathrm{in}, b^\mathrm{in}:\T^3\to\R^3$ are prescribed divergence-free initial data, namely 
$$
\dive u^\mathrm{in} = \dive b^\mathrm{in} = 0,
$$
and $\mathbb{P}$ denotes the Leray projector onto the closed subspace of divergence-free vector fields. Throughout the paper, $\T^3:=\R^3/ 2 \pi \Z^3$ denotes the three-dimensional flat torus of period $2 \pi$. 

By invariance under translations in time, there is no loss of generality into prescribing the initial datum at~$t=0$.
For a given time horizon $T > 0$, the unknowns are the
the velocity field $u:[0,T)\times\T^3\to\R^3$ and  
the magnetic field $b:[0,T)\times\T^3\to\R^3$. We call a solution {\em global} when $T= +\infty$.
The term {\em non-resistive} refers to the regime in which the magnetic resistivity vanishes while the fluid viscosity remains strictly positive. Without loss of generality, the kinematic viscosity coefficient, namely the coefficient of $\Delta u$, has been normalized to unity.

Introduced by Alfvén to describe plasmas \cite{Alfven}, these equations have remained fundamental to plasma physics ever since.

\subsection{Well-posedness}
The MHD system has been extensively studied in recent years by a mathematical point of view. 

In the resistive case, namely when a magnetic diffusion term ($\eta\Delta$) is added to the magnetic field equation, well-posedness of regular solutions is known in two space dimensions, while in three dimensions one can establish 
local well-posedness for large initial data as well as global well-posedness for small initial data in several functional spaces, see for instance \cite{DL72, ST} and the references therein. 
This situation is very much the same as for the Navier-Stokes equations. 

In the non-resistive setting considered here things change quite dramatically. The reason is that the problem becomes much more degenerate, in particular it admits as equilibria all the couples $(0, B)$ where $B$ is a stationary solution of the Euler equation.
However, local well-posedness results are available at an essentially sharp level of Sobolev regularity \cite{CCRR, FMRR, Feff}. On the other hand, even if we consider small initial data, the global behavior of the relative solutions is still very much unclear, except in the case in which the magnetic field is exactly $b=0$, in which case the problem reduces to the Navier-Stokes equations (for which a reasonable global theory is available in the small data setting, see \cite{GIP, KT12} and the references therein).
   
A major breakthrough in the global theory of \eqref{eq:mhdMoreStandard} arises when the initial magnetic field is a small perturbation of a non-zero constant 
background field $\bar{b} \in \mathbb{R}^3 \setminus
\{0\}$. The physical mechanism underlying this stabilization can be understood by linearizing
the system around the equilibrium $(u, b) = (0, \bar{b})$: perturbations of both $u$ and $b$
propagate as waves at the Alfvén speed $|\bar{b}|$ along the background field lines, dispersing
energy rather than allowing it to concentrate.
In \cite{AZ}, Abidi and Zhang proved that the equilibria of the form $(0, \bar{b})$ are stable, 
introducing a Lagrangian formulation that reduces the system to a damped wave equation which is degenerate only in the direction of the background field. Moreover, they proved that if one perturbs only the velocity field, the corresponding solution in fact relaxes to the equilibrium, in the limit $t \to \infty$.  
We also refer to \cite{CaiLei, CMS, HXY, LXZ, LZ, PZZ, RWXZ, WZ, XZ} and the references therein for some related interesting results.

It is worth noting that these results do not give any information about the stability of the equilibrium $(u,b) = (0,0)$, 
that is still very much open, nor on  the large time behavior of small-data solutions.

\subsection{Magnetic relaxation}
A remarkable property of the equations \eqref{eq:mhdMoreStandard} is that the topology of the integral lines of the magnetic field, that we will call {\em magnetic lines}, is preserved under the evolution, as far as the solutions remain sufficiently regular. 
In fact, the induction equation
\begin{equation}\label{induction}
\partial_t b+u\cdot\nabla b=b\cdot\nabla u
\end{equation}
tells exactly that the magnetic lines are transported by the flow associated to the velocity field~$u$. 
More precisely, the flow map $\Psi_{t_1,t_2}$ defined via 
\begin{equation}\label{DiffPsiDef}
\begin{cases}
\frac{\de}{\de t_2} \Psi_{t_1,t_2} (x) = u(t_2,\Psi_{t_1,t_2} (x)),\\
\Psi_{t_1,t_1} (x)=x,
\end{cases}
\end{equation}
is a diffeomorphism that maps the magnetic lines at time $t_1$ into the magnetic lines at times $t_2$. 
We recall that a magnetic line $\gamma$ at time $t$ is a solution of the ODE
\begin{equation}\label{fdsjkudjfghrere}
\frac{\de}{\de s} \gamma(s) = b(t,\gamma(s)).
\end{equation}

This topological constraint no longer holds in resistive MHD: magnetic diffusion breaks the {\em frozen-in} property, allowing the connectivity of magnetic field lines to change, a phenomenon known as {\em magnetic reconnection}. While magnetic reconnection has long been recognized as central to plasma physics, rigorous analytical constructions of smooth solutions to the resistive MHD equations that exhibit a genuine change in magnetic topology have only recently been obtained, see \cite{CCL, CL2, LP23}. Here, however, we restrict ourselves to the ideal case, in which the topology of the magnetic field is preserved for all times.

The problem of \emph{magnetic relaxation} concerns then the long-time behavior of $b(t)$ and can be roughly described as follows.
A straightforward computation shows that sufficiently regular solutions of \eqref{eq:mhdMoreStandard} satisfies the 
identity 
\begin{equation}\label{hfbEnergy}
\frac{\de}{\de t}\left(\| u(t,\cdot)\|_{L^2}^2 + \| b(t,\cdot)\|_{L^2}^2\right) = - 2\|\nabla u(t,\cdot)\|_{L^2}^2 ,
\end{equation}
so the energy $\| u(t,\cdot)\|_{L^2}^2 + \| b(t,\cdot)\|_{L^2}^2$ is dissipated until the velocity is non-constant (non zero if we assume zero average velocities). This suggests that, at least formally, we can expect\footnote{This can be also, again formally, justified noting that the energy is a Lyapunov function and then invoking the La Salle principle \cite{LaSalle}.} $v(t,\cdot) \to 0$ as $t \to \infty$ for zero average velocities. Thus, looking again at equations \eqref{eq:mhdMoreStandard}, the magnetic field $b(t,\cdot)$ should converge as $t \to \infty$ to a limit state $B$ that is a solution of the stationary Euler equations
\begin{equation}\label{nfjdkwEuler}\tag{MHS}
\mathbb{P}[(B \cdot \nabla)  B] = 0, \qquad \dive B = 0.
\end{equation}

These considerations induced Moffatt \cite{Moff85} to suggest to use the magnetic relaxation  
as an elegant strategy for constructing steady Euler flows of arbitrarily complex topology. 
To give any rigorous justification of the Moffatt program is at the moment out of reach. 

It is worth mentioning that the topological complexity of steady Euler flows has a long history going back to Arnold \cite{Arnold66} and even to Kelvin \cite{Kelvin}. A solution to the long standing open problem of constructing 
steady Euler flows, in fact Beltrami fields, which vortex lines realise any (finite) knot and link 
 was provided by Enciso and Peralta-Salas in \cite{Annals, Acta}, we also refer to the result 
 of Enciso, Peralta-Salas and Torres de Lizaur for the case of the 3D flat torus (and the 3D sphere) \cite{EPT}.

It is worth mentioning that there are other topology preserving models, for which the Moffatt program 
could be also developed, or at least stated. One can consider for instance the equations
\begin{equation}\label{eq:mre}\tag{MRE}
\begin{cases}
\partial_t b + (u\cdot\nabla) b = (b\cdot\nabla) u,\\
(-\Delta)^{\gamma}u = \mathbb{P} (b \cdot \nabla ) b, \\
\dive b = 0,\\
b(0,\cdot)= b^\mathrm{in},
\end{cases}
\end{equation}
with assigned divergence free initial datum $b^\mathrm{in}$, where $\gamma\geq 0$ is a regularisation parameter. The case $\gamma=0$ corresponds to the Darcy-type law originally proposed by Moffatt \cite{Moff85, Moff21}, while $\gamma=1$ yields a Stokes-type regularisation, see the work of Brenier \cite{Brenier}. 
Similar considerations lead to the heuristic that one can expect $u(t,\cdot) \to 0$ in the limit $t \to \infty$ and thus 
the magnetic field should relax to a solution of the stationary Euler equations.
In \cite{Brenier}, Brenier introduced a suitable notion of weak solution for \eqref{eq:mre} with $\gamma =0$, drawing on the theory of optimal transport (see also \cite{Brenier1989}).  In \cite{Vicol-relaxation}, Beekie, Friedlander and Vicol proved local well-posedness for all $\gamma\geq 0$ with initial data $b_0\in H^s$ with $s>d/2+1$, and global well-posedness when the regularisation is sufficiently strong, namely $\gamma>d/2+1$. This is relevant, since it allows at least to work in a framework where the Moffatt question is rigorously well defined.
Concerning the long-time behaviour, they proved in fact that the velocity $u(t,\cdot)$ converges to zero in strong norms as $t\to\infty$, yet this is (remarkably) not sufficient to establish convergence of $b$ to a steady Euler flow $B$, 
not even in a weak sense. 
A similar analysis can be presumably carried out for the hyper-dissipative MHD equations, namely equations 
\eqref{eq:mhdMoreStandard} when the Laplacian is replaced by $(-\Delta)^{\gamma}$ and $\gamma$ is sufficiently large. 


A different regularisation strategy was pursued by Constantin and Pasqualotto in \cite{Const-relaxation}. They replace the standard viscous MHD evolution by a Voigt approximation, which modifies the time derivative through higher-order spatial operators without introducing artificial diffusion.
Their main result is the following: for arbitrary initial data in bounded three-dimensional domains and on the torus, the Voigt-MHD system admits global smooth solutions whose infinite-time limits are regular, nontrivial \eqref{nfjdkwEuler} equilibria. Notably, these equilibria are generically \emph{not} Beltrami fields, demonstrating that the relaxation procedure reaches a genuinely large class of equilibrium configurations. Related stochastic extensions were subsequently investigated in \cite{AJPS}, considering a stochastic perturbation into the \eqref{eq:mre} framework. They consider a randomly forced resistive \eqref{eq:mre} with resistivity $\kappa>0$ and a force of order $\sqrt{\kappa}$, establish path-wise global well-posedness and the existence of invariant measures, and construct a \emph{random} MHS equilibrium $b\in H^1(\T^d)$ as the non-resistive limit $\kappa\to 0$ of statistically stationary solutions. In two dimensions, the invariant measure $\mu_0$ does not concentrate on any compact subset of $H^1(\T^2)$ with finite Hausdorff dimension; in particular, all realisations of the random MHS equilibrium are almost surely not finite Fourier mode solutions, providing probabilistic evidence for the complexity of the equilibrium set. 
The drawback of these approaches is that the magnetic lines are not longer frozen-in, thus there is no way to relate the topology of the limit state with that of the initial magnetic field, that is one of the main features of the Moffatt program. 

Despite these (partial) positive results, it has been shown in \cite{EP-top} that, in any axisymmetric toroidal domain $\Omega\subset\R^3$, a locally generic set of divergence-free vector fields cannot be topologically equivalent to any MHS equilibrium in $\Omega$, providing fundamental topological obstructions to Moffatt's program. The fields in this set are Morse-Smale on the boundary, do not admit a non-constant first integral, and exhibit fast growth of periodic orbits; in particular, this set is residual in the Newhouse domain. The key dynamical mechanism is that a vector field with a dense set of non-degenerate periodic orbits cannot be topologically equivalent to a generic MHS equilibrium, and this geometric obstruction is implemented analytically through a novel rigidity theorem for the relaxation of magnetic fields with complex orbit structure. This result shows that topological relaxation, even in the well-posed setting of \cite{Vicol-relaxation, Const-relaxation}, cannot succeed for a topologically generic initial datum.

Complementing the general obstruction results, recent work has established rigorous convergence of the \eqref{eq:mre} to specific classes of equilibria in lower-dimensional settings. Specifically, in \cite{Relaxation} the authors study the Darcy-type \eqref{eq:mre} ($\gamma=0$) and prove two convergence results: first, for a class of non-constant shear flows in a two-dimensional periodic channel; second, for a class of \emph{2.5D equilibria} in $\Omega\times\R$ with $\Omega\subset\R^2$, addressed via a geometric approach that applies both to periodic channels and bounded domains.
These results constitute, to date, among the few instances where the conjectured convergence of the MRE to genuine Euler equilibria has been rigorously verified.\\

Lastly, the compressible case has been analyzed in \cite{Kim}. This setting presents additional difficulties, as the frozen-in condition interacts non-trivially with density variations and acoustic waves. In particular, in \cite{Kim} it is considered a compressible \eqref{eq:mre} derived from compressible MHD by replacing the acceleration term with a Darcy-type regularisation. Under a planar symmetry ansatz, three results are established: local well-posedness for smooth initial data; magnetic relaxation for smooth perturbations of constant steady states; and the absence of vacuum formation or implosion prior to and at the time of a potential singularity. While the restriction to planar symmetry reduces the problem to a one-dimensional hyperbolic system, these results open the mathematical study of compressible magnetic relaxation and reveal mechanisms with no incompressible analogue.

\subsection{Main results}

We recall that a vector field $B$ is a stationary solution of the Euler equations iff \eqref{nfjdkwEuler} holds.
In the rest of the paper, $B$ will always denote a stationary solution of the Euler equations, often without further mention.

We are now in position to state our main result. In the following theorem $\varepsilon_1 \in (0,1)$ is  
a small quantity.

\begin{thm}
\label{RealMAINTHM}
Let $\ell \geq 4$ be an integer. There exists $\varepsilon_1 \in (0,1)$ such that 
for all $\varepsilon \in (0, \varepsilon_1)$ the following holds. 
Let $B:\T^3 \to \R^3$ a stationary solution to the Euler equation such that 
$$\|B\|_{H^{\ell +2}} \leq \varepsilon.$$
For any zero-average divergence-free vector field  
$y : \T^3 \to \R^3$ that satisfies
\begin{equation}\label{SmAss185437456} 
\|y\|_{H^{\ell}} \leq \frac{\varepsilon}{3},
\end{equation}
there exists a couple of zero-average divergence-free vector fields $(\phi_1(y), \phi_2(y))$ with 
\begin{equation}\label{fnewiu66438InitialDatum}
\|\phi_j(y)\|_{H^{\ell}} \leq \varepsilon, \quad j= 1, 2,
\end{equation}
such that the following holds. 
The  \eqref{eq:mhdMoreStandard} equations with initial datum
\begin{equation}\label{InitialDatumIn MainTHM}
(u^\mathrm{in}, b^\mathrm{in}) = (\phi_1(y), B + \phi_2(y)) 
\end{equation}
admit a global (classical) solution $(u,b)$ that relaxes to the equilibrium $(0,B)$. More precisely
\begin{equation}\label{Bound:QuantitativeRelaxation}
\sup_{t \geq 0} \, e^{\frac14 t}  \|u(t)\|_{H^{\ell}}  \leq (1 + C \varepsilon) \|y\|_{H^{\ell}} ,
\qquad
 \sup_{t \geq 0} \, e^{\frac14 t}\|b(t) - B\|_{H^{\ell}}  \leq   C \varepsilon \|y\|_{H^{\ell}}.
\end{equation}
where the constant $C$ only depends on $\ell$.
The topology of the magnetic lines is preserved in the limit
$t \to + \infty$, in the following sense.
For all $t_1 \geq 0$ there exists a volume preserving $C^1$ diffeomorphism 
$$
\Psi_{t_1, \infty} : x \in \T^3 \to \Psi_{t_1, \infty}(x) \in \T^3,
$$ 
that satisfies
$$
B(\Psi_{t_1, \infty}) = (\nabla \Psi_{t_1 ,\infty}) b(t_1),
$$
namely $B$ is the pushforward of $b(t_1)$ by $\Psi_{t_1 ,\infty}$. In fact  
\begin{equation}\label{pushfryuwerungf}
\Psi_{t_1, \infty} = \lim_{t_2 \to + \infty} \Psi_{t_1,t_2},\end{equation} 
where $\Psi_{t_1,t_2}$ is the flow of the velocity field defined in \eqref{DiffPsiDef}. 
Moreover, the map 
\begin{equation}\label{LambdaMapDEf}
\Lambda : y \to \Lambda(y) := (\phi_1(y), B + \phi_2(y))
\end{equation}
is injective and $\Lambda(0) = (0,B)$ correspond to the equilibrium. 

\end{thm}

\begin{rem}
There is nothing special in the exponent $1/4$ in \eqref{Bound:QuantitativeRelaxation}, except that it belongs to 
$(0,1)$. Indeed, the $1/4$ could be replaced by any $\theta \in (0,1)$
and we would get a (small) threshold $\varepsilon_1(\theta) >0$ and a constant $C(\theta)$ (in \eqref{Bound:QuantitativeRelaxation}) with the following property: 
$$
 \lim_{\theta \to 0^+} \varepsilon_1(\theta) = \lim_{\theta \to 1^-} \varepsilon_1(\theta)  = 0; 
$$ 
$$
 \lim_{\theta \to 0^+} C(\theta) = \lim_{\theta \to 1^-} C(\theta) = + \infty ; 
$$
this fact will be clear looking at the proof, in which we have kept track of the parameter $\theta$.  
\end{rem}

Theorem \ref{RealMAINTHM} gives no information about the stability of the equilibrium $(0,B)$, that remains a widely open problem.
However, we can prove a stability result once we restrict the initial data to the family that we have constructed in Theorem \ref{RealMAINTHM}. This is the content of the following.

\begin{thm}\label{ContMap}
Let $\ell \geq 4$ be an integer. There exists $\varepsilon_1 \in (0,1)$ such that 
for all $\varepsilon \in (0, \varepsilon_1)$ the following holds. 
Let $B:\T^3 \to \R^3$ be a stationary solution to the Euler equation such that 
$$\|B\|_{H^{\ell +2}} \leq \varepsilon.$$
For any couple of zero-average divergence-free vector fields   
$y, y' : \T^3 \to \R^3$ that satisfies
\begin{equation} 
\|y\|_{H^{\ell}}, \|y'\|_{H^{\ell}}  \leq \frac{\varepsilon}{3},
\end{equation}
the following holds. 
Let $(u, b)$ and $(u', b')$ be given by Theorem \ref{RealMAINTHM}, relatively to the vector fields~$y$ and~$y'$,
respectively.  
We have the following estimates 
\begin{equation}\label{fjisuhfd66835985u}
\|u - u'\|_{\ZetaZero} \leq   (1 + C \varepsilon)\|y - y'\|_{L^2}, \qquad 
\|b - b'\|_{\ZetaZero} \leq  
C \varepsilon \|y - y'\|_{L^2}. 
\end{equation}
In particular, the map $\Lambda$ defined in \eqref{LambdaMapDEf}
is uniformly continuous from $L^2$ to $L^2 \times L^2$. 
\end{thm}

It is worth noting that Theorem \ref{RealMAINTHM} is trivial only in the case $B=0$. In this case
$\Lambda(y) = (y,0)$ and we recover the solutions obtained via the 
Navier-Stokes small data theory. For any other~$B \neq 0$ we have provided indeed new global regular solutions of the non-resistive MHD equations that relax to the equilibrium~$(0,B)$. The injectivity of the function $\Lambda$ shows that we have indeed constructed an infinite dimensional family of solutions. In the next theorem we prove that also the set of the orbits is large, but first we need to introduce some notations.

Given any $y$ as in the assumption of Theorem \ref{RealMAINTHM}
we denote by $\mathcal{O}(y)$ the orbit that passes through $\Lambda(y)$, that is the set
$$
\mathcal{O}(y) := \bigcup_{t \geq 0} (u(t), b(t)),
$$ 
where $(u,b)$ is the solution given by Theorem \ref{RealMAINTHM}, with initial datum $\Lambda(y)$. Let $s \geq 0$ and consider the set\footnote{With a small abuse of notation, throughout the paper we simply write $F \in H^{\ell}$, rather than $F \in [H^{\ell}]^3$, to indicate that the components of the vector field $F$ are in $H^{\ell}$.}
$$
\mathbb{B}^s_{\rho}:= \left\{ F \in H^{s}(\T^3) : \quad \dive F=0, \quad \int_{\T^3} F =0, \quad  \|F\|_{H^{s}} \leq \rho \right\}.
$$

We have the following.  
\begin{thm}\label{OrbitsTHMPrequel} 
Let $\ell \geq 4$ be an integer. There exists $\varepsilon_2 \in (0,\varepsilon_1)$ such that 
for all $\varepsilon \in (0, \varepsilon_2)$ the following holds.  
Let $B:\T^3 \to \R^3$ a stationary solution to the Euler equation and let $\sigma \in (0,1)$ such that 
$$\|B\|_{H^{\ell +2}} \leq \varepsilon, \qquad \|B\|_{L^2} = \sigma \varepsilon.$$
Then, there is a subset $\Sigma$ of $\mathbb{B}^\ell_{\frac{\varepsilon}{3}}$ homeomorphic to $\mathbb{B}^\ell_{\frac{\varepsilon}{3}} \cap \partial \mathbb{B}^0_{\frac{\sigma \varepsilon}{3\cdot 2^{\ell + 1}}}$, such that the function 
$$
y \in \Sigma \to \mathcal{O}(y),
$$
is injective (thus different choices of $y$ inside the set $\Sigma$ produces different orbits).
\end{thm}

This theorem ensures that the family of solutions provided by Theorem \ref{RealMAINTHM} does not merely arise from parametrizing the same orbit multiple times. Indeed, there exists an infinite-dimensional subset of the parameter space such that distinct parameters generate distinct orbits. Hence, even after identifying solutions belonging to the same orbit, Theorem \ref{RealMAINTHM} yields an infinite-dimensional family of dynamically distinct solutions.\\

In \cite{Moff85}, Moffatt introduced the concept of {\em topological accessibility} which is weaker than topological equivalence, because it allows for the appearance of discontinuities in the limit field $b_\infty$. The definition is the following.
\begin{defn}\label{MoffatAccessibiklity}
Consider the induction equation \eqref{induction}.
A vector field $b_\infty$ is {\em topologically accessible} from $b^\mathrm{in}$ if $b_\infty=\lim_{t\to +\infty}b(t)$ where 
$b$ is a solution of \eqref{induction} for some divergence-free vector field $u$, under the additional  properties that
\begin{equation}\label{ObvCon1}
\int_0^\infty\left|\int_{\T^3}b(t,x)\cdot(b(t,x)\cdot\nabla )u(t,x)\de x\right| \de t<\infty .
\end{equation}
\end{defn}

We complement this definition with some additional conditions that exclude some
trivial instance of relaxations such as
the constant (in time) solutions~$(u,b):= (0,B)$ and solutions of the form
$(u,b):= (e^{-t}B,B)$, where $B$ is a Beltrami field\footnote{We recall that a Beltrami field is an eigenvector of the $\curl$ operator associated to a non-zero eigenvalue.}.

\begin{defn}\label{Def:triviality}
Let $(u, b)$ a solution of the MHD equations. We say that $(u,b)$ relaxes in a non trivial way to 
$(0, b_{\infty})$ if $\lim_{t\to +\infty} (u,b) =(0, b_{\infty})$ and if there
exists $(\bar{t}, \bar{x}) \in [0, +\infty) \times \T^3$ such that either
\begin{equation}\label{NonTRivialityCond} 
 b(\bar{t}, \bar{x}) \wedge u(\bar{t}, \bar{x})  \neq 0,
\end{equation} 
or
\begin{equation}\label{NonTRivialityCond2} 
(\partial_t b)(\bar{t}, \bar{x})  \neq 0.
\end{equation} 
 \end{defn}
 
The non-collinearity condition \eqref{NonTRivialityCond} prevents the trivial situation in which (the image of) each magnetic line
remains unchanged under the evolution. 
This fact is formalised in Proposition \ref{NonCollinearity}.
The condition \eqref{NonTRivialityCond2} prevents the trivial situation in which the magnetic field itself is constant in time. 
A consequence of our result is that any sufficiently small and regular stationary solution of the Euler equations
is topologically accessible via MHD in a highly non trivial way. Moreover, the topology is preserved in the limit
$t \to + \infty$, in the sense that the stationary Euler state $B$ is the push-forward of the magnetic field 
by the limit flow 
associated to the 
velocity field.  
 More precisely we have the following
\begin{cor}\label{MainCor}
Let $\ell \geq 4$ be an integer. 
Let~$B:\T^3 \to \R^3$ be a stationary solution to the Euler equation with $\|B\|_{H^{\ell +2}} = 1$ and 
define $\Omega := \{x \in \T^3 : B(x)  \neq 0\}$. There exists an open dense subset $E$ of $\Omega$
such that, for all compact subset $K$
of~$E$
there exists\footnote{The quantity $\varepsilon_3$ depends on $B$ and $K$.} $\varepsilon_3 \in (0, \varepsilon_1)$ such that for all
$\sigma \in (0,\varepsilon_3)$ there exists  
a (non-zero) divergence-free, zero-average vector field $y \in H^{\ell}(\T^3)$ that satisfies the assumptions of Theorem \ref{RealMAINTHM}
and such that the solution $(u,b)$ with initial datum
\begin{equation}\label{InitialDatumIn MainTHMtr467346gtr}
(u^\mathrm{in}, b^\mathrm{in}) = (\phi_1(y), \sigma B + \phi_2(y)) 
\end{equation}
given by Theorem \ref{RealMAINTHM}
relaxes in a non-trivial way to $(0, \sigma B)$. More precisely, for $(u,b)$
the condition~\eqref{ObvCon1} holds and the conditions \eqref{NonTRivialityCond}-\eqref{NonTRivialityCond2}
are both satisfied for all $(\bar{t}, \bar{x}) \in \{ 0\} \times K$.  
As proved in Theorem \ref{RealMAINTHM}, $\sigma B$ is the pushforward of $b(t_1)$ by the limit flow $\Psi_{t_1 ,\infty}$ defined in \eqref{pushfryuwerungf}, 
for all $t_1 \geq 0$.  
\end{cor}

\begin{rem}
The choice of the vector field $y$ in Corollary \ref{MainCor} is generic, in the sense that we can choose 
any $y$ in an appropriate residual set, as specified in Corollary
\ref{MainCorStronger}, that is in fact a stronger version of Corollary \ref{MainCor}.
\end{rem}

In \cite{FMRR}, it was left as an open problem whether the Beltrami fields constructed in \cite{Annals} could be obtained as relaxation limits of \eqref{eq:mhdMoreStandard}. Our theorem provides an affirmative answer to this question once it is reformulated in the case of the 3D flat torus. We proved, in particular, that the Beltrami fields constructed in \cite{EPT} are non-trivially topologically accessible; one needs to rescale the Beltrami fields by a sufficiently small factor, however this has no impact on 
the topological richness of the class of Beltrami fields taken into account.
It is also worth to recall that any finite knot and link can be realised by vortex lines or vortex tubes of the Beltrami fields constructed in~\cite{EPT}.  
Thus we see that  
the same knots and links are realised by 
magnetic lines or magnetic tubes of the relative solutions $b(t)$ given by Theorem \ref{RealMAINTHM}, at any time $t \geq 0$. 

The proof of Theorem \ref{RealMAINTHM} consists of implementing the 
Lyapunov-Perron stable manifold method in the PDEs framework.  This is far from being trivial as working with the MHD equations relevant issues arise. This fact can be seen as an explanation of why the stable manifold theory is much less developed for PDEs rather than for ODEs. These issues and the way to fix them are informally described in the next section, where we also give a general overview of the proof.

Before presenting the proof, it is worth emphasizing that the regularity assumptions in Theorem~\ref{RealMAINTHM} 
(and so in Corollary \ref{MainCor}) are deliberately generous. Our primary objective is to convey the underlying method in the clearest and most transparent setting possible. With additional technical effort, these assumptions could likely be weakened considerably.
%
Determining the minimal regularity threshold at which the techniques developed in this paper remain applicable is, however, a nontrivial problem. In the $L^2$-based Sobolev setting, in light of the local well-posedness 
result established in \cite{Feff}, it 
is natural to conjecture that the regularity assumptions could be lowered to $B \in H^\ell(\T^3), y \in H^{\ell - 1 + \delta}(\T^3)$, 
$\ell > \frac32$, $\delta >0$. Nevertheless, establishing such a result could require new ideas. It remains an interesting open problem.

We conclude the introduction mentioning another contribution \cite{LinStable} in which the stable/unstable manifold has been constructed for the Euler equations, under suitable spectral assumptions. Interestingly, as shown in \cite{LinStable}, these spectral assumptions are satisfied for a class of shear flows, in 3D.

\subsection{Strategy of the proof}
\label{sec:proofstrategy}
The main idea to construct the solutions of Theorem \ref{RealMAINTHM} is to implement the Lyapunov-Perron stable manifold method 
in the context of the MHD equations. Indeed, given a stationary solution of the Euler equation $B$,
one can check immediately that the equilibrium $(0, B)$ is a (constant in time) solution of the MHD equations.
It is thus natural to perturb the problem around the equilibrium introducing the new variables
\begin{equation}\label{fdnjasknfd6648378234hffg}
(u,\tilde b)= (u , b - B).
\end{equation}
This leads to the following, equivalent formulation, of MHD
\begin{equation}\label{fdhsbfewbfew6466462783642}
\begin{cases}
\partial_{t} u = \Delta u + \mathbb{P} \left( (B \cdot \nabla) \tilde b +  (\tilde b \cdot \nabla) B 
-  (u \cdot \nabla) u +  (\tilde b \cdot \nabla) \tilde b \right)\\
\partial_{t} \tilde b = (B \cdot \nabla) u - (u \cdot \nabla) B + (\tilde b \cdot \nabla) u - (u \cdot \nabla) \tilde b 
\end{cases}
\end{equation}
Linearising this problem around the equilibrium $(u, \tilde b) = (0,0)$ we obtain,
\begin{equation}\label{fmedksnfe8736726432894}
\begin{cases}
\partial_{t} u = \Delta u + \mathbb{P} \left( (B \cdot \nabla) \tilde b +  (\tilde b \cdot \nabla) B \right)
\\
\partial_{t} \tilde b = (B \cdot \nabla) u - (u \cdot \nabla) B
\end{cases}
\end{equation}
that is very hard to study perturbatively because the operator 
$$ 
\begin{pmatrix}
\Delta & \mathbb{P} \big( (B \cdot \nabla)  [\cdot]  + ([\cdot] \cdot \nabla)    B \big)
\\
 (B \cdot \nabla)  [\cdot]  - ([\cdot] \cdot \nabla)    B & 0
 \end{pmatrix}
$$ 
is degenerate. For instance one can immediately check that 
$ \begin{pmatrix}
0
\\
B
\end{pmatrix}$ 
is in the kernel of the operator.  


However, since $B$ is very small, we could try (at least at a first naive attempt) pretend to attack this problem perturbatively.
To this purpose, it is instructive to set $B=0$ into equation \eqref{fmedksnfe8736726432894}, that leads to  
 the following linear problem 
\begin{equation}\label{fdhsbfewbfew6466linear}
\begin{cases}
\partial_{t} u = \Delta u  
\\
\partial_{t} \tilde b =  0,
\end{cases}
\end{equation}
%

Now, given any $y : \T^3 \to \R^3$, the linear problem \eqref{fdhsbfewbfew6466linear} admits a family 
of explicit solutions
\begin{equation}\label{RealStableManifold}
(u, \tilde b) = (e^{t \Delta} y, 0).
\end{equation}
Moreover, if $y$ has zero mean, 
the solution \eqref{RealStableManifold} decays exponentially to the equilibrium $$(u, \tilde b) = (0,0).$$ 
This is a consequence of the well known exponential decay of the heat flow
\begin{equation}\label{fjdklasjdks64}
\| e^{t \Delta} y  \|_{H^s} \leq e^{-t} \|y\|_{H^s}, \qquad t \geq 0, \qquad s \in \R.
\end{equation}
that is indeed nothing more than Parseval identity and Poincar\'e inequality. 
This completely describes the linear stable manifold associated to the equilibrium $(0,0)$ of \eqref{fdhsbfewbfew6466linear}.

Thus, coming back to \eqref{fdhsbfewbfew6466462783642}, 
the Lyapunov-Perron method would suggest to look for a (nonlinear) stable manifold solving, for positive times $t \geq0$, the following integral problem
\begin{equation}\label{fdhsbfewbfew6466462783642Integral}
\begin{cases}
\displaystyle
 u(t,x) = e^{t \Delta} y(x) + \int_0^t e^{(t-s) \Delta} \mathbb{P} \left( (B \cdot \nabla) \tilde b +  (\tilde b \cdot \nabla) B 
-  (u \cdot \nabla) u +  (\tilde b \cdot \nabla) \tilde b \right)(s,x) \, \de s
\\ \displaystyle
\tilde b(t,x) = -\int_t^{\infty} 
\left( (B \cdot \nabla) u - (u \cdot \nabla) B + (\tilde b \cdot \nabla) u - (u \cdot \nabla) \tilde b \right)(s,x)\, \de s .
\end{cases}
\end{equation}

More precisely, because of \eqref{fjdklasjdks64}, 
we must look for solutions with exponential decay rate $e^{- \theta t}$, with $\theta \in (0,1)$. 
In the PDEs context, it would be reasonable trying to solve \eqref{fdhsbfewbfew6466462783642Integral} by an iterative method in a suitable 
functional space. If we choose to work in the Sobolev scale, a reasonable choice would be to look for a solution 
$$(u, \tilde b) \in  \Zeta \times \Zeta, \qquad \|u \|_{\Zeta} + \| \tilde b \|_{\Zeta} \ll 1, $$ 
where 
\begin{equation*}
\Zeta:=\left\{ F \in L^\infty([0, +\infty);H^\ell(\T^3)): \| F \|_{\Zeta}:=\sup_{t \geq 0}e^{\theta t}\|F(t,\cdot)\|_{H^\ell}<\infty\right\},
\end{equation*}
for some sufficiently large value of $\ell$, as we need the solution to have enough regularity.
Note that this would automatically imply that the solution relaxes  to the equilibrium exponentially fast, as in  
\eqref{Bound:QuantitativeRelaxation}.\\

Unfortunately (and interestingly) this does not work because of the loss of derivatives in the Duhamel term.
Note that the derivative loss appears in both the linear and nonlinear part of \eqref{fdhsbfewbfew6466462783642Integral}, which are 
\begin{equation}\label{DerLoss1}
\begin{pmatrix}
(B \cdot \nabla) \tilde b +  (\tilde b \cdot \nabla) B
\\
(B \cdot \nabla) u - (u \cdot \nabla) B
\end{pmatrix},
\end{equation} 
and
\begin{equation}
\begin{pmatrix}
(\tilde{b} \cdot \nabla) \tilde b -  (u \cdot \nabla) u
\\
(\tilde{b} \cdot \nabla) u - (u \cdot \nabla) \tilde{b}
\end{pmatrix},
\end{equation}
respectively. The derivative loss in \eqref{DerLoss1} prevents any attempt to construct a solution of \eqref{fdhsbfewbfew6466462783642Integral} in the space~$\Zeta \times \Zeta$. Thus, this term must be canceled by a suitable change of variables. To do so, we introduce the new variables
\begin{equation}\label{fdsjnkhfe663472y4gtbnjglwrsjbgdyuas}
\begin{pmatrix}
v
\\
m
\end{pmatrix}
:=
e^{- K} 
\begin{pmatrix}
u
\\
\tilde{b}
\end{pmatrix},
\end{equation}
where $e^K$ is close to the identity and $K$ is chosen in such a way that 
\begin{equation}\label{fnjdkskjnfjnd32}
e^{-K} L e^{K}  =  
\begin{pmatrix}
\Delta
& 0
\\
 0 &  0
\end{pmatrix}
 + \mathcal{A} 
\end{equation}
but now $\mathcal{A}$ is a matrix of pseudo differential operators of order zero.
The transformation $e^{K}$ is introduced in Section \ref{Sec;Hom}, where formula \eqref{fnjdkskjnfjnd32} is proved. It is conceptually inspired by the normal-form reduction developed in \cite{CMS}, although its implementation in the present setting is substantially different. 
This is the key ``linear'' ingredient of our proof. \\

After this change of variable the equations for $(v, m)$ can be rewritten as
\begin{equation}\label{mannaggia}
\begin{cases}
\displaystyle
 v(t) = e^{t \Delta} y + \int_0^t e^{(t-s) \Delta} \left( \mathcal{A}_{11} v + \mathcal{A}_{12} m 
 + \mathbb{P} \left(  (m \cdot \nabla) m -  (v \cdot \nabla) v + \mbox{``lower order''} \right)  \right)(s) \, \de s,
\\ \displaystyle
m(t) = - \int_t^{\infty} 
\left( \mathcal{A}_{21} v + \mathcal{A}_{22} m  + (m \cdot \nabla) v - (v \cdot \nabla) m + \mbox{``lower order''} \right) (s)\, \de s.
\end{cases}
\end{equation}
To be more precise, the term $(m \cdot \nabla) v - (v \cdot \nabla) m$ must be indeed 
$(m \cdot \nabla) V(v, m) - (v \cdot \nabla) V(v, m)$
where
$V(v, m)  = v + \mbox{``lower order''}$ (see equation \eqref{DEf:AdjustedTransport}). For the sake of simplicity, in this heuristic explanation of the 
method we will keep working with 
$v$ in place of the (correct) distorted vector field $V(v, m)$. 

This is still not enough to close the argument, as a derivative loss still appears in the nonlinear part of both the equations.
The key idea to solve this problem is to introduce a mixed Eulerian/Lagrangian formulation. In order to motivate this 
formulation, we observe that the presence of a diffusive term in the first equation should allow in principle to control
also the $\| \cdot \|_{L^{2}_tH^{\ell +1}}$ norm of $v$, by standard energy estimates. This suggests the following:
given $\varepsilon >0$ sufficiently small and given any divergence-free vector field $y$ with $\|y\|_{H^{\ell}} \leq \frac{\varepsilon}{2}$ one can try
to implement an iterative 
scheme where the approximate solution $(v_n, m_m)$ is, at each step of the iteration,
controlled by
 \begin{equation}\label{fjkukudfshfds56e33623689safhb}
 \|v_n\|_{\Zeta} +  \|m_n\|_{\Zeta} + \| v_n \|_{L^{2}_tH^{\ell +1}}  \leq \varepsilon,
 \end{equation}
 here the small parameter $\varepsilon >0$ must be independent of $n \in \N$.
  This would indeed be possible, if not for the term $- (v \cdot \nabla) m$ 
 which appears in the second equation. This is the only really dangerous derivative loss, as the derivative 
 falls on $m$ and in the second equation there is no help from the diffusion to fight against it.  
Fortunately, this term has a transport structure, so we can get rid of it by introducing Lagrangian coordinates with respect 
to the flow associated to $v$ (in fact, to be precise, we must consider the flow associated to the distorted vector field 
$V(v,m)$). In these mixed Eulerian/Lagrangian coordinates, the equations \eqref{mannaggia} become 
\begin{equation}\label{fdhsbfewbfew6466462783642Integral2}
\begin{cases}
\displaystyle
 v(t,x) = e^{t \Delta} y(x) + \int_0^t e^{(t-s) \Delta} \left( \mathcal{A}_{11} v + \mathcal{A}_{12} m 
 + \mathbb{P} \left(  (m \cdot \nabla) m -  (v \cdot \nabla) v + \mbox{``lower order''} \right)  \right)(s,x) \, \de s
\\ \displaystyle
m(t,x) = - \int_t^{\infty} 
\left( \mathcal{A}_{21} v + \mathcal{A}_{22} m  + (m \cdot \nabla) v  + \mbox{``lower order''} \right) (s,\Phi_{t,s}(x))\, \de s ;
\end{cases}
\end{equation}
where the flow $\Phi_{t,s}$ satisfies 
\begin{equation}\label{fdsjkudjfghrere}
\begin{cases}
\partial_s \Phi_{t,s} (x) = v(s,\Phi_{t,s}(x)),\\
\Phi_{t,t} (x)=x.
\end{cases}
\end{equation}
Again, to be precise, $\Phi_{t,s}$ should actually be defined as in \eqref{nFlowDefPreq}, 
namely replacing $v$ by the distorted vector field~$V(v,m)$.
The use of Lagrangian coordinates interacts particularly well with the Perron/Lyapunov method. We refer to \cite{LinStable} for another example in which Lagrangian coordinates were used to construct the stable/instable manifold for the Euler equations.

In this final formulation all the remaining derivative loss can be absorbed taking advantage of the diffusion from the 
first equation, and we can indeed define (see Section \ref{Sec:The iterative scheme}) a suitable iteration of approximate solutions $(v_n, m_n)$ 
for which the uniform (w.r.t. $n$) bound \eqref{fjkukudfshfds56e33623689safhb} holds.
This uniform bound is the bulk of our proof and is obtained, in two steps, in Sections \ref{Sec:Picard-1}-\ref{Sec:Picard}.

 Once \eqref{fjkukudfshfds56e33623689safhb} is proved, we can actually prove that the sequence $(v_n, m_n)$
converge to a solution~$(v,m)$ of the equations \eqref{fdhsbfewbfew6466462783642Integral2} that satisfies moreover 
 \begin{equation}
 \|v\|_{\Zeta} +  \|m\|_{\Zeta} + \| v \|_{L^{2}_tH^{\ell +1}}  \leq \varepsilon,
 \end{equation}
see Section \ref{Sec:Convergence of the iterative scheme}. Coming back to the old variables via \eqref{fdnjasknfd6648378234hffg}-\eqref{fdsjnkhfe663472y4gtbnjglwrsjbgdyuas} 
we have indeed constructed the family of   
solutions $(u,b)$ of the MHD equations of 
our main Theorem \ref{RealMAINTHM} (this is worked out in Section \ref{Sec:RealMAINTHM}).

\section{Notations}
\subsection{Some basic notation} 

In this paper, the symbol $0$ will denote the zero element of various algebraic structures, with its specific meaning always clear from the context. For natural numbers, we adopt the standard notations $\mathbb{N} := \{0, 1, 2, \ldots\}$ and $\mathbb{N}^* := \mathbb{N} \setminus \{0\}$. Unless otherwise specified, all functions and vector fields considered in this section are assumed to be measurable.

Given two vector fields $F, G : \mathbb{T}^3 \to \mathbb{R}^3$, we define the projection operators $\pi_1, \pi_2$ over the product space $\mathbb{R}^3 \times \mathbb{R}^3$ as
\begin{equation}\label{def:proiezioni}
\pi_1 \begin{pmatrix} F \\ G \end{pmatrix} := F, \quad \pi_2 \begin{pmatrix} F \\ G \end{pmatrix} := G.
\end{equation}
Moreover, $F \otimes G$ denotes the tensor product defined component-wise by
\[
(F \otimes G)_{i,j} := F_i G_j, \qquad i, j \in \{1, 2, 3\}.
\]
We adopt the standard convention that differential operators act on vector fields and tensors component-wise. Thus, the divergence of a tensor $T$ and the Laplacian of a vector field $F$ are respectively given by
\[
(\operatorname{div} T)_i := \sum_{j=1}^3 \partial_j T_{i,j}, \qquad (\Delta F)_i := \Delta F_i.
\]
Consequently, for divergence-free vector fields $F$ and $G$, we obtain
\begin{equation}\label{fnsjaksjfngdajskfdjnskDive}
(\operatorname{div} (F \otimes G))_i = \sum_{j=1}^3 \partial_j (F_i G_j) = \sum_{j=1}^3 G_j \partial_j F_i,
\end{equation}
which implies the identity
\begin{equation}\label{OtimesVSnabla}
\operatorname{div} (F \otimes G) = (G \cdot \nabla) F.
\end{equation}

For a function $f : \mathbb{T}^3 \to \mathbb{R}$ and an integer $\ell \in \mathbb{N}$, we define the Sobolev norm
\begin{equation}\label{Sob233}
\| f \|_{W^{\ell, \infty}} := \sum_{0 \leq |\alpha| \leq \ell} \| \partial^\alpha f \|_{L^\infty}.
\end{equation}
For a vector field $F$ or a tensor field $T$
these norms are extended by summing over their components. Thus we have, in particular
\begin{equation}\label{Sob234}
\| F \|_{W^{\ell, \infty}} := \sum_{j=1}^3 \| F_j \|_{W^{\ell, \infty}}, \quad
\| \nabla F \|_{W^{\ell, \infty}} := \sum_{k =1}^3 \sum_{j =1}^3 \| \partial_k F_j \|_{W^{\ell, \infty}}.
\end{equation}
and, similarly, given a map $\Phi : \mathbb{T}^3 \to \mathbb{T}^3$ and its
 Jacobian (matrix) $\nabla \Phi$, we set 
$$
\| \nabla \Phi \|_{W^{\ell, \infty}} := \sum_{k=1}^3 \sum_{j=1}^3 \| \partial_k \Phi_j \|_{W^{\ell, \infty}}.
$$

Given $s \in \mathbb{R}$, we define the $L^2$-based Sobolev spaces via the norm
\[
\| f \|^2_{H^s} := \sum_{k \in \mathbb{Z}^3} \langle k \rangle^{2s} |\widehat{f}(k)|^2,
\]
where $\langle k \rangle := \sqrt{1 + |k|^2}$. We also define the fractional derivative operator $\langle D \rangle^s$ via  $$\widehat{\langle D \rangle^s f}(k) := \langle k \rangle^s \widehat{f}(k),$$ so that $\|f\|_{H^s} = \|\langle D \rangle^s f\|_{L^2}$.  
For any $\ell \geq 0$, we can define generalised norms
\begin{equation}\label{fjewi75}
\| f \|_{W^{\ell, \infty}} := \| \langle D \rangle^\ell f \|_{L^\infty},
\end{equation}
These definitions extend to vector and tensor fields in the way that we have just mentioned. 
When $\ell \in \mathbb{N}$, the norms \eqref{fjewi75} are equivalent to the classical norms defined in \eqref{Sob233}-\eqref{Sob234}. 

For a pair of vector fields\footnote{We recall that throughout the paper we simply write $F \in H^{\ell}$, rather than $F \in [H^{\ell}]^3$, to indicate that the components of the vector field $F$ are in $H^{\ell}$.} $F, G \in H^\ell(\mathbb{T}^3)$, we set
\[
\| (F, G) \|_{H^\ell \times H^\ell} := \| F \|_{H^\ell} + \| G \|_{H^\ell}.
\]

Given $\ell \geq 0$ and $\theta \in (0, 1)$, we introduce the Banach space
\begin{equation}\label{Def:Z}
\Zeta := \Big\{ F \in L^\infty([0, +\infty); H^\ell(\mathbb{T}^3)) : \sup_{t \geq 0} e^{\theta t} \| F(t, \cdot) \|_{H^\ell} < \infty \Big\},
\end{equation}
equipped with the norm
\begin{equation}
\| F \|_{\Zeta} := \sup_{t \geq 0} e^{\theta t} \| F(t, \cdot) \|_{H^\ell}.
\end{equation}
We will also often work with the following space-time norm:
\[
\| \nabla F \|_{L^2_t H^\ell}^2 := \int_0^{+\infty} \| \nabla F (t, \cdot) \|^2_{H^\ell} \, \de t.
\]

\subsection{Littlewood-Paley decomposition}

Finally, We recall the Littlewood-Paley decomposition for periodic functions, following the notations of \cite{Tao book}. Let $\varphi \in \mathcal{S}(\mathbb{R}^3)$ be a Schwartz function such that its Fourier transform satisfies
\begin{equation}
\widehat{\varphi}(\xi) = 
\begin{cases}
1 & \text{for } |\xi| \leq 1, \\
0 & \text{for } |\xi| > 2.
\end{cases}
\end{equation}
For $N > 0$, we set $\varphi_N(x) := N^3 \varphi(N x)$. For $f \in L^2(\mathbb{T}^3)$, we define the Littlewood-Paley projectors
\begin{equation}\label{def:Sj}
P_{\leq N} f(x) := \sum_{k \in \mathbb{Z}^3} \widehat{\varphi_N}(k) \, \widehat{f}(k) \, e^{i k \cdot x},
\end{equation}
and $P_N := P_{\leq N} - P_{\leq N/2}$. We thus obtain the identity (valid in $L^2$, or point-wise if $f \in H^s(\mathbb{T}^3)$ with $s > 3/2$):
\[
f = P_{\leq 1} f + \sum_{N \in 2^{\mathbb{N}^*}} P_N f.
\]
Defining $\Phi_N(x) := \sum_{k \in \mathbb{Z}^3} \varphi_N(x+2\pi k)$, Poisson's summation formula yields $P_{\leq N} f = \Phi_N * f$. Since $\| \Phi_N \|_{L^p(\mathbb{T}^3)} \simeq \| \varphi_N \|_{L^p(\mathbb{R}^3)} \simeq N^{3(1 - 1/p)}$ for $N \geq 1$, we recover the classical Bernstein inequalities in the periodic setting. We also denote $P_{>N} := \operatorname{Id} - P_{\leq N}$ and recall the telescoping identities for $N \in 2^{\mathbb{N}}$:
\[
P_{\leq N} f = \sum_{M \leq N} P_M f, \qquad P_{> N} f = \sum_{M > N} P_M f.
\]

\begin{lem}[Bernstein's inequalities, \cite{Tao book}]
Let $N \geq 1$. For any $s \geq 0$ and $1 \leq p \leq q \leq \infty$, the following bounds hold:
\begin{align}
\| P_{\geq N} f \|_{L^p} &\lesssim_{s,p} N^{-s} \| D^s P_{\geq N} f \|_{L^p}, \label{eq:bern1} \\
\| D^s P_{\leq N} f \|_{L^p} &\lesssim_{s,p} N^s \| P_{\leq N} f \|_{L^p}, \\
\| P_{\leq N} f \|_{L^q} &\lesssim_{p,q} N^{3(\frac{1}{p} - \frac{1}{q})} \| P_{\leq N} f \|_{L^p}.
\end{align}
Moreover, for dyadic $N \in 2^{\mathbb{N}^*}$, we have
\begin{align}
\| P_N D^{\pm s} f \|_{L^p} &\sim_{p,s} N^{\pm s} \| P_N f \|_{L^p}, \\
\| P_N f \|_{L^q} &\lesssim_{p,q} N^{3(\frac{1}{p} - \frac{1}{q})} \| P_N f \|_{L^p}. \label{eq:bern5}
\end{align}
\end{lem}

We also recall that, by Plancherel's Theorem, equivalent characterizations of Sobolev norms hold:
\begin{equation}\label{definizione besov}
\| f \|_{\dot{H}^s} \sim_s \Big( \sum_N N^{2s} \| P_N f \|_{L^2}^2 \Big)^{1/2}, \qquad \| f \|_{H^s} \sim_s \| P_{\leq 1} f \|_{L^2} + \Big( \sum_{N > 1} N^{2s} \| P_N f \|_{L^2}^2 \Big)^{1/2}.
\end{equation}

\section{The homological equation}\label{Sec;Hom}

In this section, we seek an invertible change of variables capable of diagonalizing the principal part of our system, modulo a lower-order remainder. Throughout this section, we assume that all vector fields belong to $H^2(\T^3)$. This level of regularity ensures that all the operators under consideration are well defined. The results can be extended to spaces of lower regularity with minor modifications. 
First of all, we define the operators $\mathrm{p}$ and $\mathrm{q}$ as follows
\begin{equation}\label{DEf:P,Q}
\mathrm{p}[F] := \mathbb{P} \dive (F \otimes B ) = \mathbb{P} \left( (B \cdot \nabla) F \right), \qquad
\mathrm{q}[F] := \mathbb{P} \dive (B \otimes F ) = \mathbb{P} \left( (F \cdot \nabla ) B \right),
\end{equation}
where the second equalities hold thanks to the divergence-free constraints $\dive B = 0$ and $\dive F = 0$, respectively. Hereafter, the Laplacian $\Delta$ is understood to act component-wise. 

The primary goal of this section is to find a conjugation matrix $e^K$ that diagonalizes the operator 
\begin{equation}\label{SumupBOLD}
L := \begin{pmatrix}
\Delta & \mathrm{p} + \mathrm{q} \\
\mathrm{p} - \mathrm{q} & 0
\end{pmatrix}
\end{equation}
modulo a remainder $\mathcal{A}$ of order $0$. The crucial role of the operator $L$ will become apparent in the subsequent analysis. 

\begin{thm}\label{ConiugioTHM}
Let $K$ be the operator matrix defined by
\begin{equation}\label{SumupB}
K := \begin{pmatrix}
0 & - \Delta^{-1} \circ (\mathrm{p} + \mathrm{q}) \\
(\mathrm{p} - \mathrm{q}) \circ \Delta^{-1} & 0
\end{pmatrix}.
\end{equation}
Then, we have the conjugation identity
\begin{align}\label{Diagonalization}
e^{-K} L e^{K} &= \begin{pmatrix}
\Delta & 0 \\
0 & 0
\end{pmatrix} + \mathcal{A},
\end{align}
where the remainder $\mathcal{A}$ is given by
\begin{align}\label{Def:MathcalA}
\mathcal{A} &:= \begin{pmatrix} \frac12 M_1+L_1 +\frac12 T_1 & 0 \\ 0 & \frac12 M_2+L_2 + \frac12 T_2 \end{pmatrix} + \frac12\left[ \begin{pmatrix} M_1 + L_1 & - \mathrm{q} \\ \mathrm{q} & M_2 + L_2 \end{pmatrix}, K \right] \nonumber \\ 
&\quad + \frac{1}{2} \int_0^1 (1-s)^2 e^{-s K} \left[ \begin{pmatrix} -M_1 + T_1 & 0 \\ 0 & -M_2 + T_2 \end{pmatrix} + \left[ \begin{pmatrix} M_1 + L_1 & - \mathrm{q} \\ \mathrm{q} & M_2 + L_2 \end{pmatrix}, K \right], K \right] e^{s K} \, ds,
\end{align}
and the auxiliary operators are defined as
\begin{align}
M_1 &:= \mathrm{p} \circ \mathrm{p} \circ \Delta^{-1} + \Delta^{-1} \circ \mathrm{p} \circ \mathrm{p}, \label{Def:M1} \\
M_2 &:= -2 \mathrm{p} \circ \Delta^{-1} \circ \mathrm{p}, \label{Def:M2} \\
L_1 &:= \mathrm{q} \circ ( \mathrm{p} - \mathrm{q}) \circ \Delta^{-1} - \mathrm{p} \circ \mathrm{q} \circ \Delta^{-1} + \Delta^{-1} \circ \mathrm{q} \circ (\mathrm{p} - \mathrm{q}) - \Delta^{-1} \circ \mathrm{p} \circ \mathrm{q}, \label{Def:L1} \\
L_2 &:= 2 \mathrm{q} \circ \Delta^{-1} \circ (\mathrm{p} + \mathrm{q}) - 2 \mathrm{p} \circ \Delta^{-1} \circ \mathrm{q}, \label{Def:L2} \\
T_1 &:= \mathrm{p} \circ \mathrm{q} \circ \Delta^{-1} - \Delta^{-1} \circ \mathrm{q} \circ \mathrm{p}, \label{Def:N1} \\
T_2 &:= \mathrm{p} \circ \Delta^{-1} \circ \mathrm{q} - \mathrm{q} \circ \Delta^{-1} \circ \mathrm{p}. \label{Def:N2}
\end{align}
\end{thm}

\begin{rem}\label{RemWellDef}
The operator $\Delta^{-1}$ is well-defined throughout \eqref{SumupB}--\eqref{Def:N2}. Indeed, it either acts directly on divergence-free, zero-average vector fields, or it acts on the images of $\mathrm{p}$ and $\mathrm{q}$, which inherently satisfy $\int_{\mathbb{T}^3} \mathrm{p}[F] = \int_{\mathbb{T}^3} \mathrm{q}[F] = 0$.
\end{rem}


\begin{rem}
While the diagonal terms $L_j + T_j$ admit further algebraic simplifications, we omit them here as they do not affect the structural properties required for the subsequent analysis.
\end{rem}

\begin{proof}
The statement follows from the Lie expansion \eqref{Taylor} and a direct computation.
For the reader’s convenience, we give a proof in which the operator \eqref{SumupB} is defined, a posteriori, 
solving the associated homological equation. We thus define 
$$
K =
\begin{pmatrix}
0 & \Psi
\\
\Phi & 0
\end{pmatrix}.
$$
where the operators $\Psi$ and $\Phi$ will be determined later. We use the Lie expansion to write
\begin{equation}\label{Taylor}
e^{-K} L e^{K} =  L + [L, K] + \frac12[[L, K],K]  + \frac12 \int_{0}^1 (1-s)^2 e^{-s K} [[[L, K],K],K] e^{s K} \, \de s.
\end{equation}
Then, we compute the commutator
$$
[L, K] =  
\begin{pmatrix}
(\mathrm{p} + \mathrm{q}) \circ \Phi - \Psi \circ (\mathrm{p} - \mathrm{q}) 
& \Delta \circ \Psi  
\\
 - \Phi \circ \Delta & (\mathrm{p} - \mathrm{q}) \circ \Psi - \Phi \circ (\mathrm{p} + \mathrm{q})
\end{pmatrix},
$$
and thus 
$$
L + [L, K] = 
\begin{pmatrix}
\Delta + (\mathrm{p} + \mathrm{q}) \circ \Phi - \Psi \circ (\mathrm{p} - \mathrm{q}) 
& \mathrm{p} + \mathrm{q} + \Delta \circ \Psi  
\\
\mathrm{p} - \mathrm{q} - \Phi \circ \Delta & (\mathrm{p} - \mathrm{q}) \circ \Psi - \Phi \circ (\mathrm{p} + \mathrm{q})
\end{pmatrix}.
$$
The homological equations that allows to remove the off-diagonal terms can be solved simply inverting $\Delta$, namely
 \begin{equation}\label{HomEq}
\Psi := -\Delta^{-1} \circ (\mathrm{p} + \mathrm{q}), \qquad 
\Phi := (\mathrm{p} - \mathrm{q}) \circ \Delta^{-1} ,
\end{equation}
which gives
\begin{equation}\label{seeThisnmfdjsld}
 L + [L, K]  =  
\begin{pmatrix}
\Delta   & 0  
\\ 
0 & 0 
\end{pmatrix}
+
\begin{pmatrix}
M_1  & 0  
\\
0 & M_2
\end{pmatrix}
+ 
\begin{pmatrix}
L_1  & 0  
\\
0 & L_2 
\end{pmatrix},
\end{equation}
where we have defined the operators
$$
M_1:= \mathrm{p} \circ  \mathrm{p}  \circ \Delta^{-1} 
+ \Delta^{-1} \circ \mathrm{p}  \circ \mathrm{p},
$$
$$
M_2:= -2  \mathrm{p}  \circ \Delta^{-1} \circ \mathrm{p}  ,
$$
and
$$
L_1 := \mathrm{q} \circ ( \mathrm{p} - \mathrm{q}) \circ \Delta^{-1} 
- \mathrm{p} \circ \mathrm{q} \circ \Delta^{-1} 
+ \Delta^{-1} \circ \mathrm{q} \circ (\mathrm{p} - \mathrm{q}) 
- \Delta^{-1} \circ \mathrm{p}  \circ  \mathrm{q},
$$
$$L_2 
:= 
2    \mathrm{q} \circ \Delta^{-1} \circ (\mathrm{p} + \mathrm{q}) 
-2  \mathrm{p}  \circ \Delta^{-1} \circ \mathrm{q}.
$$
We also note that under the choice \eqref{HomEq} we have
\begin{equation}
 K = 
\begin{pmatrix}
0 & - \Delta^{-1} \circ \mathrm{p}
\\
\mathrm{p}  \circ \Delta^{-1} & 0
\end{pmatrix}
+
\begin{pmatrix}
0 & -\Delta^{-1} \circ   \mathrm{q}
\\
 - \mathrm{q} \circ \Delta^{-1} & 0
\end{pmatrix},
\end{equation}
as well as
$$
[L, K] = 
\begin{pmatrix}
(\mathrm{p} + \mathrm{q}) \circ (\mathrm{p} - \mathrm{q}) \circ \Delta^{-1} + \Delta^{-1} \circ (\mathrm{p} + \mathrm{q}) \circ (\mathrm{p} - \mathrm{q}) 
& - \mathrm{p} - \mathrm{q}
\\
 - \mathrm{p} + \mathrm{q}  & -2 (\mathrm{p} - \mathrm{q}) \circ \Delta^{-1} \circ (\mathrm{p} + \mathrm{q}) 
\end{pmatrix}.
$$
Notice that from \eqref{seeThisnmfdjsld} and \eqref{SumupBOLD} it follows
\begin{align*}
 [L, K]  = 
\begin{pmatrix}
0  & - \mathrm{p} 
\\
- \mathrm{p}  & 0
\end{pmatrix}
+
\begin{pmatrix}
M_1   &  - \mathrm{q}
\\
 \mathrm{q} & M_2 
\end{pmatrix}
+
\begin{pmatrix}
L_1  & 0  
\\
0 & L_2 
\end{pmatrix}.
\end{align*}
In order to compute $[[L, K],K]$, it is useful to note that  
\begin{align*}
\left[ 
\begin{pmatrix}
0 & - \mathrm{p}
\\ 
- \mathrm{p}, & 0
\end{pmatrix},
\begin{pmatrix}
0 & - \Delta^{-1} \circ \mathrm{p}
\\ 
\mathrm{p} \circ  \Delta^{-1} & 0
\end{pmatrix}
\right] 
& = 
\begin{pmatrix}
- \mathrm{p} \circ \mathrm{p} \circ \Delta^{-1} - \Delta^{-1} \circ \mathrm{p} \circ \mathrm{p}
& 0
\\ 
 0 & 2  \mathrm{p}  \circ \Delta^{-1} \circ \mathrm{p}
\end{pmatrix}
\\
&=
\begin{pmatrix}
-M_1 & 0
\\
 0 &  - M_2
\end{pmatrix}.
\end{align*}
and 
\begin{align*}
\left[ 
\begin{pmatrix}
0 & - \mathrm{p}
\\ 
- \mathrm{p}, & 0
\end{pmatrix},
\begin{pmatrix}
0 & - \Delta^{-1} \circ \mathrm{q}
\\ 
- \mathrm{q} \circ  \Delta^{-1}, & 0
\end{pmatrix}
\right] 
& = 
\begin{pmatrix}
 \mathrm{p} \circ \mathrm{q} \circ \Delta^{-1} - \Delta^{-1} \circ \mathrm{q} \circ \mathrm{p}
& 0
\\ 
 0 &   \mathrm{p}  \circ \Delta^{-1} \circ \mathrm{q} -  \mathrm{q}  \circ \Delta^{-1} \circ \mathrm{p}
\end{pmatrix}
\\
&=
\begin{pmatrix}
T_1 & 0
\\
 0 &   T_2
\end{pmatrix},
\end{align*}
where 
$$
T_1 : = \mathrm{p} \circ \mathrm{q} \circ \Delta^{-1} - \Delta^{-1} \circ \mathrm{q} \circ \mathrm{p},
$$
$$
T_2 : = \mathrm{p} \circ \Delta^{-1} \circ \mathrm{q} - \mathrm{q} \circ  \Delta^{-1} \circ  \mathrm{p}.
$$
Thus, we arrive to
\begin{equation}\label{fdhjskaduhf}
 [[L, K],K]  
= 
\begin{pmatrix}
-M_1 & 0
\\
 0 &  - M_2
\end{pmatrix}
+
\begin{pmatrix}
T_1  & 0
\\
0 & T_2
\end{pmatrix}
+
  \left[
\begin{pmatrix}
M_1   &  - \mathrm{q}
\\
 \mathrm{q} & M_2 
\end{pmatrix}
+  
\begin{pmatrix}
L_1  & 0  
\\
0 & L_2 
\end{pmatrix}, 
K \right].
\end{equation}

Summing \eqref{seeThisnmfdjsld}-\eqref{fdhjskaduhf} 
we find 
\begin{align}\nonumber
L + [L, K] + \frac12 [[L, K],K]  & =  
\begin{pmatrix}
\Delta
& 0
\\
 0 &  0
\end{pmatrix} 
+\frac12 \begin{pmatrix}
M_1 & 0\\
 0 &  M_2
\end{pmatrix} 
\\ \label{fdjskjgfhjdsk1}
& + 
\begin{pmatrix}
L_1 + \frac12 T_1 & 0
\\
0 & L_2 + \frac12 T_2 
\end{pmatrix}
+\frac12
  \left[
\begin{pmatrix}
M_1   &  - \mathrm{q}
\\
 \mathrm{q} & M_2 
\end{pmatrix}
+  
\begin{pmatrix}
L_1  & 0  
\\
0 & L_2 
\end{pmatrix}, 
K \right]
\end{align}
Recalling \eqref{Taylor} and using again \eqref{fdhjskaduhf} we arrive to the desired identity \eqref{Diagonalization}.
\end{proof}

\section{Stable manifold analysis}\label{Sec:Lyapunov}
It will be useful to rewrite the equation \eqref{eq:mhdMoreStandard} in the following form:
\begin{equation}\label{eq:mhd}
\begin{cases}
\partial_t u + \mathbb{P} \operatorname{div} (u \otimes u) - \Delta u =  \mathbb{P} \operatorname{div}  (b \otimes b),\\
\partial_t b +  \operatorname{div} (b \otimes u) =  \operatorname{div} (u \otimes b), \\
\operatorname{div} u = \operatorname{div} b = 0,\\
u(0) = u^\mathrm{in}, \quad b(0) = b^\mathrm{in}.
\end{cases}
\end{equation}
The main goal of this section is to rewrite the Lyapunov-Perron formulation for the stable variety associated to the MHD equations, after introducing the change of coordinates
\[
\begin{pmatrix} u \\ b \end{pmatrix} 
:=
\begin{pmatrix} 0 \\ B \end{pmatrix}
+
e^{K} 
\begin{pmatrix} v \\ m \end{pmatrix},
\]
where the operator $K$ was defined in Section \ref{Sec;Hom}. Given $v, m : \mathbb{T}^3 \to \mathbb{R}^3$, we define the non-linear operators
\begin{equation}\label{DefTildeN1}
\tilde N_1(v,m) :=  \pi_2 e^{K} \begin{pmatrix} v \\ m \end{pmatrix} \otimes \pi_2 e^{K} \begin{pmatrix} v \\ m \end{pmatrix} 
- \pi_1 e^{K} \begin{pmatrix} v \\ m \end{pmatrix} \otimes \pi_1 e^{K} \begin{pmatrix} v \\ m \end{pmatrix},
\end{equation}
\begin{equation}\label{DefTildeN2}
\tilde N_2(v,m) :=  \pi_1 e^{K} \begin{pmatrix} v \\ m \end{pmatrix} \otimes \pi_2 e^{K} \begin{pmatrix} v \\ m \end{pmatrix} 
- \pi_2 e^{K} \begin{pmatrix} v \\ m \end{pmatrix} \otimes \pi_1 e^{K} \begin{pmatrix} v \\ m \end{pmatrix}.
\end{equation}
We recall that the projection operators $\pi_1, \pi_2$ on the product space $\mathbb{R}^3 \times \mathbb{R}^3$ are defined in \eqref{def:proiezioni}.
Furthermore, decomposing $e^K = \operatorname{Id} + (e^K - \operatorname{Id})$, we introduce
\begin{align}\label{Def:N_1Nonlinear}
 N_1(v,m) 
 &:= m \otimes \pi_2 (e^{K} - \operatorname{Id}) \begin{pmatrix} v \\ m \end{pmatrix} 
 - v \otimes \pi_1 (e^{K} - \operatorname{Id}) \begin{pmatrix} v \\ m \end{pmatrix} \nonumber \\
 &\quad + \pi_2 (e^{K} - \operatorname{Id}) \begin{pmatrix} v \\ m \end{pmatrix} \otimes m 
 - \pi_1 (e^{K} - \operatorname{Id}) \begin{pmatrix} v \\ m \end{pmatrix} \otimes v \nonumber \\ 
 &\quad + \pi_2 (e^{K} - \operatorname{Id}) \begin{pmatrix} v \\ m \end{pmatrix} \otimes \pi_2 (e^{K} - \operatorname{Id}) \begin{pmatrix} v \\ m \end{pmatrix} \nonumber \\
 &\quad - \pi_1 (e^{K} - \operatorname{Id}) \begin{pmatrix} v \\ m \end{pmatrix} \otimes \pi_1 (e^{K} - \operatorname{Id}) \begin{pmatrix} v \\ m \end{pmatrix},
\end{align}
and
\begin{align}\label{Def:N_2Nonlinear}
 N_2(v,m) 
 &:= v \otimes \pi_2 (e^{K} - \operatorname{Id}) \begin{pmatrix} v \\ m \end{pmatrix} 
 - \pi_2 (e^{K} - \operatorname{Id}) \begin{pmatrix} v \\ m \end{pmatrix} \otimes v \nonumber \\
 &\quad + \pi_1 (e^{K} - \operatorname{Id}) \begin{pmatrix} v \\ m \end{pmatrix} \otimes \pi_2 (e^{K} - \operatorname{Id}) \begin{pmatrix} v \\ m \end{pmatrix} \nonumber \\
 &\quad - \pi_2 (e^{K} - \operatorname{Id}) \begin{pmatrix} v \\ m \end{pmatrix} \otimes \pi_1 (e^{K} - \operatorname{Id}) \begin{pmatrix} v \\ m \end{pmatrix}.
\end{align}
By construction, we have the relations
\begin{equation}\label{NtildeNeq1}
\tilde N_1(v,m) = m \otimes m - v \otimes v + N_1(v,m),
\end{equation}
\begin{equation}\label{NtildeNeq2}
\tilde N_2(v,m) = V(v,m) \otimes m - m \otimes V(v,m) + N_2(v,m),
\end{equation}
where the distorted transport velocity field is given by
\begin{equation}\label{DEf:AdjustedTransport}
V(v,m) :=  \pi_1 e^{K} \begin{pmatrix} v \\ m \end{pmatrix} = v + \pi_1 (e^{K} - \operatorname{Id}) \begin{pmatrix} v \\ m \end{pmatrix}.
\end{equation}
Another identity that will be crucial in the forthcoming analysis is
\begin{equation}\label{NtildeNeq3}
e^{-K} \begin{pmatrix} \mathbb{P} \operatorname{div} \tilde N_1(v,m) \\ \operatorname{div} \tilde N_2(v,m) \end{pmatrix} = 
\begin{pmatrix} \mathbb{P} \operatorname{div} \tilde N_1(v,m) \\ \operatorname{div} \tilde N_2(v,m) \end{pmatrix} + \Gamma(v,m),
\end{equation}
where we have defined
\begin{equation}\label{DEf:Gamma}
\Gamma(v,m) := \sum_{\ell \geq 1} \frac{1}{\ell !} (-K)^\ell \begin{pmatrix} \mathbb{P} \operatorname{div} \tilde N_1(v,m) \\ \operatorname{div} \tilde N_2(v,m) \end{pmatrix}.
\end{equation}

\subsection{Algebraic preliminaries}
We now collect the main algebraic properties of the operators introduced so far, enabling us to work effectively with the reformulation of the MHD equations below. In particular, we show that these operators preserve both the zero-average condition and the divergence-free constraint satisfied by the underlying vector fields.

Building upon the setting established in Section \ref{Sec;Hom}, we tacitly assume that all operators act on zero-average, divergence-free vector fields in $H^2(\mathbb{T}^3)$. 

\begin{prop}\label{prop:algebraic_properties}
Let $F, G : \mathbb{T}^3 \to \mathbb{R}^3$ be zero-average, smooth divergence-free vector fields. Then, the following properties hold.
\begin{enumerate}
    \item \label{AlgLemma1} \textbf{Basic operators:} 
    $$ \dive \mathrm{q} [F] = \dive \mathrm{p} [F] = 0 \quad \text{and} \quad \int_{\mathbb{T}^3} \mathrm{q} [F] = \int_{\mathbb{T}^3} \mathrm{p} [F] = 0. $$
    
    \item \label{AlgLemma2} \textbf{Action of $K$:} 
    $$ \dive \pi_1 K \binom{F}{G} = \dive \pi_2 K \binom{F}{G} = 0 \quad \text{and} \quad \int_{\mathbb{T}^3} \pi_1 K \binom{F}{G} = \int_{\mathbb{T}^3} \pi_2 K \binom{F}{G} = 0. $$
    
    \item \label{AlgLemma2,5} \textbf{Exponential operators:} the vector fields $\pi_i (e^{K} - \mathrm{Id}) \binom{F}{G}$ and $\pi_i e^{K} \binom{F}{G}$ are zero-average and divergence-free for $i \in \{1, 2\}$. In particular, 
    \begin{equation}\label{fnejskkhjfg6458293u2}
    \dive \pi_1 (e^{K} - \mathrm{Id}) \binom{F}{G} = \dive \pi_2 (e^{K} - \mathrm{Id}) \binom{F}{G} = 0.
    \end{equation}
    
    \item \label{AlgLemma4} \textbf{Auxiliary linear operators:} the operators $L_i, M_i, T_i$ for $i \in \{1,2\}$ preserve both constraints:
    $$ \dive L_i F = \dive M_i F = \dive T_i F = 0 \quad \text{and} \quad \int_{\mathbb{T}^3} L_i F = \int_{\mathbb{T}^3} M_i F = \int_{\mathbb{T}^3} T_i F = 0. $$
    
    \item \label{AlgLemma5} \textbf{Nonlinear and transport operators:} the fields $\pi_i \mathcal{A} \binom{F}{G}$ for $i \in \{1,2\}$, $V(F,G)$, and $\pi_i \Gamma(F,G)$ for $i \in \{1,2\}$ are zero-average and divergence-free. In particular,
    \begin{equation}\label{fdsjlkedgj9856735219}
    \dive V(F,G) = 0.
    \end{equation}
    
    \item \label{AlgLemma7} \textbf{Double divergence property:} the nonlinear operator $\tilde N_2$ satisfies
    $$ \dive \dive \tilde N_2(F,G) = 0. $$
\end{enumerate}
\end{prop}

\begin{proof}
All assertions follow from direct computations by combining the definitions of the respective operators with the structural properties of the Leray projector $\mathbb{P} := \mathrm{Id} - \Delta^{-1} \nabla \otimes \nabla$.
\begin{itemize}
    \item Property \ref{AlgLemma1} is an immediate consequence of \eqref{DEf:P,Q}, recalling that $\mathbb{P}$ is the orthogonal projection onto the space of divergence-free, zero-average vector fields.
    \item Properties \ref{AlgLemma2} and \ref{AlgLemma4} follow directly from \ref{AlgLemma1} by exploiting the linear definitions \eqref{SumupB} and \eqref{Def:M1}--\eqref{Def:N2}.
    \item Property \ref{AlgLemma2,5} is obtained by expanding the operator exponential in power series, $e^{K} - \mathrm{Id} = \sum_{j \geq 1} \frac{1}{j!} K^{j}$, and applying \ref{AlgLemma2} term-by-term. The claim for $e^K$ follows analogously since $e^K = \mathrm{Id} + (e^K - \mathrm{Id})$ and the identity operator trivially preserves the constraints.
    \item Properties \ref{AlgLemma5} and \ref{AlgLemma7} are deduced similarly from definitions \eqref{Def:MathcalA}, \eqref{DEf:AdjustedTransport}, \eqref{DEf:Gamma}, and \eqref{NtildeNeq2}, combined with the relations established in the previous steps.
\end{itemize}
\end{proof}

\subsection{The Lyapunov-Perron formula for the stable manifold}
In the following theorem, we work with classical solutions, as this is sufficient for the purpose of this paper. However, the statement can be adapted to more general notions of solutions with obvious modifications.

\begin{thm}\label{fdanjknjfdsjn756734}
Let $B$ be a stationary smooth solution of the Euler equations. Let $v, m : [0, +\infty) \times \mathbb{T}^3 \to \mathbb{R}^3$ be zero-average divergence-free vector fields. If $(v, m)$ is a smooth solution of the system
\begin{equation}\label{nfdjkshjbffbbfbf}
\begin{cases}
\partial_t v = \Delta v + \pi_1 \mathcal{A} \begin{pmatrix} v \\ m \end{pmatrix} + \mathbb{P} \operatorname{div} \big(m \otimes m - v \otimes v + N_1(v,m) \big) + \pi_1 \Gamma(v,m), \\
\partial_t m = \pi_2 \mathcal{A} \begin{pmatrix} v \\ m \end{pmatrix} + \operatorname{div} \big(V(v,m) \otimes m - m \otimes V(v,m) + N_2(v, m)\big) + \pi_2 \Gamma(v,m),
\end{cases}
\end{equation}
then
\[
\begin{pmatrix} u \\ b \end{pmatrix} := \begin{pmatrix} 0 \\ B \end{pmatrix} + e^{K} \begin{pmatrix} v \\ m \end{pmatrix}
\]
is a global smooth solution of the MHD equations
\begin{equation}\label{eq:mhdInLemma}
\begin{cases}
\partial_t u + \mathbb{P} \operatorname{div} (u \otimes u) - \Delta u = \mathbb{P} \operatorname{div} (b \otimes b),\\
\partial_t b + \operatorname{div} (b \otimes u) = \operatorname{div} (u \otimes b).
\end{cases}
\end{equation}
\end{thm}  

\begin{proof}
Let $\tilde b := b - B$. We notice that $(u,b)$ satisfies \eqref{eq:mhdInLemma} if and only if $(u,\tilde b)$ satisfies
\begin{equation}\label{eq:mhd2}
\partial_t \begin{pmatrix} u \\ \tilde b \end{pmatrix} = L \begin{pmatrix} u \\ \tilde b \end{pmatrix} + \begin{pmatrix} \mathbb{P} \operatorname{div} (\tilde b \otimes \tilde b) - \mathbb{P} \operatorname{div} (u \otimes u) \\ \operatorname{div} (u \otimes \tilde b) - \operatorname{div} (\tilde b \otimes u) \end{pmatrix},
\end{equation}
where $L$ is the operator defined in \eqref{SumupBOLD}. For the reader's convenience, we recall that
\[
L := \begin{pmatrix} \Delta & \mathrm{p} + \mathrm{q} \\ \mathrm{p} - \mathrm{q} & 0 \end{pmatrix},
\]
and
\[
\mathrm{p} [\cdot] := \mathbb{P} \operatorname{div} ([\cdot] \otimes B ) = \mathbb{P} \big( (B \cdot \nabla) [\cdot] \big), \qquad \mathrm{q} [\cdot] := \mathbb{P} \operatorname{div} (B \otimes [\cdot] ) = \mathbb{P} \big( ([\cdot] \cdot \nabla ) B \big),
\]
where the equalities hold since $\operatorname{div} B = 0$ and $\operatorname{div} [\cdot] = 0$. Consequently, since $\operatorname{div} u = \operatorname{div} B = 0$, we indeed have
\[
\mathrm{p}[u] - \mathrm{q}[u] = \operatorname{div} (u \otimes B ) - \operatorname{div} (B \otimes u ) = (B \cdot \nabla) u - (u \cdot \nabla ) B.
\]
We now introduce the invertible change of variables
\[
\begin{pmatrix} v \\ m \end{pmatrix} := e^{-K} \begin{pmatrix} u \\ \tilde b \end{pmatrix}.
\]
Invoking Theorem \ref{ConiugioTHM}, we see that $(u,\tilde b)$ solves \eqref{eq:mhd2} if and only if $(v,m)$ solves
\begin{equation}\label{eq:mhd3preq}
\partial_t \begin{pmatrix} v \\ m \end{pmatrix} = \begin{pmatrix} \Delta v \\ 0 \end{pmatrix} + \mathcal{A} \begin{pmatrix} v \\ m \end{pmatrix} + e^{-K} \begin{pmatrix} \mathbb{P} \operatorname{div} \tilde N_1(v,m) \\ \operatorname{div} \tilde N_2(v,m) \end{pmatrix},
\end{equation}
where $\mathcal{A}$ is defined in \eqref{Def:MathcalA}. Using identities \eqref{NtildeNeq1}, \eqref{NtildeNeq2}, and \eqref{NtildeNeq3}, this system can be rewritten as
\[
\partial_t \begin{pmatrix} v \\ m \end{pmatrix} = \begin{pmatrix} \Delta v \\ 0 \end{pmatrix} + \mathcal{A} \begin{pmatrix} v \\ m \end{pmatrix} + \begin{pmatrix} \mathbb{P} \operatorname{div} (m \otimes m - v \otimes v + N_1(v, m)) + \pi_1 \Gamma(v,m) \\ \operatorname{div} \big(V(v,m) \otimes m - m \otimes V(v,m) + N_2(v, m)\big) + \pi_2 \Gamma(v,m) \end{pmatrix},
\]
which completes the proof.
\end{proof}

We denote by $\Phi_{t,s}$ the flow associated with the velocity field $V(v, m)$, defined by
\begin{equation}\label{nFlowDefPreq}
\begin{cases}
\partial_s \Phi_{t,s} (x) = V(v, m)(s, \Phi_{t,s}(x)),\\
\Phi_{t,t} (x) = x.
\end{cases}
\end{equation}

The following is the Lyapunov-Perron formulation to identify the stable manifold associated to equation \eqref{nfdjkshjbffbbfbf}. Note that we have moved to Lagrangian coordinates in the second equation (referring to the flow $\Phi_{t,s}$), while retaining Eulerian coordinates in the first. As explained in Section \ref{sec:proofstrategy}, this mixed Eulerian/Lagrangian formulation is crucial to handle the derivative loss arising from the non-linearity in a manner compatible with the stable manifold method presented in this paper.   

\begin{thm}\label{FinalTHM1dfhgsjdhgfs}
Let $B$ be a stationary smooth solution of the Euler equations and let $y : \mathbb{T}^3 \to \mathbb{R}^3$ be a zero-average divergence-free smooth vector field. Let $v,m: [0, +\infty) \times \mathbb{T}^3 \to \mathbb{R}^3$ be a pair of zero-average divergence-free vector fields. If $(v, m)$ is a smooth solution of
\begin{equation}
\label{iterazione integralePrequel}
\begin{cases}
\displaystyle v(t,x) = e^{t\Delta}y(x) + \int_0^t e^{(t-s)\Delta} \mathbb{P} \big( (m \cdot\nabla) m - (v \cdot\nabla) v + \operatorname{div} N_{1}(v, m) \big) (s,x) \, \de s \\ 
\displaystyle \qquad \qquad \qquad + \int_0^t e^{(t-s)\Delta} \left( \pi_1 \mathcal{A} \begin{pmatrix} v \\ m \end{pmatrix} + \pi_1 \Gamma(v,m) \right) (s,x) \, \de s, \\
\\
\displaystyle m(t,x) = -\int_t^\infty \big((m \cdot \nabla) V(v, m) \big)(s,\Phi_{t,s}(x)) \, \de s \\ 
\displaystyle \qquad \qquad \qquad - \int_t^\infty \left( \pi_2 \mathcal{A} \begin{pmatrix} v \\ m \end{pmatrix} + \pi_2 \Gamma(v,m) + \operatorname{div} N_2(v,m) \right) (s,\Phi_{t,s}(x)) \, \de s,
\end{cases}
\end{equation}
then $(v,m)$ is a global solution of \eqref{nfdjkshjbffbbfbf}.
\end{thm}

\begin{proof}
Since $\operatorname{div} v = \operatorname{div} m = 0$, we recall from \eqref{OtimesVSnabla} that
\[
\operatorname{div} (m \otimes m - v \otimes v) = (m \cdot\nabla) m - (v \cdot\nabla) v.
\]
Keeping this in mind, we see that the first equation in \eqref{iterazione integralePrequel} is simply the Duhamel formulation of the first equation in \eqref{nfdjkshjbffbbfbf}. 
Similarly, since $\operatorname{div} V(v,m) = 0$ (see Proposition \ref{prop:algebraic_properties}) we have
\[
\dive(V(v,m) \otimes m - m \otimes V(v,m)) = (m \cdot \nabla)V(v,m) - (V(v,m) \cdot \nabla)m.
\]
Thus, to complete the proof, it suffices to show that
\begin{equation}\label{jdkslgndjsjkdsgn3}
\partial_t m + (V(v, m)\cdot\nabla) m = (m \cdot \nabla) V(v,m) + \pi_2 \mathcal{A} \begin{pmatrix} v \\ m \end{pmatrix} + \pi_2 \Gamma(v,m) + \operatorname{div} N_2(v,m).
\end{equation}
We recall that $\Phi_{t,s}$ denotes the flow from time $t$ to time $s$, which satisfies the semigroup property:
\[
\Phi_{t,s}(\Phi_{\tau,t}(y)) = \Phi_{\tau,s}(y), \quad y \in \mathbb{T}^3.
\]
Applying the change of variables $x = \Phi_{\tau,t}(y)$ in the second equation of \eqref{iterazione integralePrequel}, we obtain
\begin{align}\label{fdnjskgdksdg74736}
m(t,\Phi_{\tau,t}(y)) &= -\int_t^\infty \big((m \cdot \nabla) V(v ,m) \big)(s,\Phi_{\tau,s}(y)) \, \de s \nonumber \\ 
&\quad - \int_t^\infty \left( \pi_2 \mathcal{A} \begin{pmatrix} v \\ m \end{pmatrix} + \pi_2 \Gamma(v,m) + \operatorname{div} N_2(v,m) \right) (s,\Phi_{\tau,s}(y)) \, \de s.
\end{align}
From \eqref{nFlowDefPreq}, we deduce that
\begin{equation}\label{fdjuhsiuhfghdskgjdhgsfdhFischer}
\frac{\de }{\de  t} \big( m(t,\Phi_{\tau,t}(y)) \big) = \partial_t m + (V(v ,m) \cdot \nabla) m \Big|_{(t,x) = (t,\Phi_{\tau,t}(y))}. 
\end{equation}
Thus, differentiating both sides of \eqref{fdnjskgdksdg74736} with respect to $t$, utilizing \eqref{fdjuhsiuhfghdskgjdhgsfdhFischer} alongside the fundamental theorem of calculus, and taking the limit as $\tau \to t$, we arrive at \eqref{jdkslgndjsjkdsgn3}.
\end{proof}

\section{Preliminary estimates}
In this section we collect some preliminary estimates that will be repeatedly used in the rest of the paper.
%
We recall some classic estimates for the heat flow.
\begin{lem}\label{lemma calore}
Let $f : \T^3 \to \R$ with zero-average. For any $k \geq 0$ and $s>0$ one has
\begin{equation}\label{fdjskhggnjds2}
\|e^{s\Delta}f \|_{H^{k+1}} \lesssim s^{-1/2} \|f\|_{H^k},
\end{equation}
and
\begin{equation}\label{fdjskhggnjds1}
\|e^{s\Delta}f \|_{H^k} \lesssim e^{-s}\|f\|_{H^k}.
\end{equation}
 \end{lem}

The following product estimates can be found, for instance, in \cite[Proposition 2]{BenyOh}:

\begin{lem}\label{lem:GubinelliBound}
Let $s >0$ and $f, g : \T^3 \to \R$. Then
\begin{equation}\label{GubinelliBound}
\| fg \|_{H^{-s}} \lesssim \| f \|_{H^{-s}} \|g\|_{W^{s, \infty}}.
\end{equation}
\end{lem}

%
The following inequality is known as {\em fractional Leibniz rule}, see \cite{Grafakos}.
\begin{lem}\label{lem:fractional}
Let $s>0$, $1<r<\infty$ and $1<p_1,q_1,p_2,q_2\leq \infty$ such that
$$
\frac1r=\frac{1}{p_1}+\frac{1}{q_1}=\frac{1}{p_2}+\frac{1}{q_2}.
$$
Then, there exists a constant $C>0$ (depending on $s,r,p_1,q_1,p_2,q_2$) such that for any $f,g\in C^{\infty}(\T^3)$
\begin{equation}
\|\langle D\rangle^s(fg)\|_{L^r}\leq C\left(\|f\|_{L^{p_1}}\|\langle D\rangle^s g\|_{L^{q_1}}+\|\langle D\rangle^sf\|_{L^{p_2}}\| g\|_{L^{q_2}}\right).
\end{equation}
\end{lem}

We recall the Higher order chain rule in Sobolev spaces and we deduce an associated bound that will be useful in 
the rest of the paper.

\begin{lem}[Faà di Bruno's formula]\label{LemmaFaaDiBruno}
Let $g:\T^d\to\T^m$ and $F:\T^m\to \R$ be smooth functions. For each multi-index $\alpha \in \N^d$, we have
\begin{equation}
\partial^\alpha(F\circ g)=\sum_{\mu,\nu}C_{\mu,\nu}\, (\partial^\mu F) \circ g\prod_{\overset{1\leq |\beta|\leq |\alpha|}{1\leq j\leq m}}(\partial^\beta g_j)^{\nu_{\beta, j}},
\end{equation}
where the constants $C_{\mu,\nu}$ are non-negative integers, and the sum is taken over those $\mu$ and $\nu$ such that $1\leq |\mu|\leq |\alpha|$, $\nu_{\beta, j}\in \N$,
$$
\sum_{1\leq |\beta|\leq |\alpha|}\nu_{\beta, j}=\mu_j, \mbox{ for }1\leq j\leq m,\quad \mbox{and}\quad \sum_{\overset{1\leq |\beta|\leq |\alpha|}{1\leq j\leq m}}\beta\nu_{\beta, j}=\alpha .
$$
\end{lem}

\begin{lem}\label{LemProd}
Let $f_1, \ldots, f_n : \T^3 \to \R$ be smooth functions and $\alpha_1, \ldots, \alpha_n\in\N^3$ multi-indices with $|\alpha_i| \geq 1$ for all $i= 1, \ldots, n$. Let $k = \displaystyle\sum_{i=1}^n |\alpha_i|$.
Then, for all $\e >0$ we have that
\begin{equation}
\left\| \prod_{i=1}^n \partial^{\alpha_i} f_i \right\|_{L^2} \lesssim 
\sum_{j=1}^{n} \| \nabla f_j \|_{H^{k - n + \e}} \prod_{i \neq j} \| \nabla f_i \|_{L^{\infty}}.
\end{equation}
\end{lem}

\begin{proof}
For each function $f_i$, the Littlewood-Paley decomposition involves summing over dyadic frequencies of the form $2^\ell$ with $\ell \in \N$. Rigorously, one could denote the frequency $2^\ell$ associated with $f_i$ as $N_{i,\ell} = 2^\ell$. However, to lighten the notation, we will simply write $N_i \in 2^{\N}$ to denote the generic dyadic frequency associated with $f_i$, and write $\sum_{N_i}$ to mean the sum over all such dyadic scales. Expanding the product point-wise yields a multiple sum over all possible dyadic frequencies:
\begin{equation*}
\prod_{i=1}^n \partial^{\alpha_i} f_i = \sum_{N_1, \ldots, N_n \in 2^{\N}} \prod_{i=1}^n P_{N_i} \partial^{\alpha_i} f_i.
\end{equation*}

To estimate the $L^2$ norm of this sum, we divide the frequency space into $n$ symmetric regions based on which function has the highest frequency. In the $j$-th region, the frequency $N_j$ associated with $f_j$ is dominant, meaning $N_i \leq N_j$ for all $i \neq j$. 

Recalling that $|\alpha_i| \geq 1$ for all $i = 1, \ldots, n$, we extract one derivative from each term to form gradients, leaving the remaining $k-n$ derivatives on the dominant frequency term. Applying the triangle inequality, we obtain:
\begin{align*}
\left\| \prod_{i=1}^n \partial^{\alpha_i}  f_i \right\|_{L^2} 
& \leq \sum_{j=1}^n \sum_{N_j} \sum_{\substack{N_i \leq N_j \\ i \neq j}}  
\left\| P_{N_j} \partial^{\alpha_j} f_j \prod_{i \neq j} P_{N_i} \partial^{\alpha_i} f_{i} \right\|_{L^2}
\\ 
& \lesssim \sum_{j=1}^n \sum_{N_j} \sum_{\substack{N_i \leq N_j \\ i \neq j}}  
N_j^{k-n}  \| P_{N_j} \nabla  f_j \|_{L^2} \prod_{i \neq j}  \| \nabla P_{N_i} f_{i} \|_{L^\infty}
\\
& \lesssim \sum_{j=1}^n \prod_{i \neq j}  \| \nabla  f_{i} \|_{L^\infty} \sum_{N_j} \sum_{\substack{N_i \leq N_j \\ i \neq j}}  
N_j^{k-n}  \| P_{N_j} \nabla  f_j \|_{L^2} 
\\ 
& \lesssim \sum_{j=1}^n \prod_{i \neq j}  \| \nabla f_{i} \|_{L^\infty} \sum_{N_j} N_j^{k-n + \e/2}  \| P_{N_j} \nabla  f_j \|_{L^2} 
\\ 
& \lesssim \sum_{j=1}^n \prod_{i \neq j}  \| \nabla f_{i} \|_{L^\infty} \left( \sum_{N_j} N_j^{2(k-n + \e)}  \| P_{N_j} \nabla  f_j \|^2_{L^2} \right)^{1/2} \left( \sum_{N_j} N_j^{-\e} \right)^{1/2}
\\ 
& \lesssim \sum_{j=1}^n \prod_{i \neq j}  \| \nabla f_{i} \|_{L^\infty} \|\nabla f_{j}\|_{H^{k-n + \e}}.
\end{align*}

In the second inequality, we used the distribution of the $k$ derivatives and Hölder's inequality.
In the third inequality, we bounded each $\| \nabla P_{N_i} f_i \|_{L^\infty}$ by the full norm $\| \nabla f_i \|_{L^\infty}$, using the uniform boundedness of the Littlewood-Paley projectors on $L^\infty$.
In the fourth inequality, we summed over the $n-1$ non-dominant dyadic frequencies $N_i \leq N_j$. Since the frequencies are dyadic, this summation produces at most an $O((\log N_j)^{n-1})$ factor, which is absorbed by the $N_j^{\e/2}$ term for any $\e > 0$.
In the fifth inequality, we applied the Cauchy-Schwarz inequality with respect to the dominant frequency scale $N_j$. Since $N_j \in 2^{\N}$, the geometric series $\sum N_j^{-\e}$ converges, yielding a uniform constant. The final implicit constant depends on $\e$, although we do not stress this to simplify the notations.
\end{proof}
\begin{rem}
The loss of $\e$-derivatives in the statement is removable by a finer argument, but this is not relevant for our purposes.
\end{rem}

\begin{lem}\label{CorollaryFaaDiBruno}
Let $k \geq 4$, $k \in \N$. Let $F:\T^3\to \R$ be a smooth function and $\Phi : \T^3\to\T^3$ an invertible smooth map. 
Suppose that 
\begin{equation}\label{JacobianBounded}
\| \det \nabla \Phi \|_{L^{\infty}} \sim 1
\end{equation} 
where $\nabla \Phi$ is the Jacobian matrix of the map $\Phi$. Then, 
for each multi-index $\alpha \in \N^3$, we have
\begin{equation}\label{fdshaukghudsjkg}
\| F\circ \Phi \|_{H^{k}} \lesssim \| F \|_{H^k}  (1+ \| \nabla \Phi \|_{H^{k-1}}^k) 
\end{equation}
\end{lem}
\begin{proof}
First of all, by \eqref{JacobianBounded} and the change of variables $y = \Phi(x)$
$$
\| F\circ \Phi \|_{L^2} \lesssim \|F\|_{L^2}.
$$
Hence, in order to prove \eqref{fdshaukghudsjkg}, it suffices to prove
\begin{equation}\label{fjdhushjgfff1222}
\| \partial^{\alpha} (F\circ \Phi ) \|_{L^{2}} \lesssim \| F \|_{H^k}  (1+ \| \nabla \Phi \|_{H^{k-1}}^k) 
\end{equation}
for all multi-index $\alpha$ with $|\alpha| = k$. \\
\\
By the Faà di Bruno formula (Lemma \ref{LemmaFaaDiBruno}) we have 
\begin{equation}\label{fdhsudjkghnsjadkjhgksdejgn}
\partial^\alpha(F\circ \Phi) = \sum_{\mu,\nu}C_{\mu,\nu}\, ( \partial^\mu F) \circ \Phi \prod_{\overset{1\leq |\beta|\leq k}{1\leq j\leq 3}}(\partial^\beta \Phi_j)^{\nu_{\beta, j}},
\end{equation}
where the constants $C_{\mu,\nu}$ are non-negative integers, and the sum runs over $\mu \in \N^3$ with $1\le|\mu|\le k$ and $\nu_{\beta, j}\in \N^*$ such that
\[
\sum_{1\le|\beta|\le k}\nu_{\beta,j} = \mu_j \ \ (1\le j\le 3), 
\qquad
\sum_{\substack{1\le|\beta|\le k\\ 1\le j\le 3}} \beta\,\nu_{\beta,j} = \alpha.
\]
Summing the components of these two identities gives
\begin{equation}\label{recallthismfdkslejdgfkn}
k := |\alpha| = \sum_{i = 1}^{3} \alpha_{i} =  \sum_{i = 1}^{3} \sum_{\overset{1\leq |\beta|\leq k}{1\leq j\leq 3}}\beta_{i}\nu_{\beta, j} = \sum_{\overset{1\leq |\beta|\leq k}{1\leq j\leq 3}} \sum_{i = 1}^{3} \beta_{i}\nu_{\beta, j} = \sum_{\overset{1\leq |\beta|\leq k}{1\leq j\leq 3}} |\beta| \nu_{\beta, j},
\end{equation}
and
\begin{equation}\label{fdjnsakjfhgn}
|\mu|= \sum_{j=1}^{3} \mu_j = \sum_{\overset{1\leq |\beta|\leq k}{1\leq j\leq 3}} \nu_{\beta, j},
\end{equation}
which also implies that
\begin{equation}\label{MainCounting}
k - |\mu| =\sum_{\overset{1\leq |\beta|\leq k}{1\leq j\leq 3}}  (|\beta| - 1)\nu_{\beta, j}.\tag{$\star$}
\end{equation}

We say that a multi-index $\beta$ in the sum is {\it active} (for a given term of \eqref{fdhsudjkghnsjadkjhgksdejgn}) if there exists a $j \in \{1,2,3\}$ such that $\nu_{\beta, j} \geq 1$; {\it non-active} multi-indices do not contribute to the corresponding product. In order to prove \eqref{fjdhushjgfff1222}, we split the sum \eqref{fdhsudjkghnsjadkjhgksdejgn} into five families of terms, not mutually exclusive, whose union covers every term.
We will denote by $\widetilde{\sum}$ the sum restricted to the case under consideration, and we will repeatedly use the (not-sharp) Sobolev embedding $H^3(\T^3)\hookrightarrow W^{1,\infty}(\T^3)$ without further mentioning it.\\

\smallskip
\noindent\textbf{Case 1: $1\le|\mu|\le k-2$ and $|\beta|\le k-2$ for every active $\beta$.}
By \eqref{fdjnsakjfhgn} we can bound every factor in $L^\infty$, using that $|\T^3|<\infty$, by Sobolev embeddings we have
\begin{align}
\left\|  \widetilde{\sum}_{\mu,\nu}C_{\mu,\nu}\, ( \partial^\mu F) \circ \Phi \prod_{\overset{1\leq |\beta|\leq k-2}{1\leq j\leq 3}}(\partial^\beta \Phi_j)^{\nu_{\beta, j}} \right\|_{L^{2}}
&\lesssim \sup_{1 \leq |\mu| \leq k-2} \| ( \partial^\mu F) \circ \Phi \|_{L^{\infty}} (1+\| \nabla \Phi \|^{|\mu|}_{W^{k -3, \infty}})\nonumber\\
&\lesssim \sup_{1 \leq |\mu| \leq k-2} \| \partial^\mu F\|_{L^{\infty}} (1+ \| \nabla \Phi \|^{k-2}_{H^{k-1}} ) \nonumber\\
&\lesssim \| F\|_{W^{k-2, \infty}} (1+ \| \nabla\Phi \|^{k-2}_{H^{k-1}}) \nonumber\\
&\lesssim \| F \|_{H^{k}} (1+  \| \nabla\Phi \|^{k-2}_{H^{k-1}}).\label{Crude1}
\end{align} 
Note that in the last estimate we use $k\geq 4$ for the Sobolev embedding $H^k(\T^3)\hookrightarrow W^{k-2, \infty}(\T^3)$.

\smallskip
\noindent\textbf{Case 2: $|\mu|=k$, with active $\beta$.}
By \eqref{MainCounting}, $|\beta|=1$ for every {\it active} $\beta$, so \eqref{recallthismfdkslejdgfkn} gives $k$ factors, each an entry of $\nabla\Phi$. Thus, we have that 
\begin{align}
\left\|  \widetilde{\sum}_{\mu,\nu}C_{\mu,\nu}\, ( \partial^\mu F) \circ \Phi \prod_{\overset{1\leq |\beta|\leq |\alpha|}{1\leq j\leq 3}}(\partial^\beta \Phi_j)^{\nu_{\beta, j}} \right\|_{L^{2}} 
&\lesssim \sup_{|\mu| =k} \| ( \partial^\mu F) \circ \Phi \|_{L^{2}} \| \nabla \Phi \|^{|\mu|}_{L^{\infty}} \nonumber\\
&\lesssim \sup_{|\mu| =k} \|  \partial^\mu F  \|_{L^{2}} \|\nabla \Phi \|^{k}_{H^{3}} \nonumber\\
&\lesssim \| F  \|_{H^k} \|\nabla \Phi \|^{k}_{H^{3}}. \label{Crude2}
\end{align} 

\smallskip
\noindent\textbf{Case 3: $|\mu|=k-1$, with active $\beta$.}
By \eqref{MainCounting}, exactly one {\it active} pair $(\tilde\beta,\tilde j)$ has $|\tilde\beta|=2$, with $\nu_{\tilde\beta,\tilde j}=1$ and $\nu_{\tilde\beta,j}=0$ for $j\ne\tilde j$; every other {\it active} $\beta$ satisfies $|\beta|=1$. Thus, we obtain
\begin{align}
\left\|  \widetilde{\sum}_{\mu,\nu}C_{\mu,\nu}\, ( \partial^\mu F) \circ \Phi \prod_{\overset{1\leq |\beta|\leq |\alpha|}{1\leq j\leq 3}}(\partial^\beta \Phi_j)^{\nu_{\beta, j}} \right\|_{L^{2}} 
& \lesssim \sup_{|\mu| =k-1} \| ( \partial^\mu F) \circ \Phi \|_{L^{2}} \| \nabla \Phi \|_{W^{1, \infty}} \| \nabla \Phi \|^{|\mu|-1}_{L^{\infty}}  \nonumber \\ 
& \lesssim\sup_{ |\mu| = k-1} \|  \partial^\mu F \|_{L^{2}} \| \nabla \Phi \|^{k-1}_{W^{1, \infty}} \nonumber\\
&\lesssim \|  F \|_{H^k} \| \nabla \Phi \|^{k-1}_{H^3}.\label{Crude3}  
\end{align}

\smallskip
\noindent\textbf{Case 4: there is an active $\tilde\beta$ with $|\tilde\beta|=k$.}
By \eqref{MainCounting}, $k-|\mu|\ge k-1$, so we deduce that $|\mu|=1$; by \eqref{fdjnsakjfhgn}, $\nu_{\tilde\beta,\tilde j}=1$ for a single $\tilde j$, and no other $\beta$ is {\it active}. Using this information we obtain 
\begin{align}
\left\|  \widetilde{\sum}_{\mu,\nu}C_{\mu,\nu}\, ( \partial^\mu F) \circ \Phi \prod_{\overset{1\leq |\beta|\leq |\alpha|}{1\leq j\leq 3}}(\partial^\beta \Phi_j)^{\nu_{\beta, j}} \right\|_{L^{2}} 
&\lesssim \sup_{|\mu| =1} \| ( \partial^\mu F) \circ \Phi \|_{L^{\infty}} \| \nabla \Phi \|_{H^{k-1}} \nonumber \\
&\lesssim \sup_{ |\mu| = 1} \|  \partial^\mu F \|_{L^{\infty}} \| \nabla \Phi \|_{H^{k-1}} \nonumber\\
&\lesssim \|F \|_{H^3} \| \nabla \Phi \|_{H^{k-1}}.  \label{Crude4}
\end{align} 

\smallskip
\noindent\textbf{Case 5: there is an active $\tilde\beta$ with $|\tilde\beta|=k-1$.}
Let $M=\sum_j\nu_{\tilde\beta,j}\ge1$ be its total multiplicity. From \eqref{MainCounting} we deduce that 
$$
M(k-2)\le k-|\mu|\le k-1.
$$
Now, note that $M\ge2$ would force $k\le3$, but this is impossible since we are assuming that $k\ge4$. So $M=1$. By \eqref{recallthismfdkslejdgfkn}, the order budget left after $\tilde\beta$ is $k-(k-1)=1$, which must be filled by \emph{exactly one} further active pair $(\beta^*,j^*)$ with $|\beta^*|=1$. In particular, from \eqref{fdjnsakjfhgn} we deduce $|\mu|=2$. Hence
\begin{align}
\left\|  \widetilde{\sum}_{\mu,\nu}C_{\mu,\nu}\, ( \partial^\mu F) \circ \Phi \prod_{\overset{1\leq |\beta|\leq |\alpha|}{1\leq j\leq 3}}(\partial^\beta \Phi_j)^{\nu_{\beta, j}} \right\|_{L^{2}} 
&\lesssim \sup_{|\mu| = 2} \| ( \partial^\mu F) \circ \Phi \|_{L^{\infty}} \| \nabla \Phi \|_{H^{k-2}} \nonumber\\ 
&\lesssim \sup_{ |\mu| = 2} \|  \partial^\mu F \|_{L^{\infty}} \| \nabla \Phi \|_{H^{k-2}} \nonumber\\
&\lesssim \|  F \|_{H^4} \| \nabla \Phi \|_{H^{k-2}}. \label{Crude5}  
\end{align} 
Combining the five estimates above \eqref{Crude1}--\eqref{Crude5} we obtain \eqref{fdshaukghudsjkg}, and the Lemma follows.
\end{proof}

\subsection{Linear estimates}

The estimates proved in this section are not optimal, but they are sufficient for our purposes. Some of them can be easily improved, which would be necessary if one aims to extend the results of this paper to less regular solutions.

We will assume, without further mention, that we are always working with $H^2$ vector fields (which is sufficient for our setting), ensuring that the operators are well-defined. The statements may be extended to a less regular scenario with due modifications. We will also assume in the statements that we work with zero-average, divergence-free vector fields. This ensures, in particular, that $\Delta^{-1}$ is always well-defined, see Remark \ref{RemWellDef}.

\begin{lem}\label{Lemm1Linear}
Let $F:\mathbb{T}^3\to \mathbb{R}^3$ be a zero-average, divergence-free vector field. For $\ell > 3/2$, we have  
\begin{align}
\|\mathrm{p} [F] \|_{H^{\ell}} &\lesssim \|B\|_{H^{\ell}} \|F\|_{H^{\ell+1}}, \label{pBound} \\
\|\mathrm{q} [F] \|_{H^{\ell}} &\lesssim \|B\|_{H^{\ell +1}} \|F\|_{H^{\ell}}. \label{qBound}
\end{align}
For fractional norms and $H^1$, the following bounds hold:
\begin{align}
\|\mathrm{p} [F] \|_{H^1} &\lesssim \left( \| B \|_{H^{\frac{3}{2}}} + \| B \|_{L^{\infty}} \right) \|F\|_{H^{2}}, \label{pBoundFractionalH1} \\
\|\mathrm{q} [F] \|_{H^1} &\lesssim \left( \| B \|_{H^{\frac{5}{2}}} + \| B \|_{W^{1,\infty}} \right) \|F\|_{H^1}, \label{qBoundFractionalH1} \\
\|\mathrm{p} [F] \|_{L^2} &\lesssim \|B\|_{L^{\infty}} \|F\|_{H^{1}}, \label{pBoundFractional} \\
\|\mathrm{q} [F] \|_{L^2} &\lesssim \|B\|_{W^{1, \infty}} \|F\|_{L^2}. \label{qBoundFractional}
\end{align}
Finally, for negative Sobolev spaces with $\ell > 0$:
\begin{align}
\|\mathrm{p} [F] \|_{H^{-\ell}} &\lesssim \|B\|_{W^{\ell, \infty}} \|F\|_{H^{-\ell +1}}, \label{pBoundFractionalNegative} \\
\|\mathrm{q} [F] \|_{H^{-\ell}} &\lesssim \|B\|_{W^{\ell +1, \infty}} \|F\|_{H^{-\ell}}. \label{qBoundFractionalNegative}
\end{align}
\end{lem}

\begin{proof}
Recalling the definitions
\begin{align}
\mathrm{p} [\cdot] &= \mathbb{P} \big( (B \cdot \nabla) [\cdot] \big), \nonumber \\
\mathrm{q} [\cdot] &= \mathbb{P} \big( ([\cdot] \cdot \nabla ) B \big), \label{dshuikhudfgs7}
\end{align}
and the fact that $\mathbb{P}$ is a Fourier multiplier with a bounded symbol, estimates \eqref{pBound} and \eqref{qBound} follow from the algebra property of $H^{\ell}$. Estimates \eqref{pBoundFractional} and \eqref{qBoundFractional} follow from Hölder's inequality, while \eqref{pBoundFractionalNegative} and \eqref{qBoundFractionalNegative} follow from Lemma \ref{lem:GubinelliBound}.

To prove \eqref{pBoundFractionalH1}, we use the fractional chain rule (Lemma \ref{lem:fractional}) and the Sobolev embeddings $H^{1} \hookrightarrow L^{6}$, $H^{1/2} \hookrightarrow L^{3}$ to get
\begin{align*}
\|\mathrm{p} [F] \|_{H^1} 
&= \| \langle D \rangle \big( (B \cdot \nabla) F \big) \|_{L^2} \\
&\leq \|\langle D \rangle B \|_{L^{3}} \| \nabla F\|_{L^6} + \| B \|_{L^{\infty}} \| \langle D \rangle \nabla F\|_{L^2} \\
&\leq \left( \| B \|_{H^{\frac{3}{2}}} + \| B \|_{L^{\infty}} \right) \| F\|_{H^2}. 
\end{align*}
Lastly, for the proof of \eqref{qBoundFractionalH1}, we proceed analogously to obtain
\begin{align*}
\|\mathrm{q} [F] \|_{H^1} 
&= \| \langle D \rangle \big( (F \cdot \nabla) B \big) \|_{L^2} \\
&\leq \|\langle D \rangle F \|_{L^{2}} \| \nabla B\|_{L^\infty} + \| F \|_{L^{6}} \| \langle D \rangle \nabla B\|_{L^{3}} \\
&\leq \left( \| B \|_{H^{\frac{5}{2}}} + \| B \|_{W^{1,\infty}} \right) \| F\|_{H^1}.
\end{align*}
This concludes the proof.
\end{proof}

\begin{lem}\label{Lemm2Linear}
Let $F, G :\mathbb{T}^3\to \mathbb{R}^3$ be zero-average, divergence-free smooth vector fields. Then, for $\ell > 7/2$ we have
\begin{align}
\left\|\pi_1 K \begin{pmatrix} F \\ G \end{pmatrix} \right\|_{H^{\ell}} &\lesssim \| B \|_{H^{\ell - 1}} \|G\|_{H^{\ell - 1}}, \label{fmdksdsjnsgd1} \\
\left\|\pi_2 K \begin{pmatrix} F \\ G \end{pmatrix} \right\|_{H^{\ell}} &\lesssim \| B \|_{H^{\ell + 1}} \|F\|_{H^{\ell - 1}}, \quad (\mbox{valid for } \ell > 3/2). \label{fmdksdsjnsgd2}
\end{align}
Consequently, for $\ell > 7/2$ it holds
\begin{equation}\label{fdnjskbjnfhjskhjdbfgsjk}
\left\| K \begin{pmatrix} F \\ G \end{pmatrix} \right\|_{H^{\ell} \times H^{\ell}} \lesssim \| B \|_{H^{\ell + 1}} \|(F, G)\|_{H^{\ell - 1} \times H^{\ell - 1}}.
\end{equation}
For the $L^2$ bounds, we have that
\begin{align}
\left\|\pi_1 K \begin{pmatrix} F \\ G \end{pmatrix} \right\|_{L^2} &\lesssim \| B \|_{W^{3, \infty}} \|G\|_{H^{-1}}, \label{fmdksdsjnsgd3} \\
\left\|\pi_2 K \begin{pmatrix} F \\ G \end{pmatrix} \right\|_{L^2} &\lesssim \| B \|_{W^{1,\infty}} \|F\|_{H^{- 1}}, \label{fmdksdsjnsgd4}
\end{align}
which imply
\begin{equation}\label{fdnjskbjnfhjskhjdbfgsjkBis}
\left\| K \begin{pmatrix} F \\ G \end{pmatrix} \right\|_{L^2 \times L^2} \lesssim \| B \|_{W^{3, \infty}} \|(F, G)\|_{H^{-1} \times H^{-1}}.
\end{equation}
Similarly, for the $H^1$ bounds
\begin{align}
\left\|\pi_1 K \begin{pmatrix} F \\ G \end{pmatrix} \right\|_{H^1} &\lesssim \| B \|_{W^{2, \infty}} \|G\|_{L^2}, \label{fmdksdsjnsgd3H1} \\
\left\|\pi_2 K \begin{pmatrix} F \\ G \end{pmatrix} \right\|_{H^1} &\lesssim \left( \| B \|_{H^{\frac{5}{2}}} + \| B \|_{W^{1,\infty}} \right) \|F\|_{L^2}, \label{fmdksdsjnsgd4H1}
\end{align}
yielding
\begin{equation}\label{fdnjskbjnfhjskhjdbfgsjkBisH1}
\left\| K \begin{pmatrix} F \\ G \end{pmatrix} \right\|_{H^1 \times H^1} \lesssim \left( \| B \|_{H^{\frac{5}{2}}} + \| B \|_{W^{2,\infty}} \right) \|(F, G)\|_{L^2 \times L^2}.
\end{equation}
\end{lem}

\begin{proof}
Recalling that
\begin{equation*}
 K = 
\begin{pmatrix} 0 & - \Delta^{-1} \circ \mathrm{p} \\ \mathrm{p} \circ \Delta^{-1} & 0 \end{pmatrix} 
+
\begin{pmatrix} 0 & -\Delta^{-1} \circ \mathrm{q} \\ - \mathrm{q} \circ \Delta^{-1} & 0 \end{pmatrix},
\end{equation*}
the component-wise inequalities follow easily from Lemma \ref{Lemm1Linear}. The combined bounds \eqref{fdnjskbjnfhjskhjdbfgsjk}, \eqref{fdnjskbjnfhjskhjdbfgsjkBis}, and \eqref{fdnjskbjnfhjskhjdbfgsjkBisH1} are then obtained by directly summing the estimates.
\end{proof}

\begin{lem}\label{ExpLemma} 
Let $F, G :\mathbb{T}^3\to \mathbb{R}^3$ be zero-average, divergence-free vector fields. For all $t \in \mathbb{R}$ and $\ell > 7/2$ it holds
\begin{align}
\left\| e^{t K} \begin{pmatrix} F \\ G \end{pmatrix} \right\|_{H^{\ell} \times H^{\ell}} &\leq e^{|t| C \| B \|_{H^{\ell + 1}}} \|(F, G)\|_{H^{\ell} \times H^{\ell}}, \label{fdnjskbjnfhjskhjdbfgsjk2} \\
\left\| (e^{tK} - \operatorname{Id}) \begin{pmatrix} F \\ G \end{pmatrix} \right\|_{H^{\ell} \times H^{\ell}} &\leq \left(e^{|t| C \|B\|_{H^{\ell +1}}} -1 \right) \|(F, G)\|_{H^{\ell-1} \times H^{\ell-1}}, \label{fdnjskbjnfhjskhjdbfgsjk3}
\end{align}
where $C > 0$ depends only on $\ell$. Additionally, for the $L^2$ and $H^1$ norms we have
\begin{align}
\left\| e^{t K} \begin{pmatrix} F \\ G \end{pmatrix} \right\|_{L^2 \times L^2} &\leq e^{|t| C \| B \|_{ W^{3, \infty} } } \|(F, G)\|_{L^2 \times L^2}, \label{fdnjskbjnfhjskhjdbfgsjk2L2} \\
\left\| (e^{tK} - \operatorname{Id}) \begin{pmatrix} F \\ G \end{pmatrix} \right\|_{L^2 \times L^2} &\leq \left(e^{|t| C \|B\|_{W^{3, \infty} }} -1 \right) \|(F, G)\|_{H^{-1} \times H^{-1}}, \label{fdnjskbjnfhjskhjdbfgsjk3L2} \\
\left\| (e^{tK} - \operatorname{Id}) \begin{pmatrix} F \\ G \end{pmatrix} \right\|_{H^1 \times H^1} &\leq \left(e^{|t| C \left( \| B \|_{H^{\frac{5}{2}}} + \| B \|_{W^{2,\infty}} \right) } -1 \right) \|(F, G)\|_{L^2 \times L^2}, \label{fdnjskbjnfhjskhjdbfgsjk3H1}
\end{align}
where $C > 0$ is an absolute constant.
\end{lem}

\begin{proof}
Estimate \eqref{fdnjskbjnfhjskhjdbfgsjk2} follows immediately from \eqref{fdnjskbjnfhjskhjdbfgsjk}. To prove \eqref{fdnjskbjnfhjskhjdbfgsjk3}, we note that as a consequence of \eqref{fdnjskbjnfhjskhjdbfgsjk}, for all $j \geq 1$ we have
\begin{equation}\label{mgnrjekrndgfsjk1}
\left\| K^j \begin{pmatrix} F \\ G \end{pmatrix} \right\|_{H^{\ell} \times H^{\ell}} \lesssim \| B \|_{H^{\ell + 1}} \|K^{j-1}(F, G)\|_{H^{\ell-1} \times H^{\ell-1}}. 
\end{equation}
Iterating this estimate $j$ times, we arrive at
\begin{equation}\label{AncoraWeakerexp}
\left\| K^j \begin{pmatrix} F \\ G \end{pmatrix} \right\|_{H^{\ell} \times H^{\ell}} \leq C^j \| B \|^j_{H^{\ell + 1}} \|(F, G)\|_{H^{\ell-1} \times H^{\ell-1}}, \qquad \forall j \geq 1.
\end{equation}
Using the power series expansion
\begin{equation}\label{mndfjskdbg57483}
(e^{tK} - \operatorname{Id}) \begin{pmatrix} F \\ G \end{pmatrix} = \sum_{j \geq 1} \frac{t^j}{j!} K^j \begin{pmatrix} F \\ G \end{pmatrix},
\end{equation}
combined with \eqref{AncoraWeakerexp} and the triangle inequality, we deduce \eqref{fdnjskbjnfhjskhjdbfgsjk3}.

Estimate \eqref{fdnjskbjnfhjskhjdbfgsjk2L2} follows from \eqref{fdnjskbjnfhjskhjdbfgsjkBis}. For \eqref{fdnjskbjnfhjskhjdbfgsjk3L2}, iterating \eqref{fdnjskbjnfhjskhjdbfgsjkBis} $j$ times yields
\begin{equation}\label{AncoraWeakerexpL2}
\left\| K^j \begin{pmatrix} F \\ G \end{pmatrix} \right\|_{L^2 \times L^2} \leq C^j \| B \|^j_{W^{3, \infty}} \|(F, G)\|_{H^{-1} \times H^{-1}}, \qquad \forall j \geq 1,
\end{equation}
which, proceeding as before, gives \eqref{fdnjskbjnfhjskhjdbfgsjk3L2}.

Finally, for \eqref{fdnjskbjnfhjskhjdbfgsjk3H1}, iterating \eqref{fdnjskbjnfhjskhjdbfgsjkBisH1} $j$ times leads to
\begin{equation}\label{AncoraWeakerexpH1}
\left\| K^j \begin{pmatrix} F \\ G \end{pmatrix} \right\|_{L^2 \times L^2} \leq C^j \left( \| B \|_{H^{\frac{5}{2}}} + \| B \|_{W^{2,\infty}} \right)^j \|(F, G)\|_{L^2 \times L^2}, \qquad \forall j \geq 1.
\end{equation}
Using the series expansion again completes the proof.
\end{proof}

\begin{lem}\label{Lemm3Linear}
Let $F:\mathbb{T}^3\to \mathbb{R}^3$ be a zero-average, divergence-free vector field. Then:
\begin{align}
\| M_1 F \|_{H^{\ell}} &\lesssim \|B\|^2_{H^{\ell +1}} \|F\|_{H^{\ell}}, \quad \ell > 7/2, \\
\| M_2 F \|_{H^{\ell}} &\lesssim \|B\|^2_{H^{\ell}} \|F\|_{H^{\ell}}, \quad \ell > 5/2, 
\end{align}
as well as
\begin{align}
\| M_1 F \|_{L^2} &\lesssim \| B \|_{W^{2, \infty}}^2 \|F\|_{L^2}, \label{fmdjisonjfdnwj74} \\
\| M_2 F \|_{L^2} &\lesssim \| B \|_{W^{1, \infty}}^2 \|F\|_{L^2}. 
\end{align}
\end{lem}
\begin{proof}
Recalling definitions \ref{Def:M1} and \ref{Def:M2} of $M_1$ and $M_2$, the statement follows from Lemma \ref{Lemm1Linear} after straightforward computations.
\end{proof}

\begin{lem}\label{Lemm4Linear}
Let $F :\mathbb{T}^3\to \mathbb{R}^3$ be a zero-average, divergence-free vector field. Then, for $\ell > 7/2$
\begin{align}
\| L_1 F \|_{H^{\ell}} &\lesssim \|B\|^2_{H^{\ell+2}} \|F\|_{H^{\ell-1}}, \\
\| L_2 F \|_{H^{\ell}} &\lesssim \|B\|^2_{H^{\ell +1}} \|F\|_{H^{\ell -1}}, 
\end{align}
as well as
\begin{align}
\| L_1 F \|_{L^2} &\lesssim \|B\|^2_{W^{3, \infty}} \|F\|_{H^{-1}}, \\
\| L_2 F \|_{L^2} &\lesssim \|B\|^2_{W^{3, \infty}} \|F\|_{H^{-1}}. 
\end{align}
\end{lem}
\begin{proof}
Recalling definitions \ref{Def:L1} and \ref{Def:L2} of $L_1$ and $L_2$, the statement follows from Lemma \ref{Lemm1Linear} after straightforward computations. 
\end{proof}

\begin{lem}\label{Lemm5Linear}
Let $F :\mathbb{T}^3\to \mathbb{R}^3$ be a zero-average, divergence-free vector field. Then, for $\ell > 7/2$
\begin{align}
\| T_1 F \|_{H^{\ell}} &\lesssim \|B\|^2_{H^{\ell+2}} \|F\|_{H^{\ell-1}}, \\
\| T_2 F \|_{H^{\ell}} &\lesssim \|B\|^2_{H^{\ell +1 }} \|F\|_{H^{\ell-1}}, 
\end{align}
as well as
\begin{align}
\| T_1 F \|_{L^2} &\lesssim \|B\|^2_{W^{3, \infty}} \|F\|_{H^{-1}}, \\
\| T_2 F \|_{L^2} &\lesssim \|B\|^2_{W^{3, \infty}} \|F\|_{H^{-1}}. 
\end{align}
\end{lem}
\begin{proof}
Recalling definitions \ref{Def:N1} and \ref{Def:N2} of $T_1$ and $T_2$, the statement follows from Lemma \ref{Lemm1Linear}. 
\end{proof}

\begin{lem}\label{LemmaBoundForA}
Let $\ell > 7/2$. Let $F, G :\mathbb{T}^3\to \mathbb{R}^3$ be zero-average, divergence-free vector fields and suppose $\|B\|_{H^{\ell+2}} \lesssim 1$. Then 
\begin{equation}\label{mathcalAbound}
\left\| \mathcal{A} \begin{pmatrix} F \\ G \end{pmatrix} \right\|_{H^{\ell} \times H^{\ell}} \lesssim \|B\|^2_{H^{\ell+2}} \|(F,G)\|_{H^{\ell} \times H^{\ell}},
\end{equation}
as well as:
\begin{equation}\label{fdhbsjdhghgsdsh7473462Stronger}
\left\| \mathcal{A} \begin{pmatrix} F \\ G \end{pmatrix} \right\|_{L^2 \times L^2} \lesssim \|B\|^2_{W^{3, \infty}} \|(F,G)\|_{L^2 \times L^2}.
\end{equation}
\end{lem}

\begin{proof}
We recall the definition of the operator $\mathcal{A}$
\begin{align}\label{fmkdslkfdskjn5}
 \mathcal{A} &:= 
\begin{pmatrix} \frac12 M_1+L_1 + \frac12 T_1 & 0 \\ 0 & \frac12 M_2+L_2 +\frac12 T_2 \end{pmatrix}
+ \frac12\left[ \begin{pmatrix} M_1 + L_1 & - \mathrm{q} \\ \mathrm{q} & M_2 + L_2 \end{pmatrix}, K \right] \nonumber \\
&\quad + \frac12 \int_{0}^1 (1-s)^2 e^{-s K} \left[ \begin{pmatrix} -M_1 + T_1 & 0 \\ 0 & - M_2 + T_2 \end{pmatrix} + \left[ \begin{pmatrix} M_1 + L_1 & - \mathrm{q} \\ \mathrm{q} & M_2 + L_2 \end{pmatrix}, K \right] , K \right] e^{s K} \, \de s.
\end{align}

Let $\ell > 7/2$. We estimate the $H^{\ell} \times H^{\ell}$ norm of the right-hand side of \eqref{fmkdslkfdskjn5} term by term. From Lemmas \ref{Lemm3Linear}, \ref{Lemm4Linear} and \ref{Lemm5Linear}, we have 
\[
\left\| \begin{pmatrix} \frac12 M_1+L_1 + \frac12 T_1 & 0 \\ 0 & \frac12 M_2+L_2 + \frac12 T_2 \end{pmatrix} \begin{pmatrix} F \\ G \end{pmatrix} \right\|_{H^{\ell} \times H^{\ell}} \lesssim \|B\|^2_{H^{\ell+2}} \|(F,G)\|_{H^{\ell} \times H^{\ell}}.
\]
From inequality \eqref{qBound} and Lemmas \ref{Lemm2Linear}, \ref{Lemm3Linear}, and \ref{Lemm4Linear}, we deduce 
\begin{equation}\label{IneqCommMatr1}
\left\| \left[ \begin{pmatrix} M_1 + L_1 & - \mathrm{q} \\ \mathrm{q} & M_2 + L_2 \end{pmatrix}, K \right] \begin{pmatrix} F \\ G \end{pmatrix} \right\|_{H^{\ell} \times H^{\ell}} \lesssim \| B \|_{H^{\ell + 2}}^3 \|(F, G)\|_{H^{\ell -1 } \times H^{\ell -1}}.
\end{equation}
By Lemmas \ref{Lemm2Linear}, \ref{Lemm3Linear}, \ref{Lemm5Linear}, and \ref{ExpLemma}, we have for all $s \in [0,1]$: 
\[
\left\| e^{-s K} \left[ \begin{pmatrix} -M_1 + T_1 & 0 \\ 0 & - M_2 + T_2 \end{pmatrix}, K \right] e^{s K} \begin{pmatrix} F \\ G \end{pmatrix} \right\|_{H^{\ell} \times H^{\ell}} \lesssim \| B \|_{H^{\ell+2}}^3 \|(F, G)\|_{H^{\ell -1 } \times H^{\ell -1}}.
\] 
Furthermore, applying Lemmas \ref{Lemm2Linear} and \ref{ExpLemma} alongside inequality \eqref{IneqCommMatr1} yields for all $s \in [0,1]$: 
\[
\left\| e^{-s K} \left[ \left[ \begin{pmatrix} M_1 + L_1 & - \mathrm{q} \\ \mathrm{q} & M_2 + L_2 \end{pmatrix}, K \right] , K \right] e^{s K} \begin{pmatrix} F \\ G \end{pmatrix} \right\|_{H^{\ell} \times H^{\ell}} \lesssim \| B \|_{H^{\ell+2}}^4 \|(F, G)\|_{H^{\ell -2 } \times H^{\ell -2}}.
\] 
Collecting these estimates yields the desired bound \eqref{mathcalAbound}. 

Similarly, we estimate the $L^2 \times L^2$ norm. By Lemmas \ref{Lemm3Linear}, \ref{Lemm4Linear} and \ref{Lemm5Linear}: 
\[
\left\| \begin{pmatrix} \frac12 M_1+L_1 +\frac12  T_1 & 0 \\ 0 & \frac12 M_2+L_2 + \frac12 T_2 \end{pmatrix} \begin{pmatrix} F \\ G \end{pmatrix} \right\|_{L^2 \times L^2} \lesssim \|B\|^2_{W^{3, \infty}} \|(F,G)\|_{L^2 \times L^2}.
\]
From inequality \eqref{qBoundFractional} and Lemmas \ref{Lemm2Linear}, \ref{Lemm3Linear}, and \ref{Lemm4Linear}, we get 
\begin{equation}\label{IneqCommMatr1Bis}
\left\| \left[ \begin{pmatrix} M_1 + L_1 & - \mathrm{q} \\ \mathrm{q} & M_2 + L_2 \end{pmatrix}, K \right] \begin{pmatrix} F \\ G \end{pmatrix} \right\|_{L^2 \times L^2} \lesssim \| B \|_{W^{3, \infty}}^3 \|(F, G)\|_{H^{-1 } \times H^{-1}}.
\end{equation}
For the terms in the integral, for all $s \in [0,1]$: 
\[
\left\| e^{-s K} \left[ \begin{pmatrix} -M_1 + T_1 & 0 \\ 0 & - M_2 + T_2 \end{pmatrix}, K \right] e^{s K} \begin{pmatrix} F \\ G \end{pmatrix} \right\|_{L^2 \times L^2} \lesssim \| B \|_{W^{3, \infty}}^3 \|(F, G)\|_{H^{ -1 } \times H^{-1}},
\] 
and 
\[
\left\| e^{-s K} \left[ \left[ \begin{pmatrix} M_1 + L_1 & - \mathrm{q} \\ \mathrm{q} & M_2 + L_2 \end{pmatrix}, K \right] , K \right] e^{s K} \begin{pmatrix} F \\ G \end{pmatrix} \right\|_{L^2 \times L^2} \lesssim \| B \|_{W^{3, \infty}}^4 \|(F, G)\|_{H^{-2 } \times H^{-2}}.
\] 
Combining these completes the proof of \eqref{fdhbsjdhghgsdsh7473462Stronger}.
\end{proof}

\subsection{Nonlinear estimates for the remainder terms}
Having addressed the linear operators, we now turn our attention to the nonlinear remainders of the system. The goal of this subsection is to provide $H^\ell$ and $L^2$ bounds for these terms, by taking advantage of the exponential operator bounds and the algebraic properties derived in the previous sections.

\begin{lem}\label{RemainderEstimate1}
Let $\ell > 7/2$, and let $v, m, v', m' : \T^3 \to \R^3$ be zero-average and divergence-free vector fields. Assume $\|B\|_{H^{\ell +1}} \lesssim 1$. 
Then, for $j \in \{1,2\}$ we have that
\begin{align}
\| N_j(v,m)\|_{H^{\ell}} &\lesssim (e^{\|B\|_{H^{\ell +1}}} - 1) \| (v,m)\|_{H^{\ell} \times H^{\ell}} \| (v,m)\|_{H^{\ell-1} \times H^{\ell-1}}, \label{fndjskdjshbf664tg3w4r} \\
\| N_j(v,m) - N_j(v',m') \|_{L^2} &\lesssim (e^{\|B\|_{H^{\ell+1}}} - 1) \| (v-v',m-m')\|_{L^2 \times L^2} \nonumber \\
&\quad \times \big( \| (v,m)\|_{H^{\ell-1} \times H^{\ell-1}} + \| (v',m')\|_{H^{\ell-1} \times H^{\ell-1}} \big). \label{fndjskdjshbf664tg3w4rL2}
\end{align}
\end{lem}
\begin{proof}
Recalling definitions \eqref{Def:N_1Nonlinear} and \eqref{Def:N_2Nonlinear} of $N_j$, inequality \eqref{fndjskdjshbf664tg3w4r} follows from the algebra property of $H^{\ell}$ and inequality \eqref{fdnjskbjnfhjskhjdbfgsjk3}. Inequality \eqref{fndjskdjshbf664tg3w4rL2} follows from Hölder's inequality and inequalities \eqref{fdnjskbjnfhjskhjdbfgsjk3L2} and \eqref{fdnjskbjnfhjskhjdbfgsjk3} (together with the embedding $H^{\ell+1}(\T^3) \hookrightarrow W^{3, \infty}(\T^3)$).
\end{proof}

In the following lemma, we will prove a better estimate for the divergence of $N_2$. The crucial improvement is that there are no higher derivatives on $m$ on the right-hand side. 
 
\begin{lem}\label{ImprovedDIV}
Let $\ell > 7/2$, and let $v, m, v', m' : \T^3 \to \R^3$ be zero-average and divergence-free vector fields. Assume $\|B\|_{H^{\ell +2}} \lesssim 1$, then
\begin{align}
\| \dive N_2(v,m)\|_{H^{\ell}} &\lesssim (e^{\|B\|_{H^{\ell +2}}} - 1) \Big( \| v\|_{H^{\ell +1}} \| (v,m)\|_{H^{\ell-1} \times H^{\ell-1}} + \| (v,m)\|^2_{H^{\ell} \times H^{\ell}} \Big), \label{ImprovedDIVBound} \\
\| \dive (N_2(v,m) - N_2(v',m'))\|_{L^2} &\lesssim (e^{\|B\|_{H^{\ell+1}}} - 1) \Big( \| \nabla v - \nabla v' \|_{L^2} + \|m - m'\|_{L^2} \Big) \nonumber \\
&\quad \times \Big( \| (v,m)\|_{H^{\ell-1} \times H^{\ell-1}} + \| (v',m')\|_{H^{\ell-1} \times H^{\ell-1}} \Big). \label{ImprovedDIVBoundL2}
\end{align}
\end{lem}
\begin{proof}
Recalling definition \eqref{Def:N_2Nonlinear} of $N_2$ and the fact that 
\[
\dive \pi_j (e^{K} - \Id) \begin{pmatrix} v \\ m \end{pmatrix} = 0, \qquad j \in \{1,2\} \quad \mbox{(see Proposition \ref{prop:algebraic_properties}),}
\] 
we can use identity \eqref{OtimesVSnabla} to rewrite $\dive N_2$ as
\begin{align*}
\dive N_2(v,m) &= \left( \pi_2 (e^{K} - \Id) \begin{pmatrix} v \\ m \end{pmatrix} \cdot \nabla \right) v - (v \cdot \nabla ) \pi_2 (e^{K} - \Id) \begin{pmatrix} v \\ m \end{pmatrix} \\
&\quad + \dive \left(\pi_1 (e^{K} - \Id) \begin{pmatrix} v \\ m \end{pmatrix} \otimes \pi_2 (e^{K} - \Id) \begin{pmatrix} v \\ m \end{pmatrix} \right. \\
&\quad \left. - \pi_2 (e^{K} - \Id) \begin{pmatrix} v \\ m \end{pmatrix} \otimes (\pi_1 e^{K} - \Id) \begin{pmatrix} v \\ m \end{pmatrix} \right). 
\end{align*}
Inequality \eqref{ImprovedDIVBound} then follows from the algebra property of $H^{\ell}$ and inequality \eqref{fdnjskbjnfhjskhjdbfgsjk3}, while inequality \eqref{ImprovedDIVBoundL2} follows from Hölder's inequality and inequalities \eqref{fdnjskbjnfhjskhjdbfgsjk3}, \eqref{fdnjskbjnfhjskhjdbfgsjk3L2}, and \eqref{fdnjskbjnfhjskhjdbfgsjk3H1} (together with the embedding $H^{\ell+1}(\T^3) \hookrightarrow W^{3, \infty}(\T^3)$).
\end{proof}

\begin{lem}\label{NTildeLamma}
Let $\ell > 7/2$, and let $v, m, v', m' : \T^3 \to \R^3$ be zero-average and divergence-free vector fields. Assume $\|B\|_{H^{\ell +1}} \lesssim 1$. 
Then, for $j \in \{1,2\}$ we have that
\begin{align*}
\| \tilde N_j(v,m)\|_{H^{\ell}} &\lesssim \| (v,m)\|^2_{H^{\ell} \times H^{\ell}}, \\
\| \tilde N_j(v,m) - \tilde N_j(v',m')\|_{L^2} &\lesssim \| (v - v',m - m')\|_{L^2 \times L^2} \big(\| (v,m)\|_{H^{\ell} \times H^{\ell}} + \| (v',m')\|_{H^{\ell} \times H^{\ell}} \big).
\end{align*}
\end{lem}
\begin{proof}
Recalling definitions \eqref{DefTildeN1} and \eqref{DefTildeN2} of $\tilde N_j$, the first inequality follows from the algebra property of $H^{\ell}$ and inequality \eqref{fdnjskbjnfhjskhjdbfgsjk2}. The second inequality follows from Hölder's inequality and inequalities \eqref{fdnjskbjnfhjskhjdbfgsjk2} and \eqref{fdnjskbjnfhjskhjdbfgsjk2L2}.
\end{proof}

\begin{lem}\label{lem:GammaEstimate}
Let $\ell > 7/2$, and let $v, m, v', m' : \T^3 \to \R^3$ be zero-average and divergence-free vector fields. Assume $\|B\|_{H^{\ell +1}} \lesssim 1$, then
\begin{align}\label{GammaEstimate1}
\|\Gamma(v,m)\|_{H^{\ell}\times{H^{\ell}}} &\lesssim \left(e^{\|B\|_{H^{\ell +1}}} - 1\right) \| (v,m)\|^2_{H^{\ell} \times H^{\ell}},  \\
\label{GammaEstimate1L2} &\|\Gamma(v,m) - \Gamma(v',m')\|_{L^2 \times L^2} 
\\ 
&\lesssim \left(e^{\|B\|_{W^{3, \infty}}} - 1\right) \| (v - v',m- m')\|_{L^2 \times L^2} \nonumber  \big(\| (v,m)\|_{H^{\ell} \times H^{\ell}} + \| (v',m')\|_{ H^{\ell} \times H^{\ell}}\big). 
\end{align}
\end{lem}

\begin{proof}
Recalling definition \eqref{DEf:Gamma} of $\Gamma(v,m)$, we obtain using Lemmas \ref{Lemm2Linear} and \ref{NTildeLamma}, for all $j \geq 1$
\begin{align*}
\frac{1}{j !} \left\| (-K)^j \begin{pmatrix} \mathbb{P} \dive \tilde N_1(v,m) \\ \dive \tilde N_2(v,m) \end{pmatrix} \right\|_{H^{\ell}\times{H^{\ell}}}
&\lesssim \frac{1}{j !} \|B\|_{H^{\ell +1}}^{j} \left\| \begin{pmatrix} \mathbb{P} \dive \tilde N_1(v,m) \\ \dive \tilde N_2(v,m) \end{pmatrix} \right\|_{H^{\ell-1} \times H^{\ell -1}} \\ 
&\lesssim \frac{1}{j !} \|B\|_{H^{\ell +1}}^{j} \left\| \begin{pmatrix} \tilde N_1(v,m) \\ \tilde N_2(v,m) \end{pmatrix} \right\|_{H^{\ell} \times H^{\ell}} \\ 
&\lesssim \frac{1}{j !} \|B\|_{H^{\ell +1}}^{j} \| (v,m)\|^2_{H^{\ell} \times H^{\ell}},
\end{align*}
as well as
\begin{align*}
\frac{1}{j !}  \left\| (-K)^j \right. &\left. \begin{pmatrix} \mathbb{P} \dive ( \tilde N_1(v,m) - \tilde N_1(v',m') ) \\ \dive ( \tilde N_2(v,m) - \tilde N_2(v',m')) \end{pmatrix} \right\|_{L^2\times L^2}\\
&\lesssim \frac{1}{j !} \|B\|_{W^{3, \infty}}^{j} \left\| \begin{pmatrix} \mathbb{P} \dive ( \tilde N_1(v,m) - \tilde N_1(v',m') ) \\ \dive ( \tilde N_2(v,m) - \tilde N_2(v',m') ) \end{pmatrix} \right\|_{H^{-1} \times H^{-1}} \\ 
&\lesssim \|B\|_{W^{3, \infty}}^{j} \left\| \begin{pmatrix} \tilde N_1(v,m) - \tilde N_1(v',m') \\ \tilde N_2(v,m) - \tilde N_2(v',m') \end{pmatrix} \right\|_{L^2 \times L^2} \\ 
&\lesssim \frac{1}{j !} \|B\|_{W^{3, \infty}}^{j} \| (v - v',m- m')\|_{L^2 \times L^2} \big(\| (v,m)\|_{H^\ell \times H^\ell} + \| (v',m')\|_{H^\ell \times H^\ell} \big).
\end{align*}
The statement follows by summing over $j \geq 1$.
\end{proof}

\subsection{Flow estimates}

Here we prove some preliminary estimates on the flow of a divergence-free vector field in $\Zeta$, that is the Banach space defined in \eqref{Def:Z}, for $\ell$ sufficiently large. 
Given $t_0, t \geq 0$, we denote with $\Phi_{t_0,t} (x)$ the flow associated to the vector field $V$ between times $t_0$ and $t$, namely the solution of
\begin{equation}\label{nFlowDefInLemma}
\begin{cases}
\frac{\de}{\de t} \Phi_{t_0,t} (x)=V(t,\Phi_{t_0,t}(x)),\\
\Phi_{t_0,t_0} (x)=x,
\end{cases}
\end{equation}

We start by proving Sobolev estimates on the flow.

\begin{lem}\label{stime flussoBis}
Let $\ell \geq 4$, $t\geq t_0 \geq 0$, $\theta \in (0,1)$ and $\varepsilon \in (0, \theta)$. 
Let $V \in C^1([0, +\infty) \times \T^3; \R^3) $ be a divergence-free vector field satisfying
\begin{equation}\label{piccolezza spiegone InLemma}
\| V \|_{\Zeta} \lesssim  \varepsilon.
\end{equation}
Then, the flow $\Phi_{t_0,t}$ of $V$ satisfies
\begin{align}\label{stima flussiBis}
\sup_{t\geq 0} \| \nabla \Phi_{t_0,t} \|_{H^{\ell -1}}& \lesssim 1 .
\end{align}
\end{lem}
\begin{proof}
First we prove two preliminary bounds for the first and second partial derivatives. Namely, for all $\varepsilon >0$ it holds
\begin{equation}\label{Cidfgnusirg}
\sup_{t \geq 0} \| \partial_j  \Phi_{t_0,t}  \|_{L^{\infty}} 
 \lesssim e^{C \frac{ \varepsilon}{\theta}},
\end{equation}
and 
\begin{align} \label{fdnjskdhgjbgbg234}
\sup_{t \geq 0}\| \partial_i \partial_j  \Phi_{t_0,t} \|_{L^{\infty}} 
 \lesssim  e^{C \frac{\varepsilon}{\theta} } \frac{\varepsilon}{  \theta } .
\end{align}

\textbf{Proof of \eqref{Cidfgnusirg}}: from \eqref{nFlowDefInLemma}
we deduce
\begin{equation}\label{mnfdjksnnnfdjksjdnf}
\frac{\de}{\de t} \partial_j \Phi_{t_0,t} (x) = \sum_{k} (\partial_k V)(t, \Phi_{t_0,t} (x)) \, \partial_j (\Phi_{t_0,t} (x))_k,
\end{equation}
so that
$$
\frac{\de}{\de t} | \partial_j  \Phi_{t_0,t} (x)| = \frac{\partial_j  \Phi_{t_0,t} (x)}{|\partial_j  \Phi_{t_0,t} (x)|} \cdot  \frac{\de}{\de t} \partial_j \Phi_{t_0,t} (x)
\leq \left| \frac{\de}{\de t} \partial_j \Phi_{t_0,t} (x) \right| \lesssim  \| \nabla V(t,\cdot) \|_{L^\infty}  | \partial_j (\Phi_{t_0,t} (x))| .
$$
Thus, by Gronwall inequality, for all $t \geq 0$ we have
\begin{align} \label{fdnjskdhgjbgbg23}
| \partial_j  \Phi_{t_0,t} (x)| & \leq | \partial_j  \Phi_{t_0,t_0} (x) | e^{C\int_{t_0}^t \| \nabla V(s,\cdot) \|_{L^\infty} \, \de s} 
\\ & \nonumber
\lesssim 
e^{C \int_{t_0}^t \| V (s,\cdot) \|_{H^3} \, \de s}  \leq  e^{C \int_{t_0}^t e^{-\theta s} \| V (s,\cdot) \|_{\Zeta} \, \de s}  
\overset{\eqref{piccolezza spiegone InLemma}}{\leq}
 e^{C \varepsilon\int_{t_0}^t e^{-\theta s}  \, \de s} \leq e^{C\frac{\varepsilon}{\theta}} 
\end{align}
where in the second inequality we used Sobolev embeddings and the fact that $\Phi_{t_0,t_0}$ is the identity.\\
\\
\textbf{Proof of \eqref{fdnjskdhgjbgbg234}}: from \eqref{mnfdjksnnnfdjksjdnf} we deduce 
\begin{align*}
\frac{\de}{\de t} \partial_i \partial_j \Phi_{t_0,t} (x) &=
\sum_{k, k'}  (\partial_{k'} \partial_k  V)(t, \Phi_{t_0,t} (x)) \,  \partial_i (\Phi_{t_0,t} (x))_{k'} \, \partial_j (\Phi_{t_0,t} (x))_k 
\\
&+ \sum_{k}  (\partial_k V)(t, \Phi_{t_0,t} (x)) \, \partial_i \partial_j (\Phi_{t_0,t} (x))_k, 
\end{align*}
and then by using \eqref{fdnjskdhgjbgbg23} we obtain
\begin{align*}
\frac{\de}{\de t}  | \partial_i \partial_j  \Phi_{t_0,t} (x)| & = \frac{\partial_i \partial_j  \Phi_{t_0,t} (x)}{|\partial_i \partial_j  \Phi_{t_0,t} (x)|} \cdot  \frac{\de}{\de t} \partial_i \partial_j \Phi_{t_0,t} (x)
\leq \left| \frac{\de}{\de t} \partial_i \partial_j \Phi_{t_0,t} (x) \right| 
\\
& \lesssim  e^{C \frac{\varepsilon}{\theta}} \| \nabla^2  V(t,\cdot) \|_{L^{\infty}}  
+  | \partial_i \partial_j (\Phi_{t_0,t} (x))|  \| \nabla  V(t,\cdot) \|_{L^{\infty}}  .
\end{align*}
Thus, since $\partial_{i} \partial_{j} \Phi_{t_0, t_0} (x) = \partial_{i} \partial_{j} \Id = 0$, by Gronwall lemma we have 
\begin{align} \label{fdnjskdhgjbgbg234intheargument}
| \partial_i \partial_j  \Phi_{t_0,t} (x)| & \lesssim
e^{C \frac{\varepsilon}{\theta}} \int_{t_0}^t e^{C\int_s^t \| \nabla V(\tau,\cdot) \|_{L^\infty} \, d\tau}  \| \nabla^2  V(s,\cdot) \|_{L^{\infty}}  \, \de s
\\ & \nonumber
\lesssim e^{C \frac{\varepsilon}{\theta}}
\int_{t_0}^t e^{C\int_s^t \| V(\tau,\cdot) \|_{H^3} \, d\tau}  \|   V(s,\cdot) \|_{H^4}  \, \de s
\\ & \nonumber
\leq e^{C \frac{\varepsilon}{\theta} }
\int_{t_0}^t e^{C\int_s^t e^{- \theta \tau} \| V \|_{Z^{\ell,\theta}} \, d\tau} e^{ - \theta s} \| V \|_{Z^{\ell,\theta}}  \, \de s
\\ &
\overset{\eqref{piccolezza spiegone InLemma}}{\leq}
 e^{C \frac{\varepsilon}{\theta}} \int_{t_0}^t  e^{C \frac{\varepsilon}{\theta} - \theta s} \varepsilon  \, \de s
 \leq  e^{2 C \frac{\varepsilon}{\theta} } \frac{\varepsilon}{  \theta } .
\end{align}

\textbf{Proof of the main estimate}. We are now ready to prove the statement.  
Let $\alpha$
a multi-index with $1 \leq |\alpha| \leq \ell$. 
 From \eqref{nFlowDefInLemma} 
and Lemma \ref{LemmaFaaDiBruno} we deduce 
\begin{equation}
\frac{\de}{\de t} \partial^{\alpha} \Phi_{t_0,t} (x)=\sum_{\mu,\nu}C_{\mu,\nu}\, (\partial^\mu V)(t, \Phi_{t_0,t} (x)) 
\prod_{\overset{1\leq |\beta|\leq |\alpha|}{1\leq j\leq 3}}(\partial^\beta \Phi_{t_0,t} (x))^{\nu_{\beta, j}},
\end{equation}
where the constants $C_{\mu,\nu}$ are non-negative integers, and the sum is taken over those $\mu$ and $\nu$ such that $1\leq |\mu|\leq |\alpha|$, $\nu_{\beta, j}\in \N^*$,
$$
\sum_{1\leq |\beta|\leq |\alpha|}\nu_{\beta, j}=\mu_j, \mbox{ for }1\leq j\leq 3,\quad \mbox{and}
\quad \sum_{\overset{1\leq |\beta|\leq |\alpha|}{1\leq j\leq 3}}\beta\nu_{\beta, j}=\alpha
$$
Pairing the equation in $L^2$ with $\partial^{\alpha} \Phi_{t_0,t} (x)$ and using Cauchy-Schwarz we obtain
\begin{equation}
\frac{\de}{\de t} \| \partial^{\alpha} \Phi_{t_0,t} \|^2_{L^2} 
\lesssim  \left\| \sum_{\mu,\nu}C_{\mu,\nu}\, (\partial^\mu V)(t, \Phi_{t_0,t} ) 
\prod_{\overset{1\leq |\beta|\leq |\alpha|}{1\leq j\leq 3}}(\partial^\beta \Phi_{t_0,t} )^{\nu_{\beta, j}} \right\|_{L^2} \| \partial^{\alpha} \Phi_{t_0,t} \|_{L^2} .
\end{equation}
Summing over all the $\alpha$ with $1 \leq |\alpha| \leq \ell$ we arrive to
\begin{equation}
\frac{\de}{\de t} \| \nabla \Phi_{t_0,t} \|^2_{H^{\ell-1}} 
\lesssim \sup_{1 \leq |\alpha| \leq \ell} \left\| \sum_{\mu,\nu}C_{\mu,\nu}\, (\partial^\mu V)(t, \Phi_{t_0,t} ) 
\prod_{\overset{1\leq |\beta|\leq |\alpha|}{1\leq j\leq 3}}(\partial^\beta \Phi_{t_0,t} )^{\nu_{\beta, j}} \right\|_{L^2} \| \nabla \Phi_{t_0,t} \|_{H^{\ell-1}} .
\end{equation}
We now claim 
\begin{equation}\label{mfdjksldjngngngn}
\left\| \sum_{\mu,\nu}C_{\mu,\nu}\, (\partial^\mu V)(t, \Phi_{t_0,t} ) 
\prod_{\overset{1\leq |\beta|\leq |\alpha|}{1\leq j\leq 3}}(\partial^\beta \Phi_{t_0,t} )^{\nu_{\beta, j}} \right\|_{L^2} 
\lesssim \|  V(t,\cdot) \|_{H^{\ell}}  (1 + \| \nabla \Phi_{t_0,t} \|_{H^{\ell-1}}).
\end{equation}
If \eqref{mfdjksldjngngngn} were true, then by Young's inequality it follows that
$$
\frac{\de}{\de t} \| \nabla \Phi_{t_0,t} \|^2_{H^{\ell-1}}  \lesssim \|  V(t,\cdot) \|_{H^{\ell}}  \| \nabla \Phi_{t_0,t} \|^2_{H^{\ell-1}} + 
\|  V(t,\cdot) \|_{H^{\ell}},
$$
so that
\begin{align*}
\| \nabla \Phi_{t_0,t} \|^2_{H^{\ell-1}} 
& \leq \| \nabla \Phi_{t_0,t_0} \|^2_{H^{\ell-1}} e^{C \int_{t_0}^{t} \|  V(s,\cdot) \|_{H^{\ell}} \de s} + C 
\int_{t_0}^t e^{C \int_{s}^{t} \|  V(\tau,\cdot) \|_{H^{\ell}} \de \tau}  \|  V(s,\cdot) \|_{H^{\ell}} \, \de s
\\ &
\leq  \|  \nabla \Id \|^2_{L^2} e^{C \int_{t_0}^{t} e^{-\theta s} \|  V \|_{\Zeta} \de s}  + C \int_{t_0}^t e^{C  \int_{s}^{t} e^{-\theta \tau} \|  V \|_{\Zeta} \de \tau} e^{-\theta s}  \|V\|_{\Zeta} \, \de s
\\ &
\overset{\eqref{piccolezza spiegone InLemma}}{\lesssim}  e^{C \varepsilon \int_{t_0}^{t} e^{-\theta s}  \de s} 
+ \varepsilon \int_{t_0}^t e^{C \varepsilon \int_{s}^{t} e^{-\theta \tau} \de \tau} e^{-\theta s}  \, \de s
\\ &
\lesssim  e^{C \frac{\varepsilon}{\theta}} + \varepsilon e^{\frac{C}{\theta} \varepsilon} \int_{t_0}^t  e^{-\theta s}  \, \de s
\\ &
\lesssim  e^{C \frac{\varepsilon}{\theta}} + \frac{\varepsilon}{\theta} e^{\frac{C}{\theta} \varepsilon}.
\end{align*}
The desired estimate \eqref{stima flussiBis} follows then choosing $\varepsilon \in (0, \theta)$. \\
\\
\textbf{Proof of \eqref{mfdjksldjngngngn}}. We now prove the claim in \eqref{mfdjksldjngngngn}.
We start by considering the contribution of the multi-indices $\mu$ such that $1 \leq |\mu| \leq |\alpha| - 2$. We denote with 
$\widetilde{\sum}$ the restricted sum. Note that (see \eqref{recallthismfdkslejdgfkn}):
$$
\sum_{\overset{1\leq |\beta|\leq |\alpha|}{1\leq j\leq 3}}|\beta| \nu_{\beta, j}=|\alpha|.
$$
Letting 
$$
M := \sum_{\overset{1\leq |\beta|\leq |\alpha|}{1\leq j\leq 3}} \nu_{\beta, j} .
$$
and using Lemma \ref{LemProd} we obtain for $M \geq 2$ and for all $\varepsilon' >0$:
\begin{align}\label{mfdjksldjngngngn2}
\left\| \widetilde{\sum}_{\mu,\nu}C_{\mu,\nu}\, (\partial^\mu V)(t, \Phi_{t_0,t} ) 
\prod_{\overset{1\leq |\beta|\leq |\alpha|}{1\leq j\leq 3}}(\partial^\beta \Phi_{t_0,t} )^{\nu_{\beta, j}} \right\|_{L^2} 
&\lesssim 
\|\partial^\mu V (t,\cdot) \|_{L^{\infty}}  \| \nabla \Phi_{t_0,t} \|_{H^{|\alpha| - M + \varepsilon'}}  \| \nabla \Phi_{t_0,t} \|^{M-1}_{L^{\infty}},
\\ & \nonumber
\overset{\eqref{Cidfgnusirg}}{\lesssim}
e^{(M-1) C \frac{\varepsilon}{\theta}} \|  V(t,\cdot) \|_{H^{\ell}}  \| \nabla \Phi_{t_0,t}\|_{H^{\ell - M + \varepsilon'}}  .
\end{align}
On the other hand, if $M=1$ then only a single index $\nu_{\beta, j}$ is nonzero and equals one, thus we get immediately
\begin{align}\label{mfdjksldjngngngn3}
\left\| \widetilde{\sum}_{\mu,\nu}C_{\mu,\nu}\, (\partial^\mu V)(t, \Phi_{t_0,t}) 
\prod_{\overset{1\leq |\beta|\leq |\alpha|}{1\leq j\leq 3}}(\partial^\beta \Phi_{t_0,t})^{\nu_{\beta, j}} \right\|_{L^2} 
&\lesssim 
\|\partial^\mu V (t,\cdot) \|_{L^{\infty}}  \| \nabla \Phi_{t_0,t} \|_{H^{|\alpha| - 1 }}
\\ & \nonumber
\lesssim
  \|  V(t,\cdot) \|_{H^{\ell}}  \| \nabla \Phi_{t_0,t} \|_{H^{|\alpha| - 1}}  .
\end{align}

To handle the other cases it will be useful to recall that we have defined a multi-index $\beta$ in the sum as {\it active} if there exists a $j \in \{1,2,3 \}$ such that $\nu_{\beta, j} \geq 1$. Note that the $\beta$ that are {\it non active} do not give any contribution to the sum.

We then consider the contribution of the multi-indices $\mu$ such that $|\mu| = |\alpha| $. We denote with 
$\widetilde{\sum}$ the restricted sum. We deduce from identity \eqref{MainCounting}:
\begin{equation}
0 =\sum_{\overset{1\leq |\beta|\leq k}{1\leq j\leq 3}}  (|\beta| - 1)\nu_{\beta, j},
\end{equation} 
so that we can infer that   
$|\beta| = 1$ for all {\it active} $\beta$. Thus 
\begin{align}\label{mfdjksldjngngngn5}
\left\| \widetilde{\sum}_{\mu,\nu}C_{\mu,\nu}\, (\partial^\mu V)(t, \Phi_{t_0,t}) 
\prod_{\overset{1\leq |\beta|\leq |\alpha|}{1\leq j\leq 3}}(\partial^\beta \Phi_{t_0,t})^{\nu_{\beta, j}} \right\|_{L^2} 
&\lesssim 
\|\partial^\mu V (t,\cdot) \|_{L^2}    \| \nabla \Phi_{t_0,t} \|^{M}_{L^{\infty}},
\\ & \nonumber
\overset{\eqref{Cidfgnusirg}}{\lesssim}
 e^{M C \frac{\varepsilon}{\theta}} \|  V(t,\cdot) \|_{H^{|\alpha|}}    .
\end{align}

We finally consider the contribution of the multi-indices $\mu$ such that $|\mu| = |\alpha| - 1$. We denote with 
$\widetilde{\sum}$ the restricted sum. We deduce from identity \eqref{MainCounting}:
\begin{equation*}
1 =\sum_{\overset{1\leq |\beta|\leq |\alpha|}{1\leq j\leq 3}}  (|\beta| - 1)\nu_{\beta, j},
\end{equation*} 
so that we can infer that there exists a single couple $(\tilde{\beta}, \tilde{j})$ with $|\tilde{\beta}| = 2$ and 
 $\nu_{\tilde \beta, \tilde j} = 1$, $\nu_{\tilde \beta, j} = 0$ for all $j \neq \tilde j$, while we have   
$|\beta| = 1$ for all {\it active} $\beta \neq \tilde \beta$. Thus 
\begin{align}\label{mfdjksldjngngngn4}
\left\| \widetilde{\sum}_{\mu,\nu}C_{\mu,\nu}\, (\partial^\mu V)(t, \Phi_{t_0,t}) 
\prod_{\overset{1\leq |\beta|\leq |\alpha|}{1\leq j\leq 3}}(\partial^\beta \Phi_{t_0,t})^{\nu_{\beta, j}} \right\|_{L^2} 
&\lesssim 
\|\partial^\mu V (t,\cdot) \|_{L^2}  \| \nabla^2 \Phi_{t_0,t} \|_{L^{\infty}}  \| \nabla \Phi_{t_0,t} \|^{M-1}_{L^{\infty}},
\\ & \nonumber
\overset{\eqref{Cidfgnusirg}-\eqref{fdnjskdhgjbgbg234}}{\lesssim}
\frac{\varepsilon}{  \theta } e^{M C \frac{\varepsilon}{\theta}} \|  V(t,\cdot) \|_{H^{|\alpha|-1}}    .
\end{align}

After renaming the constants and choosing $\varepsilon \in (0, \theta)$ and 
$\varepsilon' >0$ sufficiently small (in fact $\varepsilon' \in (0, 1]$ suffices), the claim \eqref{mfdjksldjngngngn} follows by 
\eqref{mfdjksldjngngngn2}-\eqref{mfdjksldjngngngn3}-\eqref{mfdjksldjngngngn5}-\eqref{mfdjksldjngngngn4}. 
This concludes the proof.
\end{proof}

We now prove a quantitative stability result for the flows of vector fields in the space $\Zeta$.
\begin{lem}\label{stime_flussoBisStability}
Let $t_0 \geq 0$ and $\theta > 0$. 
Let $V \in C^1([0, +\infty) \times \T^3; \R^3)$ be a divergence-free vector field  
and let $W \in C^1([0, +\infty) \times \T^3; \R^3)$ be such that
\begin{equation}\label{ndfjskjd6465723}
\sup_{t \geq 0} e^{t \theta} \| \nabla W (t,\cdot) \|_{L^{\infty}} \lesssim 1.
\end{equation}
Let $\Phi_{t_0,t}$ and $\Psi_{t_0,t}$ be the flows associated with the vector fields $V$ and $W$, respectively.  Then, for all $t \geq t_0$, we have
\begin{equation}\label{stima_flussiBisStability}
 \| \Phi_{t_0,t} - \Psi_{t_0,t}\|_{L^2} \leq \frac{C}{\theta} \|V - W\|_{\ZetaZero},
\end{equation}
for some constant $C > 0$.
\end{lem}

\begin{proof}
We have 
\begin{align*}
\frac{\de}{\de t} (\Phi_{t_0,t} (x) - \Psi_{t_0,t} (x)) &= V(t,\Phi_{t_0,t}(x)) - W(t,\Psi_{t_0,t}(x)) \\ 
&= V(t,\Phi_{t_0,t}(x)) - W(t,\Phi_{t_0,t}(x)) + W(t,\Phi_{t_0,t}(x)) - W(t,\Psi_{t_0,t}(x)).
\end{align*}
Taking the scalar product with $\Phi_{t_0,t}(x) - \Psi_{t_0,t}(x)$ and using the Cauchy-Schwarz inequality, we get
\begin{align*}
\frac{1}{2} \frac{\de}{\de t} |\Phi_{t_0,t} (x) - \Psi_{t_0,t} (x)|^2 &\leq 
| V(t,\Phi_{t_0,t}(x)) - W(t,\Phi_{t_0,t}(x))| |\Phi_{t_0,t} (x) - \Psi_{t_0,t} (x)| \\
&\quad + | W(t,\Phi_{t_0,t}(x)) - W(t,\Psi_{t_0,t}(x))| |\Phi_{t_0,t} (x) - \Psi_{t_0,t} (x)|.
\end{align*}
On the other hand, using Cauchy-Schwarz again, we obtain
\begin{align*}
| V(t,\Phi_{t_0,t}(x)) - W(t,\Phi_{t_0,t}(x))| |\Phi_{t_0,t} (x) - \Psi_{t_0,t} (x)| &\leq 
\frac{1}{2} e^{\theta t} | V(t,\Phi_{t_0,t}(x)) - W(t,\Phi_{t_0,t}(x))|^2 \\
&\quad + \frac{1}{2} e^{-\theta t} |\Phi_{t_0,t} (x) - \Psi_{t_0,t} (x)|^2.
\end{align*}
We also have  
\[
|W(t,\Phi_{t_0,t}(x)) - W(t,\Psi_{t_0,t}(x))| \leq \| \nabla W(t) \|_{L^{\infty}} |\Phi_{t_0,t}(x) - \Psi_{t_0,t}(x)|.
\]
Thus, integrating over $\T^3$, we arrive at
\begin{align*}
\frac{\de}{\de t} \| \Phi_{t_0,t} - \Psi_{t_0,t} \|_{L^2}^2 
&\leq 2 \left( \| \nabla W(t,\cdot) \|_{L^{\infty}} + \frac{1}{2} e^{-\theta t} \right) \| \Phi_{t_0,t} - \Psi_{t_0,t} \|_{L^2}^2 \\
&\quad + e^{\theta t} \| V(t,\cdot) - W(t,\cdot)\|_{L^2}^2,
\end{align*}
where we used the fact that $V$ is divergence-free, which in particular implies that the flow $\Phi$ is measure-preserving. Recalling that $x = \Phi_{t_0,t_0} x = \Psi_{t_0,t_0} x$, Grönwall's lemma gives
\begin{align*}
\| \Phi_{t_0,t} - \Psi_{t_0,t} \|_{L^2}^2 
&\leq \int_{t_0}^t e^{2 \int_{s}^t \left( \| \nabla W(\tau,\cdot) \|_{L^{\infty}} + \frac{1}{2} e^{-\theta \tau} \right) \de \tau} e^{\theta s} \| V(s,\cdot) - W(s,\cdot)\|_{L^2}^2 \, \de s \\
&\overset{\eqref{ndfjskjd6465723}}{\lesssim} \int_{t_0}^t \| V(s,\cdot) - W(s,\cdot)\|_{L^2}^2 \, \de s \\
&\lesssim \int_{t_0}^t e^{- 2\theta s} e^{2 \theta s}\| V(s,\cdot) - W(s,\cdot)\|_{L^2}^2 \, \de s \\
&\leq \|V - W\|^2_{\ZetaZero} \int_{t_0}^t e^{- 2\theta s} \, \de s \lesssim \frac{1}{\theta} \|V - W\|^2_{\ZetaZero},
\end{align*}
where in the second inequality we used the crude bound
\begin{equation}\label{CBfnjdksjng}
e^{\int_{s}^t \left( \| \nabla W(\tau) \|_{L^{\infty}} + \frac{1}{2} e^{-\theta \tau} \right) \de \tau}
\overset{\eqref{ndfjskjd6465723}}{\leq} e^{\int_{s}^t \left( e^{-\theta \tau} C + \frac{1}{2} e^{-\theta \tau} \right) \de \tau}
\lesssim 1.
\end{equation}
This concludes the proof.
\end{proof}

The following result will be used to prove the convergence of the flow associated to the velocity field as $t \to \infty$.

\begin{lem}\label{LastFlowBound}
Let $t_0 \geq 0$ and $\theta > 0$. 
Let $V \in C^1([0, +\infty) \times \T^3; \R^3)$ be a divergence-free vector field  
that satisfies 
\begin{equation}\label{ndfjskjd6465723Bis}
\sup_{t \geq 0} e^{t \theta} \|  V (t,\cdot) \|_{W^{1, \infty}} \lesssim \rho.
\end{equation}
and let  $\Psi_{t_0,t}(x)$ be the associated flow.
Then, there exists a volume preserving $C^1$ diffeomorphism
$$ \Psi_{t_0,\infty} : x \in \T^3 \to \Psi_{t_0,\infty}(x) \in \T^3, $$ 
 that satisfies
\begin{equation}
\lim_{t \to \infty} \| \Psi_{t_0,t}(\cdot) - \Psi_{t_0,\infty}(\cdot)\|_{W^{1,\infty}} = 0.
\end{equation}
\end{lem}

\begin{proof}
Let $t_0 \leq t_1 \leq t_2$. 
First we note that 
$$\Psi_{t_0,t_2}(x) - \Psi_{t_0,t_1}(x) = \int_{t_1}^{t_2} \frac{\de}{\de s} \Psi_{t_0, s}(x) \, \de s $$
Thus
\begin{align}\label{4783877590345}
| \Psi_{t_0,t_2}(x) -  \Psi_{t_0,t_1}(x) | 
& \leq  \int_{t_1}^{t_2} \left| \frac{\de}{\de s}  \Psi_{t_0, s}(x) \right| \, \de s 
\\ & \nonumber
 = \int_{t_1}^{t_2} | V(s, \Psi_{t_0, s}(x)) | \, \de s 
\\ & \nonumber
\overset{\eqref{ndfjskjd6465723Bis}}{\lesssim} \rho \int_{t_1}^{t_2} e^{-\theta s}  \, \de s
\lesssim  \frac{\rho}{\theta} e^{-\theta t_1}. 
\end{align}
Then, proceeding exactly as in the proof of Lemma \ref{stime flussoBis} (in particular, 
we refer to the inequality~\eqref{fdnjskdhgjbgbg23}) we get
\begin{equation}\label{femkjfeuw44bfdjsk}
| \partial_j  \Psi_{t_0,t} (x)|  \leq | \partial_j  \Psi_{t_0,t_0} (x) | e^{C\int_{t_0}^t \| \nabla V(s,\cdot) \|_{L^\infty} \, \de s} 
\lesssim 
  e^{C \rho \int_{t_0}^t e^{-\theta s}  \, \de s}  
 \leq e^{C\frac{\rho}{\theta}}  
\end{equation}
Thus (recall \eqref{mnfdjksnnnfdjksjdnf})
\begin{align}\label{476875934}
|\partial_j \Psi_{t_0,t_2}(x) - \partial_j \Psi_{t_0,t_1}(x) | 
& \leq  \int_{t_1}^{t_2} \left| \frac{\de}{\de s}  \partial_j \Psi_{t_0, s}(x) \right| \, \de s 
\\ & \nonumber
 \lesssim \int_{t_1}^{t_2} | \nabla V(s, \Psi_{t_0, s}(x)) | |\partial_j \Psi_{t_0, s}(x)| \, \de s 
\\ & \nonumber
\overset{\eqref{ndfjskjd6465723Bis}-\eqref{femkjfeuw44bfdjsk}}{\lesssim} e^{C\frac{\rho}{\theta}}   \rho \int_{t_1}^{t_2} e^{-\theta s}  \, \de s
\leq  \frac{\rho}{\theta} e^{C\frac{\rho}{\theta}}  e^{-\theta t_1}. 
\end{align}
Estimates \eqref{4783877590345} and \eqref{476875934} show that
$\{\Psi_{t_0,t}\}_{t\geq t_0}$ is a Cauchy family in $C^1(\T^3;\T^3)$.
Therefore, there exists a map
\[
\Psi_{t_0,\infty}\in C^1(\T^3;\T^3),
\]
such that
\[
\Psi_{t_0,t}\longrightarrow\Psi_{t_0,\infty}
\qquad\text{in }W^{1,\infty}(\T^3),
\]
as $t\to\infty$. The statement follows by using a similar argument for the inverse of $\Psi_{t_0,t}$ and the divergence-free condition on $V(t)$.
\end{proof}

\section{The iterative scheme}\label{Sec:The iterative scheme}
Let $\ell \in \N$ with $\ell \geq 4$, and let $\theta \in (0,1)$. 
Recall that we defined the space
\begin{equation}
\Zeta:=\left\{ F\in L^\infty([0, + \infty);H^\ell(\T^3)): \|F\|_{\Zeta}
:= \sup_{t \geq 0}e^{\theta t}\|F(t,\cdot)\|_{H^\ell}<\infty\right\}.
\end{equation}
We initialise the scheme by setting $(v_0, m_0) = (0,0)$. For $n \in \N$, assume that the $n$-th iterate $(v_n, m_n) : [0, +\infty) \times \T^3 \to \R^3$ is well-defined and satisfies the following properties for all $t \geq 0$
\begin{align}
\int_{\T^3} v_n(t,x) \, \de x &= \int_{\T^3} m_n(t,x) \, \de x = 0, \label{piccolezza spiegone-2} \\
\dive v_n(t,x) &= \dive m_n(t,x) = 0. \label{piccolezza spiegone-1}
\end{align}
Furthermore, assume the smallness condition
\begin{equation}\label{piccolezza spiegone}
\|v_n\|_{\Zeta} + \|\nabla v_n\|_{L^2_t H^\ell} + \|m_n\|_{\Zeta} \leq \varepsilon.
\end{equation}
We denote by $\Ph_{t,s}$ the flow associated with the vector field $V(v_n, m_n)$, defined as the solution to
\begin{equation}\label{nFlowDef}
\begin{cases}
\partial_s \Ph_{t,s} (x)=(V(v_n, m_n))(s,\Ph_{t,s}(x)),\\
\Ph_{t,t} x=x,
\end{cases}
\end{equation}
where we recall that the distorted velocity field $V$ is defined as
\begin{equation}\label{|def:Vn}
V(v_n, m_n) := v_n + \pi_1 (e^{K} - \Id) \begin{pmatrix} v_n \\ m_n \end{pmatrix}.
\end{equation}
Then, we define the next iterate $(\vn,\mn)$ as the solution of the following integral system
\begin{equation}\label{iterazione integrale}
\begin{cases}
\begin{aligned}
v_{n+1}(t,x) &= e^{t\Delta} y(x) + \int_0^t e^{(t-s)\Delta} \p \Big( (m_n \cdot \nabla) m_n - (v_n \cdot \nabla) v_n + \dive N_1(v_n, m_n) \Big)(s,x) \, \de s \\
&\quad + \int_0^t e^{(t-s)\Delta} \pi_1 \left( \mathcal{A} \begin{pmatrix} v_n \\ m_n \end{pmatrix} + \Gamma(v_n, m_n) \right)(s,x) \, \de s,
\end{aligned} \\
\\
\begin{aligned}
m_{n+1}(t,x) &= -\int_t^\infty \Big( (m_{n+1} \cdot \nabla) V(v_n, m_n) \Big)(s, \Ph_{t,s}(x)) \, \de s \\ 
&\quad - \int_t^\infty \left( \pi_2 \mathcal{A} \begin{pmatrix} v_n \\ m_n \end{pmatrix} + \pi_2 \Gamma(v_n, m_n) + \dive N_2(v_n, m_n) \right)(s, \Ph_{t,s}(x)) \, \de s,
\end{aligned}
\end{cases}
\end{equation}
where $y \in H^\ell(\T^3)$ is a zero-average, divergence-free vector field.
Observe that $v_{n+1}$ is given explicitly as a function of the previous iterate $(v_n, m_n)$. In contrast, $m_{n+1}$ is introduced through an implicit relation (an integral equation along the characteristics), and its well-posedness is therefore not immediate. Theorem \ref{ThMFirstRecursiveBound} below establishes the existence and uniqueness of $m_{n+1}$, thereby ensuring that the iteration is well-defined. It is useful to note that the first equation in \eqref{iterazione integrale} is simply the Duhamel formulation of the following differential problem:
\begin{equation}\label{iterazione}
\begin{cases}
\begin{aligned}
\partial_t v_{n+1} &= \Delta v_{n+1} + \pi_1 \mathcal{A} \begin{pmatrix} v_n \\ m_n \end{pmatrix} + \pi_1 \Gamma(v_n, m_n) \\
&\quad + \p \Big( (m_n \cdot \nabla) m_n - (v_n \cdot \nabla) v_n + \dive N_1(v_n, m_n) \Big),
\end{aligned} \\
v_{n+1}(0,x) = y(x).
\end{cases}
\end{equation}

\subsection{First recursive bound}\label{Sec:Picard-1}

The parameter $\varepsilon_1$ in the following statement represents a small positive constant. For any $\rho > 0$, we define the closed ball in $\Zeta$ as:
$$
\mathcal{B}(\rho) := \left\{ m \in \Zeta \ : \ \|m\|_{\Zeta} \leq \rho \right\}.
$$

\begin{thm}\label{ThMFirstRecursiveBound}
Let $\ell \geq 4$ and $\theta = \frac14$. There exists $\varepsilon_1 \in (0,1)$ such that for all $\varepsilon \in (0, \varepsilon_1)$, the following holds. Let $B$ be a given field satisfying $\|B\|_{H^{\ell +2}} \leq \varepsilon$, and let $(v_n, m_n) : [0, +\infty) \times \T^3 \to \R^3$ be a pair of vector fields satisfying \eqref{piccolezza spiegone-2}--\eqref{piccolezza spiegone}. 
Then, the map $\Psi$ defined on $\mathcal{B}(\varepsilon)$ by
\begin{align*}
m\in \mathcal{B}(\e)\mapsto\Psi(m)(t,x) := & -\int_t^\infty \Big( (m \cdot \nabla) V(v_n, m_n) \Big)(s, \Ph_{t,s}(x)) \, \de s \\
& - \int_t^\infty \left( \pi_2 \mathcal{A} \begin{pmatrix} v_n \\ m_n \end{pmatrix} + \pi_2 \Gamma(v_n, m_n) + \dive N_2(v_n, m_n) \right)(s, \Ph_{t,s}(x)) \, \de s
\end{align*}
is a contraction. Moreover, there exists a constant $C \geq 1$ such that
\begin{equation}\label{fnjksnbfyd6}
\Psi(\mathcal{B}(\varepsilon)) \subseteq \mathcal{B}\left(C \varepsilon^2\right) \subseteq \mathcal{B}(\varepsilon).
\end{equation}
As a consequence, there exists a unique fixed point $m_{n+1} \in \mathcal{B}(\varepsilon)$ such that
\begin{equation}\label{Mn+1IsWellDEf}
m_{n+1} = \Psi(m_{n+1}),
\end{equation}
which additionally satisfies
\begin{equation}\label{fmdkslngngdgnljs734}
\int_{\T^3} m_{n+1}(t,x) \, \de x = 0, \quad \dive m_{n+1}(t,x) = 0 \qquad \forall t \geq 0.
\end{equation}
\end{thm}

\begin{rem}
We emphasize that, by virtue of \eqref{fnjksnbfyd6}, the unique fixed point satisfies the bound
$$ 
\|m_{n+1}\|_{\Zeta} \leq C \varepsilon^2.
$$
\end{rem}

\begin{proof}
Recalling the definition \eqref{|def:Vn} of the distorted velocity field $V(\cdot, \cdot)$, and using the inductive hypothesis \eqref{piccolezza spiegone-1} alongside the algebraic properties in Proposition \ref{prop:algebraic_properties}, we have for all $t \geq 0$
\begin{equation}\label{DiveVn}
\dive V(v_n, m_n) = \dive v_n + \dive \pi_1 (e^{K} - \Id) \begin{pmatrix} v_n \\ m_n \end{pmatrix} = 0.
\end{equation}
Consequently, Liouville's theorem implies that the associated flow is volume-preserving, namely:
$$
(\det \nabla \Ph_{t,s} )(x) = 1 \qquad \forall s \geq t \geq 0, \, x \in \T^3.
$$

By applying inequality \ref{fdnjskbjnfhjskhjdbfgsjk3}, we estimate the Sobolev norm of $V$ as follows:
\begin{align*}
\| V(v_n, m_n)\|_{H^{\ell}} &\leq \| v_n \|_{H^{\ell}} + \left\| \pi_1 (e^{K} - \Id) \begin{pmatrix} v_n \\ m_n \end{pmatrix} \right\|_{H^{\ell}} \\
&\lesssim \| v_n \|_{H^{\ell}} + \varepsilon \|(v_n, m_n)\|_{H^{\ell-1} \times H^{\ell-1}} \\
& \leq  \| v_n \|_{H^{\ell}} + \varepsilon \left( \| v_n \|_{H^{\ell}} + \| m_n \|_{H^{\ell}} \right).
\end{align*}
Note that this bound holds point-wise in time, even if we omit the time dependence to simplify the
notation. We will keep doing this without mentioning. Taking the supremum in time with the exponential weight $e^{\theta t}$, we obtain
\begin{align*}
\| V(v_n, m_n)\|_{\Zeta} &\leq C \| v_n \|_{\Zeta} + C \varepsilon \left( \| v_n \|_{\Zeta} + \| m_n \|_{\Zeta} \right) \\
&\overset{\eqref{piccolezza spiegone}}{\leq} C \varepsilon + 2 C \varepsilon^2 \leq 2 C \varepsilon.
\end{align*}
We are now in a position to apply Lemma \ref{stime flussoBis}, yielding the uniform bound on the flow
\begin{equation}\label{AlsoThisfdlsdlgkm}
\| \nabla \Ph_{t,s} \|_{H^{\ell-1}} \lesssim 1 \qquad \forall s \geq t \geq 0.
\end{equation}

Next, we establish the contraction property. Let $m, m' \in \mathcal{B}(\varepsilon)$. Invoking Lemma \ref{CorollaryFaaDiBruno} with $F = ((m - m') \cdot \nabla) \pi_1 (e^{K} - \Id) \begin{pmatrix} v_n \\ m_n \end{pmatrix}$ and employing estimates \eqref{AlsoThisfdlsdlgkm} and \eqref{fdnjskbjnfhjskhjdbfgsjk3}, we control the non-linear contribution by
\begin{align*}
\left\| \int_{t}^{\infty}\right.&\left. \Big( ((m - m') \cdot \nabla) \pi_1 (e^{K} - \Id) \begin{pmatrix} v_n \\ m_n \end{pmatrix} \Big)(s, \Ph_{t,s}) \, \de s \right\|_{H^{\ell}} \\
&\leq \int_{t}^{\infty} \left\| \Big( ((m - m') \cdot \nabla) \pi_1 (e^{K} - \Id) \begin{pmatrix} v_n \\ m_n \end{pmatrix} \Big) (s,\cdot) \right\|_{H^{\ell}} \, \de s \\
&\lesssim \varepsilon \int_{t}^{\infty} \| (m-m')(s,\cdot)\|_{H^{\ell}} \|(v_n, m_n)(s,\cdot)\|_{H^{\ell} \times H^{\ell}} \, \de s \\
&\lesssim \varepsilon \| m-m'\|_{\Zeta} \left( \|v_n \|_{\Zeta} + \| m_n\|_{\Zeta} \right) \int_{t}^{\infty} e^{-2\theta s} \, \de s \\
&\overset{\eqref{piccolezza spiegone}}{\lesssim} \frac{\varepsilon^2}{2 \theta} \| m-m'\|_{\Zeta} \,e^{-2\theta t}.
\end{align*}

Similarly, applying Lemma \ref{CorollaryFaaDiBruno} to the transport term with $F = ((m-m') \cdot \nabla) v_n$, we obtain
\begin{align}\label{LeadingMBound}
\left\| \int_t^\infty \Big( ((m-m') \cdot \nabla) v_n \Big)(s, \Ph_{t,s}) \, \de s \right.&\left.\right\|_{H^{\ell}} 
\leq \int_t^\infty \| ((m-m') \cdot \nabla) v_n(s,\cdot) \|_{H^{\ell}} \, \de s \nonumber \\
&\leq \int_0^\infty \mathbf{1}_{[t, \infty)}(s) \| (m-m') (s,\cdot)\|_{H^{\ell}} \| \nabla v_n(s,\cdot) \|_{H^{\ell}} \, \de s.
\end{align}
Utilizing the Cauchy-Schwarz inequality with respect to the time variable and exploiting the $L^2_t H^\ell$ bound on $\nabla v_n$ specified in \eqref{piccolezza spiegone}, we deduce
\begin{align*}
\text{R.H.S. of } \eqref{LeadingMBound} &\leq \left( \int_t^\infty \| (m-m') (s,\cdot)\|^2_{H^{\ell}} \, \de s \right)^{1/2} \| \nabla v_n \|_{L^2_t H^\ell} \\
&\overset{\eqref{piccolezza spiegone}}{\leq} \varepsilon \| m-m' \|_{\Zeta} \left( \int_t^\infty e^{-2\theta s} \, \de s \right)^{1/2} = \frac{\varepsilon}{\sqrt{2\theta}} \| m-m' \|_{\Zeta}\, e^{- \theta t}.
\end{align*}
Since the difference satisfies $\Psi(m) - \Psi(m') = - \int_t^\infty ((m-m') \cdot \nabla) V(v_n, m_n) \, \de s$, combining the two previous estimates we get that
$$
\|\Psi(m)(t,\cdot) - \Psi(m')(t,\cdot) \|_{H^{\ell}} \lesssim \| m-m'\|_{\Zeta} \left( \frac{\varepsilon^2}{2 \theta} e^{-2\theta t} + \frac{\varepsilon}{\sqrt{2\theta}} e^{- \theta t} \right) \qquad \forall t \geq 0.
$$
Multiplying by $e^{\theta t}$ and taking the supremum over $t \geq 0$, we finally obtain
$$
\|\Psi(m) - \Psi(m') \|_{\Zeta} \lesssim \|m - m'\|_{\Zeta} \left( \frac{\varepsilon^2}{2 \theta} + \frac{\varepsilon}{\sqrt{2\theta}} \right).
$$
Therefore, provided $\varepsilon_1$ (and consequently $\varepsilon$) is chosen sufficiently small, the map $\Psi$ is a contraction on $\mathcal{B}(\varepsilon)$.

We now show that $\Psi$ maps $\mathcal{B}(\varepsilon)$ into a smaller ball of radius $C\varepsilon^2$, i.e.
\begin{equation}\label{fdsjkjhdkfngjsknjdgs23453}
\|\Psi(m)\|_{\Zeta} \leq C \varepsilon^2.
\end{equation}
Setting $m'=0$ in the previous inequality directly implies
\begin{equation}\label{fjdiosnfg5}
\sup_{t \geq 0}e^{\theta t} \left\| \int_{t}^{\infty} \Big( (m \cdot \nabla) V(v_n, m_n) \Big)(s, \Ph_{t,s}) \, \de s \right\|_{H^{\ell}} \leq C \varepsilon \|m \|_{\Zeta} \leq C \varepsilon^2.
\end{equation}
It remains to control the remaining source terms in the definition of $\Psi$. Applying Lemma \ref{CorollaryFaaDiBruno} with $F = \mathcal{A} \begin{pmatrix} v_n \\ m_n \end{pmatrix}$, together with bounds \eqref{AlsoThisfdlsdlgkm} and \eqref{mathcalAbound}, we have:
\begin{align}\label{hfdgsugdyfdsUff3}
\left\| \int_t^\infty \pi_2 \mathcal{A} \begin{pmatrix} v_n \\ m_n \end{pmatrix} (s, \Ph_{t,s}) \, \de s \right\|_{H^{\ell}} 
&\leq \int_t^\infty \left\| \pi_2 \mathcal{A} \begin{pmatrix} v_n \\ m_n \end{pmatrix} (s) \right\|_{H^{\ell}} \, \de s \nonumber \\
&\leq \varepsilon^2 \int_t^\infty \left( \| v_n(s,\cdot) \|_{H^{\ell}} + \| m_n(s,\cdot) \|_{H^{\ell}} \right) \, \de s \nonumber \\
&\leq \varepsilon^2 \left( \| v_n \|_{\Zeta} + \| m_n \|_{\Zeta} \right) \int_t^\infty e^{-\theta s} \, \de s \nonumber \\
&\overset{\eqref{piccolezza spiegone}}{\leq} \varepsilon^3 \int_t^\infty e^{-\theta s} \, \de s = \frac{\varepsilon^3}{\theta} e^{- \theta t},
\end{align}
which leads to
\begin{equation}\label{fjdiosnfg3}
\sup_{t \geq 0}e^{\theta t} \left\| \int_t^\infty \pi_2 \mathcal{A} \begin{pmatrix} v_n \\ m_n \end{pmatrix} (s, \Ph_{t,s}) \, \de s \right\|_{H^{\ell}} \leq \frac{\varepsilon^{3}}{\theta}.
\end{equation}

Similarly, utilizing Lemma \ref{CorollaryFaaDiBruno} with $F = \pi_2 \Gamma (v_n, m_n)$, alongside estimates \eqref{AlsoThisfdlsdlgkm}, \eqref{GammaEstimate1}, and the elementary inequality
\begin{equation}\label{fdnjskdjggsdjnkgsd46783}
e^{\|B\|_{H^{\ell +2}}} - 1 \leq 2 \varepsilon \qquad \forall \varepsilon \in (0, \varepsilon_1),
\end{equation}
we obtain
\begin{align*}
\left\| \int_t^\infty \pi_2 \Gamma(v_n, m_n) (s, \Ph_{t,s}) \, \de s \right\|_{H^{\ell}} 
&\leq \int_t^\infty \left\| \pi_2 \Gamma(v_n, m_n) (s) \right\|_{H^{\ell}} \, \de s \\
&\label{Gamma_calc}\lesssim \varepsilon \int_t^\infty \left( \| v_n(s) \|^2_{H^{\ell}} + \| m_n(s) \|^2_{H^{\ell}} \right) \, \de s \\
&\leq \varepsilon \left( \| v_n \|^2_{\Zeta} + \| m_n \|^2_{\Zeta} \right) \int_t^\infty e^{-2 \theta s} \, \de s \\
&\overset{\eqref{piccolezza spiegone}}{\leq} \varepsilon^3 \int_t^\infty e^{-2 \theta s} \, \de s \leq \frac{\varepsilon^3}{2 \theta} e^{-2 \theta t}.
\end{align*}
This handles the $\Gamma$ term, ensuring that
\begin{equation}\label{fjdiosnfg2}
\sup_{t \geq 0}e^{\theta t} \left\| \int_t^\infty \pi_2 \Gamma(v_n, m_n) (s, \Ph_{t,s}) \, \de s \right\|_{H^{\ell}} \leq \frac{\varepsilon^{3}}{2\theta}.
\end{equation}

Lastly, we bound the divergence term. Invoking Lemma \ref{CorollaryFaaDiBruno} with $F = \dive N_2 (v_n, m_n)$, and employing \eqref{AlsoThisfdlsdlgkm}, \eqref{ImprovedDIVBound}, and \eqref{fdnjskdjggsdjnkgsd46783}, we find
\begin{align*}
\left\| \int_t^\infty \dive N_2(v_n, m_n) (s, \Ph_{t,s}) \, \de s \right\|_{H^{\ell}} 
&\leq \int_t^\infty \left\| \dive N_2(v_n, m_n) (s,\cdot) \right\|_{H^{\ell}} \, \de s \\
&\lesssim \varepsilon \int_t^\infty \| v_n(s,\cdot) \|_{H^{\ell+1}} \left( \| m_n(s,\cdot) \|_{H^{\ell}} + \| v_n(s,\cdot) \|_{H^{\ell}} \right) \, \de s \\
&+ \int_t^\infty \left( \| m_n(s,\cdot) \|^2_{H^{\ell}} + \| v_n(s,\cdot) \|^2_{H^{\ell}} \right) \, \de s.
\end{align*}
Since $\int_{\T^3} v_n = 0$, the Poincaré inequality yields $\| v_n(s,\cdot) \|_{H^{\ell+1}} \lesssim \| \nabla v_n(s,\cdot) \|_{H^{\ell}}$. Re-applying the Cauchy-Schwarz argument executed in \eqref{LeadingMBound}, we know that
$$ 
\int_t^\infty \| m_n (s,\cdot)\|_{H^{\ell}} \| \nabla v_n(s,\cdot) \|_{H^{\ell}} \, \de s \leq \frac{\varepsilon^{2}}{\sqrt{2\theta}} e^{- \theta t},
$$
and an identical bound holds when $m_n$ is replaced by $v_n$. Gathering these estimates, we deduce
\begin{align*}
\left\| \int_t^\infty \dive N_2(v_n, m_n) (s, \Ph_{t,s}) \, \de s \right\|_{H^{\ell}} 
&\lesssim \frac{\varepsilon^{3}}{\sqrt{2\theta}} e^{- \theta t} + \varepsilon \left( \| m_n \|^2_{\Zeta} + \| v_n \|^2_{\Zeta} \right) \int_t^\infty e^{-2 \theta s} \, \de s \\
&\overset{\eqref{piccolezza spiegone}}{\leq} \frac{\varepsilon^{3}}{\sqrt{2\theta}} e^{- \theta t} + \frac{\varepsilon^3}{2 \theta} e^{-2 \theta t} \lesssim \frac{\varepsilon^{3}}{\theta} e^{- \theta t},
\end{align*}
which translates to the bound
\begin{equation}\label{fjdiosnfg1}
\sup_{t \geq 0}e^{\theta t} \left\| \int_t^\infty \dive N_2(v_n, m_n) (s, \Ph_{t,s}) \, \de s \right\|_{H^{\ell}} \lesssim \frac{\varepsilon^{3}}{\theta}.
\end{equation}

Combining \eqref{fjdiosnfg5}, \eqref{fjdiosnfg3}, \eqref{fjdiosnfg2}, and \eqref{fjdiosnfg1}, and choosing $\varepsilon_1$ sufficiently small, we conclude that \eqref{fdsjkjhdkfngjsknjdgs23453} holds, proving that $\Psi(\mathcal{B}(\varepsilon)) \subseteq \mathcal{B}(C\varepsilon^2) \subseteq \mathcal{B}(\varepsilon)$. The Banach fixed-point theorem ensures the existence of a unique fixed point $m_{n+1}$ satisfying $\|m_{n+1}\|_{\Zeta} \leq C\varepsilon^2$.\\

It remains to verify the structural properties \eqref{fmdkslngngdgnljs734}. Recall that the flow maps satisfy the standard semigroup property $\Ph_{t,s}(\Ph_{\tau,t}(y)) = \Ph_{\tau,s}(y)$. Performing the change of variables $x = \Ph_{\tau,t}(y)$ in the fixed-point equation yields:
\begin{align}\label{fdjskjkfdskhje64}
m_{n+1}(t, \Ph_{\tau,t}(y)) = & -\int_t^\infty \Big( (m_{n+1} \cdot \nabla) V(v_n, m_n) \Big)(s, \Ph_{\tau,s}(y)) \, \de s \\
& - \int_t^\infty \left( \pi_2 \mathcal{A} \begin{pmatrix} v_n \\ m_n \end{pmatrix} + \pi_2 \Gamma(v_n, m_n) + \dive N_2(v_n, m_n) \right) (s, \Ph_{\tau,s}(y)) \, \de s. \nonumber
\end{align}
Differentiating \eqref{fdjskjkfdskhje64} with respect to $t$ using the relation $\partial_t \Ph_{\tau,t}(y) = V(v_n, m_n)(t, \Ph_{\tau,t}(y))$, we find
\begin{equation}\label{fdjuhsiuhfghdskgjdhgsfdhFischer2}
\frac{\de}{\de t} \Big( m_{n+1}(t, \Ph_{\tau,t}(y)) \Big) = \Big( \partial_t m_{n+1} + (V(v_n, m_n) \cdot \nabla) m_{n+1} \Big)\Big|_{(t, \Ph_{\tau,t}(y))}.
\end{equation}
Applying the Fundamental Theorem of Calculus to the right-hand side of \eqref{fdjskjkfdskhje64} and subsequently taking the limit $\tau \to t$, we recover the differential problem satisfied by $m_{n+1}$
\begin{equation}\label{kmflòdskljfdjsn74}
\partial_t m_{n+1} + (V(v_n, m_n) \cdot \nabla) m_{n+1} = (m_{n+1} \cdot \nabla) V(v_n, m_n) + \pi_2 \mathcal{A} \begin{pmatrix} v_n \\ m_n \end{pmatrix} + \pi_2 \Gamma(v_n, m_n) + \dive N_2(v_n, m_n).
\end{equation}
Taking the divergence of \eqref{kmflòdskljfdjsn74}, and using $\dive V(v_n, m_n) = 0$ along with the structural identities from Proposition \ref{prop:algebraic_properties}, we see that $\dive m_{n+1}$ satisfies a pure transport equation
$$
\partial_t \dive m_{n+1} + (V(v_n, m_n) \cdot \nabla) \dive m_{n+1} = 0.
$$
Hence, $\dive m_{n+1}$ is constant along the characteristics. Since $m_{n+1} \in \Zeta$, it decays exponentially as $t \to +\infty$, forcing $\dive m_{n+1}(t, \cdot) \to 0$. Thus, we must have $\dive m_{n+1}(t, x) = 0$ identically for all $t \geq 0$. Finally, integrating \eqref{kmflòdskljfdjsn74} over the torus $\T^3$ and utilizing $\dive m_{n+1} = \dive V = 0$ combined with Proposition \ref{prop:algebraic_properties}, all the spatial derivative terms vanish, leading to:
$$
\frac{\de}{\de t} \int_{\T^3} m_{n+1}(t,x) \, \de x = 0 \qquad \forall t \geq 0.
$$
Since $m_{n+1} \in \Zeta$ implies $\int_{\T^3} m_{n+1}(t,x) \, \de x \to 0$ as $t \to +\infty$, the conserved integral must be identically zero for all times. This concludes the proof.
\end{proof}

We conclude this section with some technical lemmas that will be used in the proof of the convergence of the iteration in Section \ref{Sec:Convergence of the iterative scheme}. The proof of these lemmas follows the scheme of the proof of Theorem \ref{ThMFirstRecursiveBound}.

\begin{lem}\label{Bunch1}
Let $\ell > \frac52$, $\theta = \frac14$, and let $m, v : [0,+\infty) \times \T^3 \to \R^3$ be zero-average, divergence-free vector fields. 
Let $\Phi_{t,s} : \T^3 \to \T^3$, $s \geq t \geq 0$, be a family of $C^1$ maps satisfying
\begin{equation}\label{JacobianBoundedInLemma}
\det \nabla \Phi_{t,s} = 1 \qquad \forall s \geq t \geq 0.
\end{equation} 
If $v$ and $m$ satisfy the smallness condition
\begin{equation}\label{piccolezza spiegoneInLemma}
\| \nabla v \|_{L^2_t H^{\ell}} + \|v\|_{\Zeta} + \|m\|_{\Zeta} \leq \varepsilon,
\end{equation} 
then we have the estimate
\begin{equation}\label{fnjwknbdshdsfg533}
\sup_{t \geq 0}e^{\theta t}\left\| \int_t^\infty \left((m \cdot \nabla) v \right)(s, \Phi_{t,s}) \, \de s \right\|_{L^2} \lesssim \varepsilon \|m\|_{\ZetaZero},
\end{equation}
as well as
\begin{equation}\label{fnjwknbdshdsfg534}
\sup_{t \geq 0}e^{\theta t} \left\| \int_t^\infty \left((m \cdot \nabla) v \right)(s, \Phi_{t,s}) \, \de s \right\|_{L^2} \lesssim \varepsilon \|\nabla v\|_{L^2_t L^2}.
\end{equation}
\end{lem}

\begin{proof}
Using the volume-preserving property \eqref{JacobianBoundedInLemma}, we control the $L^2$ norm in space as follows
\begin{align}\label{LeadingMBoundLemma}
\left\| \int_t^\infty \left(( m \cdot \nabla) v \right)(s, \Phi_{t,s}) \, \de s \right\|_{L^2}
&\leq \int_t^\infty \| ((m \cdot \nabla) v)(s,\cdot) \|_{L^2} \, \de s \nonumber \\
&\leq \int_t^\infty \| m (s,\cdot)\|_{L^{2}} \| \nabla v(s,\cdot) \|_{L^{\infty}} \, \de s.
\end{align}
Since $\ell > \frac{5}{2}$, the Sobolev embedding $H^\ell(\T^3) \hookrightarrow C^1(\T^3)$ implies $\|\nabla v(s,\cdot)\|_{L^\infty} \lesssim \|v(s,\cdot)\|_{H^\ell}$. Utilizing the smallness assumption \eqref{piccolezza spiegoneInLemma}, we obtain
\begin{align*}
\text{R.H.S. of } \eqref{LeadingMBoundLemma} &\lesssim \varepsilon \int_t^\infty e^{-\theta s} \| m(s,\cdot)\|_{L^2} \, \de s \\
&\leq \varepsilon \|m\|_{\ZetaZero} \int_t^\infty e^{-2\theta s} \, \de s = \frac{\varepsilon}{2\theta} \|m\|_{\ZetaZero} e^{- 2 \theta t}.
\end{align*}
Consequently, taking the supremum over $t \geq 0$ with the exponential weight $e^{\theta t}$, we get
\begin{equation}\label{fjdiosnfg4Lemma}
\sup_{t \geq 0}e^{\theta t} \left\| \int_t^\infty \left((m \cdot \nabla) v \right)(s, \Phi_{t,s}) \, \de s \right\|_{L^2} \lesssim \sup_{t \geq 0} \frac{\varepsilon}{2\theta} \|m\|_{\ZetaZero} e^{- \theta t} \leq \frac{\varepsilon}{2\theta} \|m\|_{\ZetaZero},
\end{equation}
which proves \eqref{fnjwknbdshdsfg533}.

In order to prove \eqref{fnjwknbdshdsfg534}, we shift the $L^\infty$ estimate onto the vector field $m$. Applying Hölder's inequality in space yields
\begin{align*}
\left\| \int_t^\infty \left((m \cdot \nabla) v \right)(s, \Phi_{t,s}) \, \de s \right\|_{L^2} 
&\leq \int_t^\infty \| ((m \cdot \nabla) v)(s,\cdot) \|_{L^2} \, \de s \\
&\leq \int_0^\infty \mathbf{1}_{[t, \infty)}(s) \| m (s,\cdot)\|_{L^{\infty}} \| \nabla v (s,\cdot) \|_{L^2} \, \de s.
\end{align*}
Since $\ell > \frac{5}{2} > \frac{3}{2}$, we exploit the embedding $H^\ell(\T^3) \hookrightarrow L^\infty(\T^3)$ and the uniform bound $\|m\|_{\Zeta} \leq \varepsilon$ from \eqref{piccolezza spiegoneInLemma} to bound $\|m(s,\cdot)\|_{L^\infty} \lesssim \varepsilon e^{-\theta s}$. Applying the Cauchy-Schwarz inequality with respect to time, we find
\begin{align*}
\left\| \int_t^\infty \left((m \cdot \nabla) v \right)(s, \Phi_{t,s}) \, \de s \right\|_{L^2} 
&\lesssim \varepsilon \left( \int_0^\infty \mathbf{1}_{[t, \infty)}(s) e^{- 2 \theta s} \, \de s \right)^{1/2} \| \nabla v \|_{L^2_t L^2} \\
&\leq \frac{\varepsilon}{\sqrt{2\theta}} e^{- \theta t} \| \nabla v \|_{L^2_t L^2}.
\end{align*}
Multiplication by $e^{\theta t}$ and taking the supremum over $t \geq 0$ yields
\begin{equation}
\sup_{t \geq 0}e^{\theta t} \left\| \int_t^\infty \left((m \cdot \nabla) v \right)(s, \Phi_{t,s}) \, \de s \right\|_{L^2} \lesssim \frac{\varepsilon}{\sqrt{2\theta}} \| \nabla v \|_{L^2_t L^2},
\end{equation}
from which \eqref{fnjwknbdshdsfg534} follows. Note that for this second bound, the Sobolev embedding only required $\ell > \frac{3}{2}$.
\end{proof}

\begin{lem}\label{Bunch2}
Let $\ell > \frac72$, $\theta = \frac14$ and  let $v, v', m, m' : [0,+\infty) \times \T^3 \to \R^3$  zero-average 
divergence-free vector fields and let $\|B\|_{H^{\ell + 1}} \leq \varepsilon$. 
Let $\Phi_{t,s} : \T^3\to\T^3$, $s \geq t \geq 0$ be a family of $C^1$ maps such that   
\begin{equation}\label{JacobianBoundedInLemmaB2}
  \det \nabla \Phi_{t,s} = 1, \qquad \forall s \geq t \geq 0.
\end{equation} 
If the vector fields $v,v',m,m'$ satisfy the smallness condition
\begin{equation}\label{piccolezza spiegoneInLemmaB2}
 \|v\|_{\Zeta}+\|m\|_{\Zeta} + \|v'\|_{\Zeta} + \|m'\|_{\Zeta} \leq \e,
\end{equation} 
then, for $j \in \{ 1,2\}$ we have that
\begin{equation}\label{fnjwknbdshdsfg533B2}
\sup_{t \geq 0}e^{\theta t} 
 \left\| \int_t^\infty \left( \pi_j \Gamma(v,m) - \pi_j \Gamma(v',m') \right) (s,\Phi_{t,s}) \, \de s \right\|_{L^2} 
  \nonumber
\lesssim \varepsilon^2 \left( \|v - v'\|_{\ZetaZero} + \|m - m'\|_{\ZetaZero}\right) .
\end{equation}
\end{lem}
\begin{proof}

Using the volume-preserving property \eqref{JacobianBoundedInLemmaB2} and inequality \eqref{GammaEstimate1L2} we get
\begin{align}\label{LeadingMBoundLemmaB2}
& \left\| \int_t^\infty \left( \pi_j \Gamma(v,m) - \pi_j \Gamma(v',m') \right)(s,\Phi_{t,s}) \, \de s \right\|_{L^2}
\\ \nonumber
&
\leq \int_t^\infty \| \left( \pi_j \Gamma(v,m) - \pi_j \Gamma(v',m') \right)(s,\cdot) \|_{L^2}  \, \de s 
\\ & \nonumber
\lesssim \varepsilon \int_t^\infty  \| (v - v',m- m')(s,\cdot)\|_{L^2 \times L^2} (\| (v,m)(s,\cdot)\|_{H^{\ell} \times H^{\ell}} 
+ \| (v',m')(s,\cdot)\|_{H^{\ell} \times H^{\ell}})  \, \de s 
\\ & \nonumber
\overset{\eqref{piccolezza spiegoneInLemmaB2}}{\leq} 
\varepsilon^2 (\|v-v'\|_{\ZetaZero} + \|m-m'\|_{\ZetaZero}) \int_t^\infty e^{-2\theta s}  \, \de s   
\\ & \nonumber
\leq \frac{\varepsilon^2}{2\theta} (\|v-v'\|_{\ZetaZero} + \|m-m'\|_{\ZetaZero})  e^{- 2\theta t}.
\end{align}
Thus, multiplying by $e^{\theta t}$ and taking the supremum over $t \geq 0$ we get the result.
\end{proof}

\begin{lem}\label{Bunch3}
Under the same assumption of Lemma \ref{Bunch2}, we have
\begin{equation}\label{fnjwknbdshdsfg533B3}
\sup_{t \geq 0}e^{\theta t} \left\| \int_t^\infty \left(  \dive \left(  N_2(v,m) -  N_2(v',m')\right) \right) (s,\Phi_{t,s}) \, \de s \right\|_{L^2} \lesssim \varepsilon^2 \left(  \| \nabla (v - v')  \|_{L^{2}_tL^2 }       
+    \|m - m'\|_{\ZetaZero} \right).
\end{equation}
\end{lem}
\begin{proof}

Using the volume-preserving property \eqref{JacobianBoundedInLemmaB2} and \eqref{ImprovedDIVBoundL2} we get
\begin{align}\label{LeadingMBoundLemmaB3}
& \left\| \int_t^\infty \left(  \dive \left(  N_2(v,m) -  N_2(v',m')\right) \right)(s,\Phi_{t,s}) \, \de s \right\|_{L^2}
\\  \nonumber
&
\leq \int_t^\infty \| \left( \dive \left(  N_2(v,m) -  N_2(v',m')\right) \right)(s,\cdot) \|_{L^2}  \, \de s 
\\ & \nonumber
\leq \varepsilon \int_t^\infty   \Big( 
\| \nabla (v - v')(s,\cdot)  \|_{L^2 }   + \|(m - m')(s,\cdot)\|_{L^2} \Big)
 \Big( \| (v,m)(s,\cdot)\|_{H^{\ell -1} \times H^{\ell -1}} + \| (v',m')(s,\cdot)\|_{H^{\ell -1} \times H^{\ell -1}}  \Big)   \, \de s 
\\ & \nonumber
\overset{\eqref{piccolezza spiegoneInLemmaB2}}{\leq} 
\varepsilon^2 \int_0^\infty 1_{[t, \infty)}(s) e^{-\theta s}  \| \nabla (v - v')(s,\cdot)  \|_{L^2 }    \, \de s   
+ \varepsilon^2 \int_t^\infty e^{-\theta s}  \|(m - m')(s,\cdot)\|_{L^2}  \, \de s   
\\ & \nonumber
\leq 
\varepsilon^2   \| \nabla (v - v')  \|_{L^{2}_tL^2 }  \left( \int_0^\infty 1_{[t, \infty)}(s)  e^{-2\theta s}  \, \de s  \right)^{1/2}    
+ \varepsilon^2 \|m - m'\|_{\ZetaZero} \int_t^\infty e^{-\theta s}    \, \de s   
\\ & \nonumber
\leq 
\frac{\varepsilon^2}{\sqrt{2 \theta}}  e^{-\theta t} \| \nabla (v - v')  \|_{L^{2}_tL^2 }       
+ \frac{\varepsilon^2}{\theta}  e^{-\theta t} \|m - m'\|_{\ZetaZero}    
\end{align}
Thus, multiplying by $e^{\theta t}$ and taking the supremum over $t \geq 0$ we get
\begin{equation}\label{fjdiosnfg4LemmaB3} 
\sup_{t \geq 0}e^{\theta t} \left\| \int_t^\infty \left(  \dive \left(  N_2(v,m) -  N_2(v',m')\right) \right)(s,\Phi_{t,s}) \, \de s \right\|_{L^2}
  \lesssim 
\frac{\varepsilon^2}{\sqrt{2 \theta}}   \| \nabla (v - v')  \|_{L^{2}_tL^2 }       
+ \frac{\varepsilon^2}{\theta}   \|m - m'\|_{\ZetaZero}.
\end{equation}
\end{proof}

\subsection{Second recursive bound}\label{Sec:Picard}

In this subsection, we establish the second recursive bound, which guarantees the uniform control of the sequence of velocity fields generated by the iterative scheme. We begin by stating the principal theorem of this section.

\begin{thm}\label{ThmSecondIterative}
Let $\ell \geq 4$ and $\theta = \frac14$. There exists $\varepsilon_1 \in (0,1)$ such that for all $\varepsilon \in (0, \varepsilon_1)$ the following holds. Let $B$ be a given field such that $\|B\|_{H^{\ell +2}} \leq \varepsilon$, and let $v_n, m_n : [0, +\infty) \times \T^3 \to \R^3$ be a pair of vector fields satisfying \eqref{piccolezza spiegone-2}--\eqref{piccolezza spiegone}.
Then, the vector field $\vn$ defined via \eqref{iterazione integrale} satisfies
\begin{equation}
\dive \vn(t,x) =0, \quad \int_{\T^3} \vn(t,x) \, \de x = 0 \qquad \forall t \geq 0,
\end{equation}
and
\begin{equation}
\|\vn\|_{\Zeta}+\|\nabla \vn\|_{L^2_tH^\ell} \leq 2 \|y\|_{H^\ell} + C \varepsilon^2
\end{equation}
for some constant $C \geq 1$.
\end{thm}

Most of the estimates required to prove Theorem \ref{ThmSecondIterative} are established in the following two lemmas, which we state and prove before proceeding to the main proof.

\begin{rem}
It is worth pointing out that $v_n, m_n,$ and $v_{n+1}$ may be generic vector fields in Lemmas \ref{FTWOM1}--\ref{gmjdskhkjdsfgjkgsdnj9435782}; however, we deliberately keep the notation $v_n, m_n, v_{n+1}$ to maintain consistency and avoid confusion with the scheme.  
\end{rem}

\begin{lem}\label{FTWOM1}
Let $\ell \geq 4$ and let $B$ be a vector field such that $\|B\|_{H^{\ell +2}} \leq \varepsilon$. Let $v_n, m_n, v_{n+1}: \T^3 \to \R^3$. Then the following estimates hold
\begin{equation}\label{fdnjskngjds1}
\left| \int_{\T^3} \left( \DL \pi_1 \mathcal{A} \begin{pmatrix} v_n \\ m_n \end{pmatrix} \right) \cdot \DL \vn \, \de x \right| 
\leq \frac{1}{100} \| \vn\|^2_{H^{\ell}} + C \varepsilon^4 (\|v_n\|^2_{H^{\ell}} + \|m_n\|^2_{H^{\ell}}),
\end{equation}
\begin{equation}\label{fdnjskngjds2}
\left| \int_{\T^3} \left( \DL \pi_1 \Gamma(v_n,m_n) \right) \cdot \DL \vn \, \de x \right| 
\leq \frac{1}{100} \| \vn\|^2_{H^{\ell}} + C \varepsilon^2 (\|v_n\|^4_{H^{\ell}} + \|m_n\|^4_{H^{\ell}}),
\end{equation}
and
\begin{equation}\label{fdnjskngjds3}
\left| \int_{\T^3} \left( \DL N_1(v_n,m_n) \right) \cdot \nabla \DL \vn \, \de x \right| 
\leq \frac{1}{100} \| \vn\|^2_{H^{\ell +1}} + C \varepsilon^2 (\|v_n\|^4_{H^{\ell}} + \|m_n\|^4_{H^{\ell}}).
\end{equation}
\end{lem}

\begin{proof}
Throughout the proof, $C>0$ will denote several constants, possibly increasing from line to line. We start by proving \eqref{fdnjskngjds1}: from inequality \eqref{mathcalAbound} and the Cauchy-Schwarz inequality, we get
\begin{align*}
\left| \int_{\T^3} \left( \DL \pi_1 \mathcal{A} \begin{pmatrix} v_n \\ m_n \end{pmatrix} \right) \cdot \DL \vn \, \de x \right|
&\leq \left\| \DL \pi_1 \mathcal{A} \begin{pmatrix} v_n \\ m_n \end{pmatrix} \right\|_{L^2} \|\DL \vn\|_{L^2} \\
&\leq \left\| \mathcal{A} \begin{pmatrix} v_n \\ m_n \end{pmatrix} \right\|_{H^{\ell} \times H^{\ell}} \| \vn\|_{H^{\ell}} \\
&\leq C \varepsilon^2 (\|v_n\|_{H^{\ell}} + \|m_n\|_{H^{\ell}}) \| \vn\|_{H^{\ell}} \\
&\leq \frac{1}{100} \| \vn\|^2_{H^{\ell}} + C \varepsilon^4 (\|v_n\|^2_{H^{\ell}} + \|m_n\|^2_{H^{\ell}}),
\end{align*}
where in the last step we applied Young's inequality. 
We now consider \eqref{fdnjskngjds2}: similarly to the previous case, we use the inequality \eqref{GammaEstimate1} and we obtain
\begin{align*}
\left| \int_{\T^3} \left( \DL \pi_1 \Gamma(v_n,m_n) \right) \cdot \DL \vn \, \de x \right|
&\leq \left\| \DL \pi_1 \Gamma(v_n,m_n) \right\|_{L^2} \|\DL \vn\|_{L^2} \\
&\leq \left\| \Gamma(v_n,m_n) \right\|_{H^{\ell} \times H^{\ell}} \| \vn\|_{H^{\ell}} \\
&\leq C \varepsilon (\|v_n\|^2_{H^{\ell}} + \|m_n\|^2_{H^{\ell}}) \| \vn\|_{H^{\ell}} \\
&\leq \frac{1}{100} \| \vn\|^2_{H^{\ell}} + C \varepsilon^2 (\|v_n\|^4_{H^{\ell}} + \|m_n\|^4_{H^{\ell}}).
\end{align*}

Lastly, in order to prove \eqref{fdnjskngjds3} we use \eqref{fndjskdjshbf664tg3w4r} and we get that
\begin{align*}
\left| \int_{\T^3} \left( \DL N_1(v_n,m_n) \right) \cdot \nabla \DL \vn \, \de x \right|
&\leq \left\| \DL N_1(v_n,m_n) \right\|_{L^2} \|\nabla \DL \vn\|_{L^2} \\
&\leq \left\| N_1(v_n,m_n) \right\|_{ H^{\ell}} \| \vn\|_{H^{\ell+1}} \\
&\leq C \varepsilon (\|v_n\|^2_{H^{\ell}} + \|m_n\|^2_{H^{\ell}}) \| \vn\|_{H^{\ell +1}} \\
&\leq \frac{1}{100} \| \vn\|^2_{H^{\ell+1}} + C \varepsilon^2 (\|v_n\|^4_{H^{\ell}} + \|m_n\|^4_{H^{\ell}}).
\end{align*}
This concludes the proof of the lemma.
\end{proof}

\begin{lem}\label{gmjdskhkjdsfgjkgsdnj9435782}
Let $\ell \geq 4$, $\theta = \frac14$, and let $B$ be a vector field such that $\|B\|_{H^{\ell +2}} \leq \varepsilon$. Let $v_n, m_n : [0, + \infty) \times \T^3 \to \R^3$ be zero-average, divergence-free vector fields such that
\begin{equation}\label{fdnjskfdjhbsjksdfh664}
\|v_n\|_{\Zeta} + \|m_n\|_{\Zeta} \lesssim \varepsilon.
\end{equation}
Then the following estimates hold
\begin{equation}\label{fdkjskldnfgnujkuu65}
\sup_{t \geq 0}e^{\theta t} \left\| \int_0^t e^{(t-s)\Delta} \p ((m_n \cdot \nabla) m_n) (s,\cdot) \, \de s \right\|_{H^{\ell}} \lesssim \varepsilon^2,
\end{equation}
\begin{equation}\label{fdkjskldnfgnujkuu64}
\sup_{t \geq 0}e^{\theta t} \left\| \int_0^t e^{(t-s)\Delta} \p ((v_n \cdot \nabla) v_n) (s,\cdot) \, \de s \right\|_{H^{\ell}} \lesssim \varepsilon^2,
\end{equation}
and
\begin{equation}\label{fdkjskldnfgnujkuu64N1}
\sup_{t \geq 0}e^{\theta t} \left\| \int_0^t e^{(t-s)\Delta} \p \dive N_{1}(v_n, m_n) (s,\cdot) \, \de s \right\|_{H^{\ell}} \lesssim \varepsilon^3.
\end{equation}
\end{lem}

\begin{proof}
We first prove \eqref{fdkjskldnfgnujkuu65}.
Using the divergence-free property $\dive m_n = 0$ and the identity \eqref{OtimesVSnabla}, we rewrite the integral as:
$$
\int_0^t e^{(t-s)\Delta} \p ((m_n \cdot \nabla) m_n) (s,\cdot) \, \de s = \int_0^t e^{(t-s)\Delta} \p \dive (m_n \otimes m_n) (s,\cdot) \, \de s.
$$
Then we bound the $H^\ell$ norm for $t \geq 0$:
\begin{equation}\label{fmgdjksldngjk1574y839}
e^{\theta t} \left\| \int_0^t e^{(t-s)\Delta} \p \dive (m_n \otimes m_n) (s,\cdot) \, \de s \right\|_{H^{\ell}} \leq e^{\theta t} \int_0^{t} \|e^{(t-s)\Delta}\dive (m_n \otimes m_n)(s,\cdot)\|_{H^\ell} \, \de s.
\end{equation}
Let $\sigma \in (0, 1)$ be chosen such that $\sigma \in (\theta , 2 \theta]$. Note that $(0,1) \cap (\theta , 2 \theta] \neq \emptyset$ for all $\theta \in (0,1)$, so such a $\sigma$ always exists. Since we have fixed $\theta = \frac14$, we can simply choose $\sigma = \frac12$. It will be instructive, however, to keep track of the explicit dependence on $\theta$ and $\sigma$ in the proof.

We rewrite the heat operator as $e^{(t-s)\Delta} = e^{\sigma(t-s)\Delta}e^{(1 - \sigma)(t-s) \Delta}$, and we use Lemma \ref{lemma calore} to get, for all $t \geq 0$
\begin{align*}
e^{\theta t} \int_0^{t} &\|e^{(t-s)\Delta}\dive (m_n \otimes m_n)(s,\cdot)\|_{H^\ell} \, \de s \\ 
&\lesssim e^{\theta t}\int_0^{t} e^{-\sigma(t-s)} \|e^{(1 - \sigma)(t-s)\Delta}\dive (m_n \otimes m_n)(s,\cdot)\|_{H^\ell} \, \de s \\
&\lesssim e^{(\theta - \sigma) t}\int_0^{t} e^{\sigma s}\left[(1-\sigma)(t-s)\right]^{-\frac12} \|(m_n \otimes m_n)(s,\cdot)\|_{H^\ell} \, \de s \\
&\leq (1-\sigma)^{-\frac12} \|m_n \|^2_{\Zeta} e^{(\theta - \sigma) t} \int_0^{t} e^{(\sigma - 2 \theta) s}\left(t-s\right)^{-\frac12} \, \de s \\
&\overset{\eqref{fdnjskfdjhbsjksdfh664}, \, \sigma \leq 2 \theta}{\lesssim} (1-\sigma)^{-\frac12} \varepsilon^2 e^{(\theta - \sigma) t} \int_0^{t} \left(t-s\right)^{-\frac12} \, \de s \\
&= (1-\sigma)^{-\frac12} \varepsilon^2 e^{(\theta - \sigma) t} \sqrt{t} \\
&\overset{\theta < \sigma}{\lesssim} (1-\sigma)^{-\frac12} (\sigma - \theta)^{-\frac12} \varepsilon^2,
\end{align*}
where we exploited the fact that for any $a > 0$, the function $t \mapsto e^{-at}\sqrt{t}$ achieves its maximum at $t = \frac{1}{2a}$, yielding a maximum value of $(2ae)^{-1/2}$.
In conclusion, taking the supremum over time yields
\begin{equation*}
\sup_{t \geq 0} e^{\theta t} \left\| \int_0^t e^{(t-s)\Delta} \p \dive (m_n \otimes m_n) (s,\cdot) \, \de s \right\|_{H^{\ell}} \overset{\theta=\frac14, \, \sigma = \frac12}{\leq} C \varepsilon^2.
\end{equation*}
The proof of \eqref{fdkjskldnfgnujkuu64} is identical, simply replacing $m_n$ with $v_n$.

We now prove the estimate \eqref{fdkjskldnfgnujkuu64N1}. For all $t \geq 0$, we have
\begin{equation*}
e^{\theta t} \left\| \int_0^t e^{(t-s)\Delta} \p \dive N_{1}(v_n, m_n) (s,\cdot) \, \de s \right\|_{H^{\ell}} \leq e^{\theta t} \int_0^{t} \|e^{(t-s)\Delta}\dive N_{1}(v_n, m_n)(s,\cdot)\|_{H^\ell} \, \de s.
\end{equation*}
Proceeding exactly as before, we exploit the splitting of the heat semigroup
\begin{align*}
e^{\theta t} \int_0^{t} &\|e^{(t-s)\Delta} \dive N_{1}(v_n, m_n) (s,\cdot)\|_{H^\ell} \, \de s \\ 
&\lesssim e^{\theta t}\int_0^{t} e^{-\sigma(t-s)} \|e^{(1 - \sigma)(t-s)\Delta}\dive N_{1}(v_n, m_n) (s,\cdot)\|_{H^\ell} \, \de s \\
&\lesssim e^{(\theta - \sigma) t}\int_0^{t} e^{\sigma s}\left((1-\sigma)(t-s)\right)^{-\frac12} \|N_{1}(v_n, m_n) (s,\cdot)\|_{H^\ell} \, \de s \\
&\overset{\eqref{fndjskdjshbf664tg3w4r}}{\lesssim} \varepsilon (1-\sigma)^{-\frac12} (\|m_n \|^2_{\Zeta} + \|v_n \|^2_{\Zeta}) e^{(\theta - \sigma) t}\int_0^{t} e^{(\sigma - 2 \theta) s}\left(t-s\right)^{-\frac12} \, \de s.
\end{align*}
Since $\|m_n \|_{\Zeta} + \|v_n \|_{\Zeta} \lesssim \varepsilon$, we can proceed precisely as in the previous bound to deduce
\begin{equation*}
\sup_{t \geq 0} e^{\theta t} \left\| \int_0^t e^{(t-s)\Delta} \p \dive N_{1}(v_n, m_n) (s,\cdot) \, \de s \right\|_{H^{\ell}} \overset{\theta=\frac14, \, \sigma = \frac12}{\leq} C \varepsilon^3.
\end{equation*}
This completes the proof.
\end{proof}
We are now ready to prove Theorem \ref{ThmSecondIterative}. 

\begin{proof}[\fbox{Proof of Theorem \ref{ThmSecondIterative}}]
We divide the proof in several steps.\\
\\
\underline{\em Step 1.} Divergence-free and zero-average properties. \\
\\
First, we prove that for all $n \in \N$ 
$$
\dive \vn = 0, \quad \int_{\T^3} \vn(t,x) \, \de x = 0.
$$
Recalling that the induction hypothesis assumes the same properties for $(v_n, m_n)$, and invoking the identity \eqref{OtimesVSnabla}, we can rewrite the differential equation \eqref{iterazione} as
\begin{equation}\label{iterazioneIn2d}
\partial_t\vn = \Delta\vn + \p\left( \dive N_1(v_n,m_n) + \dive(m_n \otimes m_n) - \dive(v_n \otimes v_n)\right) 
+ \pi_1 \mathcal{A} \begin{pmatrix} v_n \\ m_n \end{pmatrix} + \pi_1 \Gamma(v_n,m_n),
\end{equation}
supplemented with the initial condition $v_{n+1}(0,x) = y(x)$.

Taking the divergence of \eqref{iterazioneIn2d} and applying the algebraic properties from Proposition \ref{prop:algebraic_properties}, we obtain
$$ 
\partial_{t} \dive \vn = \Delta \dive \vn \qquad \forall t \geq 0.
$$ 
Since the initial data satisfies $\dive \vn(0,x) = \dive y(x) = 0$, the uniqueness of solutions to the heat equation yields
$$ 
\dive \vn(t,x) = 0 \qquad \forall t \geq 0.
$$
Similarly, integrating identity \eqref{iterazioneIn2d} over the torus $\T^3$, we find:
$$ 
\partial_{t} \int_{\T^3} \vn(t, x) \, \de x = 0 \qquad \forall t \geq 0.
$$
Since the initial velocity is zero-average, i.e.,
$$ 
\int_{\T^3} \vn(0, x) \, \de x = \int_{\T^3} y(x) \, \de x = 0,
$$ 
we immediately conclude that 
$$
\int_{\T^3} \vn(t, x) \, \de x = 0 \qquad \forall t \geq 0.
$$ 

\underline{\em Step 2.} Energy estimates for the gradient field \\
\\
We now aim to establish the uniform bound
$$
\|\vn\|_{\Zeta}+\|\nabla \vn\|_{L^2_tH^\ell} \leq 2 \|y\|_{H^\ell} + C \varepsilon^2.
$$
We start by showing that
\begin{equation}\label{fndjuksdughsugiohduNAblaV}
\|\nabla \vn\|^2_{L^2_tH^\ell} \leq \|y\|_{H^\ell}^2 + \frac{C \varepsilon^4}{\theta}.
\end{equation}
Taking the $H^\ell(\T^3)$ inner product of equation \eqref{iterazione} with $v_{n+1}$, integrating by parts, and exploiting the self-adjointness of the Leray projector $\mathbb{P}$ along with the property $\mathbb{P}v_{n+1} = v_{n+1}$, we arrive at
\begin{align*}
\frac12 \frac{\de}{\de t}\|\vn(t,\cdot)\|_{H^\ell}^2 + \|\nabla\vn(t,\cdot)\|_{H^\ell}^2 
&= \int_{\T^3} \left( \DL \left( \pi_1 \mathcal{A} \begin{pmatrix} v_n \\ m_n \end{pmatrix} + \pi_1 \Gamma(v_n,m_n) \right) \right) \cdot \DL \vn \, \de x
\\ 
 + \int_{\T^3}& \left( \DL \left(- m_n \otimes m_n + v_n \otimes v_n - N_1(v_n,m_n) \right) \right) \cdot \nabla \DL \vn \, \de x.
\end{align*}
Applying the Cauchy-Schwarz inequality, the algebra property of $H^{\ell}(\T^3)$, the Poincaré inequality, and the technical estimates \eqref{fdnjskngjds1}--\eqref{fdnjskngjds3}, we obtain:
\begin{equation}
\frac{\de}{\de t}\|\vn(t,\cdot)\|_{H^\ell}^2 + \|\nabla\vn(t,\cdot)\|_{H^\ell}^2 
\leq C\bigg( \|m_n(t,\cdot)\|_{H^\ell}^4 + \|v_n(t,\cdot)\|_{H^\ell}^4 + \varepsilon^4 (\|v_n(t,\cdot)\|_{H^\ell}^2 + \|m_n(t,\cdot)\|^2_{H^\ell}) \bigg).
\end{equation}
Integrating this inequality over the time interval $t \in [0, T]$ for $T \geq 0$ yields:
\begin{align*}
\|\vn(T,\cdot)\|_{H^\ell}^2 &+ \int_{0}^{T} \|\nabla\vn(t,\cdot)\|_{H^\ell}^2 \, \de t 
\\ 
\leq &\|y\|_{H^\ell}^2 + C\int_0^T \bigg( \|m_n(t,\cdot)\|_{H^\ell}^4 + \|v_n(t,\cdot)\|_{H^\ell}^4 + \varepsilon^4 (\|v_n(t,\cdot)\|_{H^\ell}^2 + \|m_n(t,\cdot)\|^2_{H^\ell}) \bigg) \, \de t
\\ 
\leq &\|y\|_{H^\ell}^2 + C\int_0^T \bigg( e^{-4 \theta t} \|m_n\|_{\Zeta}^4 + e^{-4 \theta t} \|v_n\|_{\Zeta}^4 + \varepsilon^4 (e^{-2 \theta t} \|v_n\|_{\Zeta}^2 + e^{-2 \theta t} \|m_n\|^2_{\Zeta}) \bigg) \, \de t
\\ 
\overset{\eqref{piccolezza spiegone}}{\leq}& \|y\|_{H^\ell}^2 + C \varepsilon^4 \int_0^T e^{-4 \theta t} \, \de t + C \varepsilon^6 \int_0^T e^{-2 \theta t} \, \de t 
\\
=& \|y\|_{H^\ell}^2 + \frac{C \varepsilon^4}{4 \theta} (1 - e^{-4 \theta T}) + \frac{C \varepsilon^6}{2 \theta} (1 - e^{-2 \theta T}).
\end{align*}
Taking the limsup as $T \to \infty$ and renaming the constants, we obtain
\begin{equation}
\left(\limsup_{T \to \infty} \|\vn(T,\cdot)\|_{H^\ell}^2\right) + \int_{0}^\infty \| \nabla \vn(t,\cdot) \|_{H^\ell}^2 \, \de t \leq \|y\|_{H^\ell}^2 + \frac{C \varepsilon^4}{\theta},
\end{equation}
which establishes the desired inequality \eqref{fndjuksdughsugiohduNAblaV}.\\
\\
\underline{\em Step 3.} Uniform bound for the $\Zeta$ norm. \\
\\
It remains to show that
\begin{equation}\label{fndjuksdughsugiohduZetaV2}
\| \vn\|_{\Zeta} \leq \|y\|_{H^\ell} + C \varepsilon^2.
\end{equation}
To this end, we exploit the integral formulation \eqref{iterazione integrale}, recalled below for convenience
\begin{align*}
\vn(t,\cdot) &= e^{t\Delta}y + \int_0^t e^{(t-s)\Delta} \p \left( (m_n \cdot \nabla) m_n - (v_n \cdot \nabla) v_n + \dive N_{1}(v_n, m_n) \right) (s,\cdot) \, \de s
\\ 
&\quad + \int_0^t e^{(t-s)\Delta} \left( \pi_1 \mathcal{A} \begin{pmatrix} v_n \\ m_n \end{pmatrix} + \pi_1 \Gamma(v_n,m_n) \right) (s,\cdot) \, \de s.
\end{align*}
We estimate the $\Zeta$ norm of the right-hand side term by term. Since $\int_{\T^3} y(x) \, \de x = 0$, 
applying~\eqref{fdjskhggnjds1} gives
\begin{equation}\label{ndjkskngdkskhjdf121}
\|e^{t\Delta} y\|_{\Zeta} = \sup_{t \geq 0} e^{\theta t} \| e^{t\Delta} y \|_{H^{\ell}} \leq \sup_{t \geq 0} e^{(\theta-1)t} \|y\|_{H^{\ell}} \overset{\theta \leq 1}{\leq} \|y\|_{H^{\ell}}.
\end{equation}
Next, invoking Lemma \ref{gmjdskhkjdsfgjkgsdnj9435782}, we also have that
\begin{equation}\label{ndjkskngdkskhjdf122}
\sup_{t \geq 0}e^{\theta t} 
\left\| \int_0^t e^{(t-s)\Delta} \p \left( (m_n \cdot \nabla) m_n - (v_n \cdot \nabla) v_n + \dive N_{1}(v_n, m_n) \right) (s,\cdot) \, \de s \right\|_{H^{\ell}} \leq C \varepsilon^2.
\end{equation}
For the linear operator term, using \eqref{fdjskhggnjds1} and \eqref{mathcalAbound}, we find that for all $t \geq 0$
\begin{align*}
e^{\theta t} \left\| \int_0^t e^{(t-s)\Delta} \pi_1 \mathcal{A} \begin{pmatrix} v_n \\ m_n \end{pmatrix} (s,\cdot) \, \de s \right\|_{H^{\ell}}
&\overset{\eqref{fdjskhggnjds1}}{\leq} e^{(\theta -1) t} \int_0^t e^{s} \left\| \mathcal{A} \begin{pmatrix} v_n \\ m_n \end{pmatrix} (s,\cdot) \right\|_{H^{\ell} \times H^{\ell}} \, \de s 
\\ 
&\overset{\eqref{mathcalAbound}}{\leq} e^{(\theta -1) t} \varepsilon^2 \int_0^t e^{s} (\| v_n(s,\cdot)\|_{H^{\ell}} + \| m_n(s,\cdot)\|_{H^{\ell}}) \, \de s 
\\ 
&\leq e^{(\theta -1) t} \varepsilon^2 (\| v_n\|_{\Zeta} + \| m_n \|_{\Zeta}) \int_0^t e^{(1-\theta) s} \, \de s
\\ 
&\overset{\eqref{piccolezza spiegone}}{\leq} e^{(\theta -1) t} \varepsilon^3 \int_0^t e^{(1-\theta) s} \, \de s \overset{\theta < 1}{\leq} \frac{\varepsilon^3}{1- \theta},
\end{align*}
which implies
\begin{equation}\label{ndjkskngdkskhjdf123}
\sup_{t \geq 0}e^{\theta t} \left\| \int_0^t e^{(t-s)\Delta} \pi_1 \mathcal{A} \begin{pmatrix} v_n \\ m_n \end{pmatrix} (s,\cdot) \, \de s \right\|_{H^{\ell}} \leq C \frac{\varepsilon^3}{1- \theta}.
\end{equation}
Finally, for the remainder term $\Gamma$, the bounds \eqref{fdjskhggnjds1} and \eqref{GammaEstimate1} yield for all $t \geq 0$
\begin{align*}
e^{\theta t} \left\| \int_0^t e^{(t-s)\Delta} \pi_1 \Gamma(v_n,m_n) (s,\cdot) \, \de s \right\|_{H^{\ell}}
&\overset{\eqref{fdjskhggnjds1}}{\leq} e^{(\theta -1) t} \int_0^t e^{s} \left\| \Gamma(v_n,m_n) (s,\cdot) \right\|_{H^{\ell} \times H^{\ell}} \, \de s 
\\ 
&\overset{\eqref{GammaEstimate1}}{\leq} e^{(\theta -1) t} \varepsilon \int_0^t e^{s} (\| v_n(s,\cdot)\|^2_{H^{\ell}} + \| m_n(s,\cdot)\|^2_{H^{\ell}}) \, \de s 
\\ 
&\leq e^{(\theta -1) t} \varepsilon (\| v_n\|_{\Zeta}^2 + \| m_n \|_{\Zeta}^2) \int_0^t e^{(1- 2\theta) s} \, \de s
\\ 
&\overset{\eqref{piccolezza spiegone}}{\leq} e^{(\theta -1) t} \varepsilon^3 \int_0^t e^{(1- 2 \theta) s} \, \de s \overset{\theta < \frac12}{\leq} \frac{\varepsilon^3}{1- 2 \theta} e^{- \theta t}.
\end{align*}
Taking the supremum over $t \geq 0$, we conclude
\begin{equation}\label{ndjkskngdkskhjdf123N}
\sup_{t \geq 0}e^{\theta t} \left\| \int_0^t e^{(t-s)\Delta} \pi_1 \Gamma(v_n,m_n) (s,\cdot) \, \de s \right\|_{H^{\ell}} \leq \sup_{t \geq 0} \frac{\varepsilon^3}{1-2 \theta} C e^{- \theta t} \overset{\theta > 0}{\leq} C \frac{\varepsilon^3}{1-2\theta}.
\end{equation}
Combining estimates \eqref{ndjkskngdkskhjdf121}, \eqref{ndjkskngdkskhjdf122}, \eqref{ndjkskngdkskhjdf123}, and \eqref{ndjkskngdkskhjdf123N} provides the bound \eqref{fndjuksdughsugiohduZetaV2}.
\end{proof}

We conclude this section with some technical lemmas that will be used in Section \ref{Sec:Convergence of the iterative scheme} to prove the convergence of the scheme. The proofs of these results follow the same arguments presented for Lemmas \ref{FTWOM1}--\ref{gmjdskhkjdsfgjkgsdnj9435782}, requiring only straightforward modifications.

\begin{lem}\label{FTWOM1IL}
Let $\ell > \frac72$ and let $B$ be a vector field such that $\|B\|_{H^{\ell +2}} \leq \varepsilon$. Let $F, G, v, m, v', m', v'' : \T^3 \to \R^3$. Then, the following estimate holds
\begin{equation}\label{fdnjskngjds1T0IL}
\left| \int_{\T^3} (F \otimes G)  v'' \, \de x \right| 
\leq \frac{1}{100} \| v'' \|^2_{L^2} + \min\left(\|F\|^2_{L^2} \|G\|^2_{L^{\infty}} , \, \|F\|^2_{L^\infty} \|G\|^2_{L^2}\right),
\end{equation}
and
\begin{equation}\label{fdnjskngjds1IL}
\left| \int_{\T^3} \pi_1 \mathcal{A} \begin{pmatrix} v \\ m \end{pmatrix} \cdot v'' \, \de x \right| 
\leq \frac{1}{100} \| v'' \|^2_{L^2} + C \varepsilon^4 (\|v\|^2_{L^2} + \|m \|^2_{L^2}).
\end{equation}
Assume further that the fields satisfy the smallness condition
\begin{equation}\label{fdjnsksdnusdf646527365rbdyfe}
\| v \|_{H^{\ell}} + \| v' \|_{H^{\ell}} + \| m \|_{H^{\ell}} + \| m' \|_{H^{\ell}} \lesssim \varepsilon.
\end{equation}
Then, we have the following estimates for the differences
\begin{equation}\label{fdnjskngjds2IL}
\left| \int_{\T^3} \pi_1 \left(\Gamma(v,m) - \Gamma(v',m')\right) \cdot v'' \, \de x \right| 
\leq \frac{1}{100} \| v'' \|^2_{L^2} + C \varepsilon^4 (\|v - v'\|^2_{L^2} + \|m - m'\|^2_{L^2}),
\end{equation}
and
\begin{equation}\label{fdnjskngjds3IL}
\left| \int_{\T^3} \left( N_1(v,m) - N_1(v',m') \right) \cdot \nabla v'' \, \de x \right| 
\leq \frac{1}{100} \| \nabla v'' \|^2_{L^2} + C \varepsilon^4 (\|v - v'\|^2_{L^2} + \|m - m'\|^2_{L^2}).
\end{equation}
\end{lem}

\begin{proof}
Throughout the proof, $C>0$ will denote several constants, possibly increasing from line to line.
First of all, the inequality \eqref{fdnjskngjds1T0IL} follows immediately by applying the Hölder and Cauchy-Schwarz inequalities. To prove \eqref{fdnjskngjds1IL}, we use \eqref{fdhbsjdhghgsdsh7473462Stronger} and we get
\begin{align*}
\left| \int_{\T^3} \pi_1 \mathcal{A} \begin{pmatrix} v \\ m \end{pmatrix} \cdot v'' \, \de x \right|
&\leq \left\| \pi_1 \mathcal{A} \begin{pmatrix} v \\ m \end{pmatrix} \right\|_{L^2} \|v'' \|_{L^2} \\
&\leq \left\| \mathcal{A} \begin{pmatrix} v \\ m \end{pmatrix} \right\|_{L^2 \times L^2} \| v'' \|_{L^2} \\
&\leq C \varepsilon^2 (\|v \|_{L^2} + \|m\|_{L^2}) \| v'' \|_{L^2} \\
&\leq \frac{1}{100} \| v'' \|^2_{L^2} + C \varepsilon^4 (\|v\|^2_{L^2} + \|m \|^2_{L^2}),
\end{align*}
where we used Young's inequality in the last step. 

We now prove \eqref{fdnjskngjds2IL}: we use \eqref{GammaEstimate1L2} together with the assumption \eqref{fdjnsksdnusdf646527365rbdyfe} and we get
\begin{align*}
\left| \int_{\T^3} \pi_1 \left( \Gamma(v,m) - \Gamma(v',m') \right) \cdot v'' \, \de x \right|
&\leq \left\| \pi_1 \left( \Gamma(v,m) - \Gamma(v',m') \right) \right\|_{L^2} \| v'' \|_{L^2} \\
&\leq \left\| \Gamma(v,m) - \Gamma(v',m') \right\|_{L^2 \times L^2} \| v'' \|_{L^2} \\
&\leq C \varepsilon^2 (\|v - v'\|_{L^2} + \|m - m'\|_{L^2}) \| v'' \|_{L^2} \\
&\leq \frac{1}{100} \| v'' \|^2_{L^2} + C \varepsilon^4 (\|v - v'\|^2_{L^2} + \|m - m'\|^2_{L^2}).
\end{align*}

Lastly, the proof of \eqref{fdnjskngjds3IL} follows from \eqref{fndjskdjshbf664tg3w4rL2} and hypothesis \eqref{fdjnsksdnusdf646527365rbdyfe}, indeed we have
\begin{align*}
\left| \int_{\T^3} \left( N_1(v,m) - N_1(v',m') \right) \cdot \nabla v'' \, \de x \right|
&\leq \left\| N_1(v,m) - N_1(v',m') \right\|_{L^2} \|\nabla v'' \|_{L^2} \\
&\leq C \varepsilon^2 (\|v - v'\|_{L^2} + \|m - m'\|_{L^2}) \| \nabla v'' \|_{L^2} \\
&\leq \frac{1}{100} \| \nabla v'' \|^2_{L^2} + C \varepsilon^4 (\|v - v'\|^2_{L^2} + \|m - m'\|^2_{L^2}).
\end{align*}
This completes the proof of the lemma.
\end{proof}

\begin{lem}\label{gmjdskhkjdsfgjkgsdnj9435782543}
Let $\ell > \frac72$, $\varepsilon \in (0,1)$ and $\theta = \frac14$. Let $B$ be a vector field such that $\|B\|_{H^{\ell+1}} \leq \varepsilon$. Let $F, G, v, v', m, m' : [0,\infty)\times \T^3 \to \R^3$ be divergence-free zero-average vector fields satisfying
\begin{equation}\label{fdjnsksdnusdf646527365rbdyfeBis}
\| v \|_{\ZetaL} + \| v' \|_{\ZetaL} + \| m \|_{\ZetaL} + \| m' \|_{\ZetaL} + \| F \|_{\ZetaL} + \| G \|_{\ZetaL} \lesssim \varepsilon.
\end{equation}
Then, the following bilinear estimate holds
\begin{equation}\label{fdkjskldnfgnujkuu65333543}
\sup_{t \geq 0}e^{\theta t} \left\| \int_0^t e^{(t-s)\Delta} \p (F \cdot \nabla) G (s,\cdot) \, \de s \right\|_{L^2} 
\lesssim \varepsilon \min \left(\|F\|_{\ZetaZero} , \|G\|_{\ZetaZero} \right),
\end{equation}
as well as the difference estimates:
\begin{equation}\label{fdkjskldnfgnujkuu64N1333543}
\sup_{t \geq 0}e^{\theta t} 
\left\| \int_0^t e^{(t-s)\Delta} \p \dive \left(N_{1}(v, m) - N_{1}(v', m')\right) (s,\cdot) \, \de s \right\|_{L^2} 
\lesssim C \varepsilon^2 (\|m-m' \|_{\ZetaZero} + \|v - v'\|_{\ZetaZero}),
\end{equation}
and
\begin{equation}\label{12344321}
\sup_{t \geq 0}e^{\theta t} 
\left\| \int_0^t e^{(t-s)\Delta} \pi_1 \left(\Gamma(v, m) - \Gamma(v', m') \right) (s,\cdot) \, \de s \right\|_{L^2} 
\lesssim \varepsilon^2 (\| m - m' \|_{\ZetaZero} + \| v - v'\|_{\ZetaZero} ).
\end{equation}
\end{lem}

\begin{proof}
We divide the proof in several steps.\\
\\
\underline{\em Step 1.} Proof of the estimate \eqref{fdkjskldnfgnujkuu65333543}. \\
\\
Using the divergence-free condition $\dive F = 0$ along with identity \eqref{OtimesVSnabla}, we rewrite the integral as:
$$
\int_0^t e^{(t-s)\Delta} \p (F \cdot \nabla) G (s) \, \de s = \int_0^t e^{(t-s)\Delta} \p \dive (G \otimes F) (s) \, \de s.
$$
Then, taking the $L^2$ norm, we bound for all $t \geq 0$:
\begin{align}\label{fmgdjksldngjk1574y839}
e^{\theta t} \left\| \int_0^t e^{(t-s)\Delta} \p \dive (G \otimes F) (s,\cdot) \, \de s \right\|_{L^2}
&\leq e^{\theta t} \int_0^{t} \|e^{(t-s)\Delta} \dive (G \otimes F)(s,\cdot)\|_{L^2} \, \de s.
\end{align}
Let $\sigma \in (0, 1)$ such that $\sigma \in (\theta , 2 \theta]$. Note that $(0,1) \cap (\theta , 2 \theta] \neq \emptyset$ for all $\theta \in (0,1)$; however, since we have fixed $\theta = \frac14$, we can simply choose $\sigma = \frac12$. It will nonetheless be instructive to keep track of the explicit dependence on $\theta$ and $\sigma$ throughout the computations.

We split the heat kernel as $e^{(t-s)\Delta} = e^{\sigma(t-s)\Delta} e^{(1 - \sigma)(t-s) \Delta}$. Invoking Lemma \ref{lemma calore} and the Sobolev embedding product inequality (valid for any $\delta > 0$)
\begin{equation}\label{fduhknggnjdkw88472}
\|G \otimes F\|_{L^2} \lesssim \min\left( \|G\|_{L^2} \|F\|_{H^{\frac32 + \delta}} , \|F\|_{L^2} \|G\|_{H^{\frac32 + \delta}} \right)
\leq (\|F\|_{H^{\frac32 + \delta}} + \|G\|_{H^{\frac32 + \delta}}) \min\left( \|G\|_{L^2} , \|F\|_{L^2} \right),
\end{equation}
we obtain
\begin{align*}
e^{\theta t} &\int_0^{t}\|e^{(t-s)\Delta}\dive (G \otimes F)(s,\cdot)\|_{L^2} \, \de s
\\ 
&\le C e^{\theta t}\int_0^{t} e^{-\sigma(t-s)}\|e^{(1 - \sigma)(t-s)\Delta}\dive (G \otimes F)(s,\cdot)\|_{L^2} \, \de s 
\\
&\le C e^{(\theta - \sigma) t}\int_0^{t} e^{\sigma s}\left[(1-\sigma)(t-s)\right]^{-\frac12}\|(G \otimes F)(s,\cdot)\|_{L^2} \, \de s 
\\
&\overset{\eqref{fduhknggnjdkw88472}}{\le} C e^{(\theta - \sigma) t}\int_0^{t} e^{\sigma s}\left[(1-\sigma)(t-s)\right]^{-\frac12} \min\left( \|G(s,\cdot)\|_{L^2} , \|F(s,\cdot)\|_{L^2} \right) (\|F(s,\cdot)\|_{H^{\frac32 + \delta}} + \|G(s,\cdot)\|_{H^{\frac32 + \delta}}) \, \de s 
\\
&\overset{\eqref{fdjnsksdnusdf646527365rbdyfeBis}}{\le} C (1-\sigma)^{-\frac12} \varepsilon \min\left( \|G\|_{\ZetaZero} , \|F\|_{\ZetaZero} \right) e^{(\theta - \sigma) t}\int_0^{t} e^{(\sigma - 2 \theta) s}\left(t-s\right)^{-\frac12} \, \de s
\\
&\overset{\sigma \leq 2 \theta}{\le} C (1-\sigma)^{-\frac12} \varepsilon \min\left( \|G\|_{\ZetaZero} , \|F\|_{\ZetaZero} \right) e^{(\theta - \sigma) t}\int_0^{t} \left(t-s\right)^{-\frac12} \, \de s
\\
&= C (1-\sigma)^{-\frac12} \varepsilon \min\left( \|G\|_{\ZetaZero} , \|F\|_{\ZetaZero} \right) e^{-(\sigma - \theta) t} \sqrt{t}.
\end{align*}
In conclusion, exploiting the boundedness of the function $t\mapsto e^{-at}\sqrt{t}$ with $a>0$ as we did above, taking the supremum over $t \geq 0$ and substituting $\theta=\frac14, \sigma = \frac12$, we deduce
\begin{equation*}
\sup_{t \geq 0} e^{\theta t} \left\| \int_0^t e^{(t-s)\Delta} \p \dive (G \otimes F) (s) \, \de s \right\|_{L^2} \leq C \varepsilon \min\left( \|G\|_{\ZetaZero} , \|F\|_{\ZetaZero} \right).
\end{equation*}

\vspace{0.3cm}
\underline{\em Step 2.} Proof of the difference estimate \eqref{fdkjskldnfgnujkuu64N1333543}.\\
\\
For the non-linear remainder term, we have for all $t \geq 0$
\begin{align*}
e^{\theta t} &\left\| \int_0^t e^{(t-s)\Delta} \p \dive \left(N_{1}(v, m) - N_{1}(v', m')\right) (s,\cdot) \, \de s \right\|_{L^2}
\\ 
&\leq e^{\theta t} \int_0^{t}\|e^{(t-s)\Delta}\dive (N_{1}(v, m) - N_{1}(v', m')) (s,\cdot)\|_{L^2} \, \de s.
\end{align*}
Let us define $\Omega(s) := \| v \|_{H^\ell} + \| v' \|_{H^\ell} + \| m \|_{H^\ell} + \| m' \|_{H^\ell}$. Applying Lemma \ref{lemma calore} together with the heat kernel splitting, we obtain:
\begin{align*}
e^{\theta t} &\int_0^{t}\|e^{(t-s)\Delta} \dive (N_{1}(v, m) - N_{1}(v', m')) (s)\|_{L^2} \, \de s
\\ 
&\le C e^{\theta t}\int_0^{t} e^{-\sigma(t-s)}\|e^{(1 - \sigma)(t-s)\Delta}\dive (N_{1}(v, m) - N_{1}(v', m')) (s,\cdot)\|_{L^2} \, \de s 
\\
&\le C e^{(\theta - \sigma) t}\int_0^{t} e^{\sigma s}\left[(1-\sigma)(t-s)\right]^{-\frac12}\|(N_{1}(v, m) - N_{1}(v', m')) (s,\cdot)\|_{L^2} \, \de s 
\\
&\overset{\eqref{fndjskdjshbf664tg3w4rL2}}{\le} C \varepsilon (1-\sigma)^{-\frac12} e^{(\theta - \sigma) t}\int_0^{t} e^{\sigma s}\left(t-s\right)^{-\frac12} (\|(m-m')(s) \|_{L^2} + \|(v - v')(s) \|_{L^2}) \Omega(s) \, \de s
\\
&\le C \varepsilon^2 (1-\sigma)^{-\frac12} (\|m-m' \|_{\ZetaZero} + \|v - v'\|_{\ZetaZero}) e^{(\theta - \sigma) t}\int_0^{t} e^{(\sigma - 2 \theta) s}\left(t-s\right)^{-\frac12} \, \de s.
\end{align*}
Proceeding exactly as in the previous chain of inequalities to bound the time integral, we arrive at the desired conclusion
\begin{equation*}
\sup_{t \geq 0} e^{\theta t} \left\| \int_0^t e^{(t-s)\Delta} \p \dive (N_{1}(v, m) - N_{1}(v', m')) (s,\cdot) \, \de s \right\|_{L^2} \leq C \varepsilon^2 (\|m-m' \|_{\ZetaZero} + \|v - v'\|_{\ZetaZero}).
\end{equation*}

\vspace{0.3cm}
\underline{\em Step 3.} Proof of the difference estimate \eqref{12344321}.\\
\\
From bounds \eqref{fdjskhggnjds1} and \eqref{GammaEstimate1L2}, it follows that for all $t \geq 0$:
\begin{align*}
e^{\theta t} &\left\| \int_0^t e^{(t-s)\Delta} \pi_1 (\Gamma(v, m) - \Gamma(v', m')) (s,\cdot) \, \de s \right\|_{L^2}
\\
&\overset{\eqref{fdjskhggnjds1}}{\leq} e^{(\theta -1) t} \int_0^t e^{s} \left\| (\Gamma(v, m) - \Gamma(v', m')) (s,\cdot) \right\|_{L^2 \times L^2} \, \de s 
\\ 
&\overset{\eqref{GammaEstimate1L2}}{\leq} C e^{(\theta -1) t} \varepsilon \int_0^t e^{s} (\|(m-m')(s) \|_{L^2} + \|(v - v')(s,\cdot) \|_{L^2}) \Omega(s) \, \de s 
\\ 
&\overset{\eqref{fdjnsksdnusdf646527365rbdyfeBis}}{\le} C e^{(\theta -1) t} \varepsilon^2 (\| v - v'\|_{\ZetaZero} + \| m - m' \|_{\ZetaZero}) \int_0^t e^{(1- 2\theta) s} \, \de s
\\ 
&\overset{\theta < \frac12}{\leq} \frac{C \varepsilon^2}{1- 2 \theta} (\| v - v'\|_{\ZetaZero} + \| m - m' \|_{\ZetaZero}) e^{- \theta t}.
\end{align*}
Taking the supremum over $t \geq 0$ and recalling that $\theta = \frac{1}{4} > 0$, we conclude
\begin{align*}
\sup_{t \geq 0}e^{\theta t} \left\| \int_0^t e^{(t-s)\Delta} \pi_1 (\Gamma(v, m) - \Gamma(v', m') ) (s,\cdot) \, \de s \right\|_{L^2} 
&\leq \sup_{t \geq 0} \frac{C \varepsilon^2}{1- 2 \theta} (\| v - v'\|_{\ZetaZero} + \| m - m' \|_{\ZetaZero}) e^{- \theta t}
\\
&\leq \frac{C \varepsilon^2}{1- 2 \theta} (\| v - v'\|_{\ZetaZero} + \| m - m' \|_{\ZetaZero}),
\end{align*}
which completes the proof of the lemma.
\end{proof}

\section{Convergence of the iterative scheme}\label{Sec:Convergence of the iterative scheme}

In this section, we establish the convergence of the sequence $(v_n, m_n)$ defined by the iterative scheme \eqref{iterazione integrale}. We start by providing a uniform bound for the sequence.

\begin{prop}\label{ThmSecondIterativeMAIN}
Let $\ell \geq 4$ and $\theta = \frac14$. There exists $\varepsilon_1 \in (0,1)$ such that for all $\varepsilon \in (0, \varepsilon_1)$, the following holds. Let $B$ be such that $\|B\|_{H^{\ell +2}} \leq \varepsilon$, and let $y : \T^3 \to \R^3$ be a zero-average, divergence-free vector field satisfying
\begin{equation}\label{yNorm1}
\|y\|_{H^{\ell}} \leq \frac{\varepsilon}{3}.
\end{equation}
Then, the sequence $(v_n, m_n)_{n \in \N}$ defined in \eqref{iterazione integrale} satisfies for all $n \in \N$
\begin{equation}\label{InductionN}
\dive v_n = 0, \quad \dive m_n = 0, \quad \int_{\T^3} v_n \, \de x = 0, \quad \int_{\T^3} m_n \, \de x = 0,
\end{equation}
and
\begin{equation}\label{InductionNBis}
\| v_n \|_{\Zeta} + \|\nabla v_n\|_{L^2_t H^\ell} + \| m_n\|_{\Zeta} \leq \varepsilon.
\end{equation}
\end{prop}

\begin{proof}
We proceed by induction on $n$. The base case $n=0$ is trivial as $(v_0, m_0) = (0,0)$. Assuming that \eqref{InductionN} and \eqref{InductionNBis} hold for some $n \in \N$, we verify the inductive step. By Theorem \ref{ThMFirstRecursiveBound}, $m_{n+1}$ satisfies $\dive m_{n+1} = 0$, $\int_{\T^3} m_{n+1} \, \de x = 0$, and
\begin{equation*}
\|m_{n+1}\|_{\Zeta} \leq C \varepsilon^2.
\end{equation*}
On the other hand, by Theorem \ref{ThmSecondIterative}, we have $\dive v_{n+1} = 0$, $\int_{\T^3} v_{n+1} \, \de x = 0$, and
\begin{equation*}
\|v_{n+1}\|_{\Zeta} + \|\nabla v_{n+1}\|_{L^2_t H^\ell} \leq 2\|y\|_{H^\ell} + C \varepsilon^2.
\end{equation*}
Thus, choosing $\varepsilon_1 > 0$ sufficiently small, we satisfy the bound \eqref{InductionNBis} for $n+1$.
\end{proof}

Relying on the uniform bound \eqref{InductionNBis} and the technical estimates developed in Sections \ref{Sec:Picard-1} through \ref{Sec:Picard}, it is a standard (though lengthy) procedure to prove that the sequence $(v_n, m_n)$ is Cauchy in the $\ZetaB$ norm for all $\ell' < \ell$. This is the content of the following theorem, the proof of which will occupy the remainder of this section.

\begin{thm}\label{ThmSecondIterativeMAIN2}
Under the assumptions of Proposition \ref{ThmSecondIterativeMAIN}, there exists a pair of zero-average, divergence-free vector fields $(v, m): [0, +\infty) \times \T^3 \to \R^3$ such that
\begin{equation}\label{ConvergenceReal}
\|v_n - v\|_{Z^{\ell',\theta}} + \| \nabla v_n - \nabla v \|_{L^2_t H^{\ell'}} + \|m_n - m\|_{Z^{\ell',\theta}} \overset{n \to \infty}{\longrightarrow} 0,
\end{equation}
for all $\ell' < \ell$. Moreover, the following estimate holds:
\begin{equation}\label{InductionNBisbdfshkbjgks}
\| v \|_{\Zeta} + \| \nabla v \|_{L^2_t H^{\ell}} + \| m\|_{\Zeta} \leq \varepsilon.
\end{equation}
\end{thm}

\begin{proof}
Once \eqref{ConvergenceReal} is proven, the fact that $v$ and $m$ are zero-average and divergence-free follows directly from the properties of $v_n$ and $m_n$ for all $n \in \N$. 

To prove \eqref{ConvergenceReal}, by the uniform bound \eqref{InductionNBis} and standard interpolation arguments, it suffices to show that there exists $(v, m)$ such that
\begin{equation}\label{Convergence0}
\|v_n - v\|_{\ZetaZero} + \| \nabla v_n - \nabla v \|_{L^2_t L^2} + \|m_n - m\|_{\ZetaZero} \overset{n \to \infty}{\longrightarrow} 0.
\end{equation}

To this end, we define the quantity
$$
\delta_n := \|\vn - v_n\|_{\ZetaZero} + \| \nabla \vn - \nabla v_n \|_{L^2_t L^2} + \| \mn - m_n\|_{\ZetaZero},
$$
and we will show that
\begin{equation}\label{ImpliesConvergence0}
\delta_{n} \leq D^{n+1} \varepsilon^{n+1}, \qquad \forall n \in \N,
\end{equation}
for some constant $D \geq 1$. This implies \eqref{Convergence0} after choosing $\varepsilon$ sufficiently small and summing the telescopic geometric series, establishing that $(v_n, m_n)$ is a Cauchy sequence. We proceed by induction.\\
\\
\textbf{\em Base case.} The base case $n=0$ is almost immediate since
$$
(v_1 - v_0, m_1 - m_0) = (e^{t \Delta} y, 0).
$$
Thus, $\delta_0 \leq D \varepsilon$ follows from \eqref{fdjskhggnjds1} and the zero-mean property of $y$. Indeed:
$$
\| e^{t \Delta} y \|_{\ZetaZero} = \sup_{t \geq 0} e^{\theta t} \|e^{t \Delta} y\|_{L^2} \leq \sup_{t \geq 0} e^{(\theta - 1) t} \| y\|_{L^2} \overset{\theta \leq 1}{=} \| y\|_{L^2} \overset{\eqref{yNorm1}}{\leq} \frac{\varepsilon}{2},
$$
along with the standard energy estimate for the heat equation
$$
\| \nabla e^{t \Delta} y \|_{L^2_t L^2} \leq \frac{1}{\sqrt{2}} \|y\|_{L^2} \overset{\eqref{yNorm1}}{\leq} \frac{\varepsilon}{2}.
$$
\textbf{\em Inductive step.} We assume that for some $n \in \N^*$, the following inductive hypothesis holds
\begin{equation}\label{IH}
\delta_{n-1} \leq D^n \varepsilon^n.
\end{equation}
We aim to prove that
\begin{equation}\label{WWDfhruehr}
\delta_{n} \leq D^{n+1} \varepsilon^{n+1}.
\end{equation}
This will be deduced from the following three bounds
\begin{align}
\label{VeryFinal1BoundIH} \| \mn - m_n\|_{\ZetaZero} &\leq \frac{1}{3} D^{n+1} \varepsilon^{n+1}, \\
\label{VeryFinal3BoundIH} \| \nabla \vn - \nabla v_n \|_{L^2_t L^2} &\leq \frac{1}{3} D^{n+1} \varepsilon^{n+1}, \\
\label{VeryFinal2BoundIH} \|\vn - v_n\|_{\ZetaZero} &\leq \frac{1}{3} D^{n+1} \varepsilon^{n+1}.
\end{align}
\noindent \textbf{Proof of \eqref{VeryFinal1BoundIH}.} From \eqref{iterazione integrale} and \eqref{|def:Vn}, we deduce
\begin{align} \nonumber
& \mn(t,x) - m_n(t,x) \\ \label{Uff2}
= & -\int_t^\infty \left( \left((\mn \cdot \nabla) v_n \right)(s,\Ph_{t,s}(x)) - \left((m_n \cdot \nabla) v_{n-1} \right)(s,\PhB_{t,s}(x)) \right) \, \de s \\ \label{Uff2.5}
& -\int_t^\infty \bigg( \left((\mn \cdot \nabla) \pi_1 (e^{K} - \Id) \begin{pmatrix} v_n \\ m_n \end{pmatrix} \right)(s,\Ph_{t,s}(x)) \\ \nonumber
& \qquad \qquad - \left((m_n \cdot \nabla) \pi_1 (e^{K} - \Id) \begin{pmatrix} v_{n-1} \\ m_{n-1} \end{pmatrix} \right)(s,\PhB_{t,s}(x)) \bigg) \, \de s \\ \label{Uff3}
& -\int_t^\infty \left( \left( \pi_2 \mathcal{A} \begin{pmatrix} v_n \\ m_n \end{pmatrix} \right)(s,\Ph_{t,s}(x)) - \left( \pi_2 \mathcal{A} \begin{pmatrix} v_{n-1} \\ m_{n-1} \end{pmatrix} \right)(s,\PhB_{t,s}(x)) \right) \, \de s \\ \label{Uff4}
& -\int_t^\infty \left( \left( \pi_2 \Gamma(v_n,m_n) \right)(s,\Ph_{t,s}(x)) - \left( \pi_2 \Gamma(v_{n-1},m_{n-1}) \right)(s,\PhB_{t,s}(x)) \right) \, \de s \\ \label{Uff5}
& -\int_t^\infty \left( \left( \dive N_2(v_n,m_n) \right)(s,\Ph_{t,s}(x)) - \left( \dive N_2(v_{n-1},m_{n-1}) \right)(s,\PhB_{t,s}(x)) \right) \, \de s.
\end{align}
We will estimate the $\ZetaZero$ norm of the right-hand side term by term.\\
\\
\textbf{Bound of \eqref{Uff2}.}
We estimate
\begin{align*}
\|\eqref{Uff2}\|_{\ZetaZero} & \lesssim 
\sup_{t \geq 0}e^{\theta t} \left\| \int_t^\infty  \left(((\mn - m_n)\cdot\nabla) v_n \right)(s,\Ph_{t,s})  \, \de s \right\|_{L^2} \\
& +  
\sup_{t \geq 0}e^{\theta t} \left\| \int_t^\infty  \left(( m_n \cdot\nabla) (v_n - v_{n-1}) \right)(s,\Ph_{t,s})  \, \de s \right\|_{L^2} \\
& +  
\sup_{t \geq 0}e^{\theta t} \left\| \int_t^\infty \left( \left((m_{n}\cdot\nabla) v_{n-1} \right)(s,\Ph_{t,s}) -  \left((m_{n}\cdot\nabla) v_{n-1} \right)(s,\PhB_{t,s}) \right) \, \de s \right\|_{L^2}.
\end{align*}
From Lemma \ref{Bunch1} and \eqref{InductionNBis} we get 
$$
\sup_{t \geq 0}e^{\theta t} \left\| \int_t^\infty  \left(((\mn - m_n)\cdot\nabla) v_n \right)(s,\Ph_{t,s})  \, \de s \right\|_{L^2} \leq C \varepsilon \|\mn - m_n\|_{\ZetaZero},
$$ 
as well as
$$
\sup_{t \geq 0}e^{\theta t} \left\| \int_t^\infty  \left(( m_n \cdot\nabla) (v_n - v_{n-1}) \right)(s,\Ph_{t,s})  \, \de s \right\|_{L^2} \leq C \varepsilon \delta_{n-1} \overset{\eqref{IH}}{\leq} C D^n \varepsilon^{n+1}.
$$ 
In order to handle the $\ZetaZero$ norm of the third term, we recall 
$$
(m_{n}\cdot\nabla) v_{n-1} := \dive ( v_{n-1} \otimes m_{n} ),
$$ 
so that 
\begin{align*}
\left| \left((m_{n}\cdot\nabla) v_{n-1} \right)\right.&\left.(s,\Ph_{t,s}(x)) -  \left((m_{n}\cdot\nabla) v_{n-1} \right)(s,\PhB_{t,s}(x)) \right| \\  
&\leq \| \nabla^2 m_{n}(s,\cdot) \|_{L^{\infty}} \| \nabla^2 v_{n-1}(s,\cdot)\|_{L^{\infty}}  \left| \Ph_{t,s}(x) - \PhB_{t,s}(x) \right| \\ 
&\leq \| m_{n}(s,\cdot) \|_{H^{\ell}} \|v_{n-1}(s,\cdot)\|_{H^{\ell}}  \left| \Ph_{t,s}(x) - \PhB_{t,s}(x) \right|.
\end{align*}
Then, using Lemma \ref{stime_flussoBisStability}
\begin{align}\label{PAI1} 
& \left\| \int_t^\infty \left( \left((m_{n}\cdot\nabla) v_{n-1} \right)(s,\Ph_{t,s}) -  \left((m_{n}\cdot\nabla) v_{n-1} \right)(s,\PhB_{t,s}) \right) \, \de s \right\|_{L^2} \\ \nonumber
& \leq \int_t^\infty \| m_{n}(s,\cdot) \|_{H^{\ell}} \|v_{n-1}(s,\cdot)\|_{H^{\ell}}  \| \Ph_{t,s} - \PhB_{t,s} \|_{L^2} \, \de s \\ \nonumber
& \leq \| v_n - v_{n-1}\|_{\ZetaZero} \|m_{n}\|_{\Zeta} \|v_{n-1}\|_{\Zeta} \int_t^\infty e^{- 2 \theta s}  \, \de s \overset{\eqref{IH}}{\lesssim} D^n \frac{\varepsilon^{n+2}}{2\theta} e^{-2 \theta t},
\end{align}
which implies
$$
\sup_{t \geq 0}e^{\theta t}  \left\| \int_t^\infty \left( \left((m_{n}\cdot\nabla) v_{n-1} \right)(s,\Ph_{t,s}) -  \left((m_{n}\cdot\nabla) v_{n-1} \right)(s,\PhB_{t,s}) \right) \, \de s\right\|_{L^2} \leq C D^n \frac{\varepsilon^{n+2}}{2\theta}.
$$ 
In conclusion
\begin{equation}\label{Absorb1}
\|\eqref{Uff2}\|_{\ZetaZero} \leq C \varepsilon \|\mn - m_n\|_{\ZetaZero} + C D^n \varepsilon^{n+1}.
\end{equation}

\textbf{Bound of \eqref{Uff2.5}.}
Proceeding analogously, we estimate
\begin{align*}
\|\eqref{Uff2.5}\|_{\ZetaZero} &
 \lesssim 
 \sup_{t \geq 0}e^{\theta t} \bigg\| \int_t^\infty  \left(((\mn - m_n)\cdot\nabla) \pi_1 (e^{K} - \Id)  \left(\begin{smallmatrix} v_n \\ m_n \end{smallmatrix}\right) \right)(s,\Ph_{t,s})  \, \de s \bigg\|_{L^2} \\
& + \sup_{t \geq 0}e^{\theta t} \bigg\| \int_t^\infty  \left(( m_n \cdot \nabla) \pi_1 (e^{K} - \Id)  \left(\begin{smallmatrix} v_n - v_{n-1} \\ m_n - m_{n-1}\end{smallmatrix}\right) \right)(s,\Ph_{t,s})  \, \de s \bigg\|_{L^2} \\
& + \sup_{t \geq 0}e^{\theta t} \bigg\| \int_t^\infty \bigg( ((m_{n}\cdot\nabla) \pi_1 (e^{K} - \Id)  \begin{pmatrix} v_{n-1} \\  m_{n-1}\end{pmatrix} )(s,\Ph_{t,s})
\\
&
\qquad \qquad \qquad \qquad  - ((m_{n}\cdot\nabla) \pi_1 (e^{K} - \Id)  \begin{pmatrix} v_{n-1} \\  m_{n-1}\end{pmatrix} )(s,\PhB_{t,s}) \bigg) \, \de s \bigg\|_{L^2}.
\end{align*}
Given any couple of divergence-free vector fields $(F, G) : [0, +\infty) \times \T^3 \to \R^3$, as 
consequence of inequality \eqref{fdnjskbjnfhjskhjdbfgsjk3L2} we have for all $t \geq 0$:
$$
\left\| \pi_1 (e^{K} - \Id)  \begin{pmatrix} F(t,\cdot) \\ G (t,\cdot)\end{pmatrix} \right\|_{L^2}
\lesssim \varepsilon (\|F(t,\cdot)\|_{H^{-1}} + \|G (t,\cdot)\|_{H^{-1}}),
$$
from which we deduce 
\begin{equation}\label{fdnjskkjfdh64632guyfew27894}
\left\| \pi_1 (e^{K} - \Id)  \begin{pmatrix} F  \\ G \end{pmatrix} \right\|_{\ZetaZero}
\lesssim \varepsilon (\| F \|_{\ZetaZero} + \|G \|_{\ZetaZero}) \overset{\eqref{InductionNBis}}{\lesssim} \varepsilon^2.
\end{equation}
Thus, we can invoke Lemma \ref{Bunch1} and \eqref{InductionNBis} to get 
$$
\sup_{t \geq 0}e^{\theta t} \left\| \int_t^\infty  \left(((\mn - m_n)\cdot\nabla) \pi_1 (e^{K} - \Id)  \begin{pmatrix} v_n  \\ m_n \end{pmatrix} \right)(s,\Ph_{t,s})  \, \de s \right\|_{L^2}
 \leq C \varepsilon^2 \|\mn - m_n\|_{\ZetaZero}  ,
$$ 
as well as
$$
\sup_{t \geq 0}e^{\theta t} \left\| \int_t^\infty  \left(( m_n \cdot\nabla) \pi_1 (e^{K} - \Id) \begin{pmatrix} v_n - v_{n-1} \\ m_n - m_{n-1}\end{pmatrix} \right)(s,\Ph_{t,s})  \, \de s \right\|_{L^2}
 \leq C \varepsilon^2 \delta_{n-1} 
\overset{\eqref{IH}}{\leq}    C D^n \varepsilon^{n+2}  .
$$ 
In order to handle the $\ZetaZero$ norm of the third term on the right hand side,  we recall 
$$
(m_{n}\cdot\nabla) \pi_1 (e^{K} - \Id)  \begin{pmatrix} v_{n-1} \\  m_{n-1}\end{pmatrix} 
:= \dive \left(\pi_1 (e^{K} - \Id)  \begin{pmatrix}  v_{n-1} \\  m_{n-1}\end{pmatrix} \otimes m_{n}\right),
$$
so that 
\begin{align*}
& \left| \left((m_{n}\cdot\nabla) \pi_1 (e^{K} - \Id)  \begin{pmatrix} v_{n-1} \\  m_{n-1}\end{pmatrix} \right)(s,\Ph_{t,s}(x)) 
-  \left((m_{n}\cdot\nabla) \pi_1 (e^{K} - \Id)  \begin{pmatrix} v_{n-1} \\  m_{n-1}\end{pmatrix} \right)(s,\PhB_{t,s}(x)) \right|
\\  
&\leq   \| \nabla^2 m_{n}(s,\cdot) \|_{L^{\infty}} \left\| \nabla^2 \pi_1 (e^{K} - \Id)  \begin{pmatrix} v_{n-1} \\  m_{n-1}\end{pmatrix}(s,\cdot)\right\|_{L^{\infty}}  \left| \Ph_{t,s}(x) - \PhB_{t,s}(x) \right|
\\ 
&\leq  \| m_{n}(s,\cdot) \|_{H^{\ell}} \left\| \pi_1 (e^{K} - \Id)  \begin{pmatrix} v_{n-1} \\  m_{n-1}\end{pmatrix}(s,\cdot) \right\|_{H^{\ell}}  \left| \Ph_{t,s}(x) - \PhB_{t,s}(x) \right|
\end{align*}
and 
\begin{align}\label{PAI1} 
& \bigg\| \int_t^\infty \bigg( \left((m_{n}\cdot\nabla) \pi_1 (e^{K} - \Id)  \begin{pmatrix} v_{n-1} \\ m_{n-1} \end{pmatrix} 
\right)(s,\Ph_{t,s}) 
\\ \nonumber
&   \qquad \qquad \qquad \qquad
-  \left((m_{n}\cdot\nabla) \pi_1 (e^{K} - \Id)  \begin{pmatrix} v_{n-1} \\ m_{n-1} \end{pmatrix} \right)(s,\PhB_{t,s}) \bigg) \, \de s \bigg\|_{L^2}
\\ \nonumber
& \lesssim 
\int_t^\infty \| m_{n}(s,\cdot) \|_{H^{\ell}} \left\|\pi_1 (e^{K} - \Id) \begin{pmatrix} v_{n-1}(s,\cdot) \\  m_{n-1}(s,\cdot)\end{pmatrix}\right\|_{H^{\ell}}  \| \Ph_{t,s} - \PhB_{t,s} \|_{L^2} \, \de s
\\ \nonumber
& \lesssim \| v_n - v_{n-1}\|_{\ZetaZero} \|m_{n}\|_{\Zeta} \left\| \pi_1 (e^{K} - \Id)  \begin{pmatrix} v_{n-1} \\  m_{n-1}\end{pmatrix} \right\|_{\Zeta}
\int_t^\infty e^{- 2 \theta s}   \, \de s
\\ & \nonumber 
\overset{\eqref{InductionNBis}-\eqref{fdnjskkjfdh64632guyfew27894}}{\lesssim} \varepsilon^3 \delta^{n-1}
\int_t^\infty e^{- 2 \theta s}   \, \de s 
\overset{\eqref{IH}}{\leq} D^n \frac{\varepsilon^{n+3}}{2\theta} e^{-2 \theta t}
\end{align}
where we have used Lemma \ref{stime_flussoBisStability} in the second inequality.
Thus
\begin{align}\label{PAI2}
& \sup_{t \geq 0}e^{\theta t} \bigg\| \int_t^\infty \bigg( \left((m_{n}\cdot\nabla) \pi_1 (e^{K} - \Id) \begin{pmatrix} v_{n-1} \\ m_{n-1} \end{pmatrix} 
\right)(s,\Ph_{t,s}) 
\\ \nonumber
&   \qquad \qquad \qquad \qquad
-  \left((m_{n}\cdot\nabla) \pi_1 (e^{K} - \Id) \begin{pmatrix} v_{n-1} \\ m_{n-1} \end{pmatrix} \right)(s,\PhB_{t,s}) \bigg) \, \de s \bigg\|_{L^2}
\\ &
\nonumber 
\leq  \sup_{t \geq 0} C D^n \frac{\varepsilon^{n+3}}{2\theta} e^{- \theta t} \leq C D^n \frac{\varepsilon^{n+3}}{2\theta} .
\end{align}
In conclusion 
\begin{equation}
\|\eqref{Uff2.5}\|_{\ZetaZero} 
\leq C \varepsilon^2 \|\mn - m_n\|_{\ZetaZero} + C D^n \varepsilon^{n+2}.
\end{equation}

\textbf{Bound of \eqref{Uff3}.}
We estimate
\begin{align*}
\|\eqref{Uff3}\|_{\ZetaZero} &   \leq 
\sup_{t \geq 0}e^{\theta t} \left\| \int_t^\infty 
 \left( \pi_2  
\mathcal{A} \begin{pmatrix} v_n - v_{n-1}\\ m_n - m_{n-1} \end{pmatrix}
 \right) (s,\Ph_{t,s})  \, \de s \right\|_{L^2}
 \\ 
 & + \sup_{t \geq 0}e^{\theta t} \left\| \int_t^\infty 
\left( \left( \pi_2  
\mathcal{A} \begin{pmatrix} v_{n-1} \\ m_{n-1} \end{pmatrix}
 \right) (s,\Ph_{t,s}) - 
 \left( \pi_2  
\mathcal{A} \begin{pmatrix} v_{n-1} \\ m_{n-1} \end{pmatrix}
 \right) (s,\PhB_{t,s}) \right) \, \de s \right\|_{L^2}.
 \end{align*}
For the first term on the right hand side, we use Lemma \ref{CorollaryFaaDiBruno}, Lemma \ref{stime flussoBis} and \eqref{InductionNBis} to get
\begin{align*}  
 \left\|  \int_t^\infty 
 \left( \pi_2  
\mathcal{A} \begin{pmatrix} v_n - v_{n-1} \\ m_n - m_{n-1} \end{pmatrix}
 \right) (s,\Ph_{t,s})  \, \de s \right\|_{L^2}
  \leq 
  \int_t^\infty 
\left\|  \left( \pi_2  
\mathcal{A} \begin{pmatrix} v_n - v_{n-1}  \\ m_n - m_{n-1} \end{pmatrix}
 \right) 
 (s,\cdot)  \, \de s \right\|_{L^2} .
\end{align*}
We now use the bound 
\begin{equation}\label{fdhbsjdhghgsdsh7473462}
\left\| \mathcal{A} \begin{pmatrix} v_n - v_{n-1}  \\ m_n - m_{n-1} \end{pmatrix} \right\|_{L^2 \times L^2} \lesssim
\varepsilon ^2 \|(v_n - v_{n-1}, m_n - m_{n-1} )\|_{L^2 \times L^2}   
\end{equation}
that follows by the estimate \eqref{fdhbsjdhghgsdsh7473462Stronger}. 
Using \eqref{fdhbsjdhghgsdsh7473462} and proceeding as in \eqref{hfdgsugdyfdsUff3} (with obvious modifications) we arrive to
\begin{align*}
\sup_{t \geq 0}e^{\theta t} \left\| \int_t^\infty 
 \left( \pi_2  
\mathcal{A} \begin{pmatrix} v_n - v_{n-1} \\ m_n - m_{n-1}\end{pmatrix}
 \right) (s,\Ph_{t,s})  \, \de s \right\|_{L^2} 
 \lesssim \frac{\varepsilon^2}{\theta} \delta_{n-1} 
 \overset{\eqref{IH}}{\lesssim} D^n \frac{\varepsilon^{n+2}}{\theta} .
 \end{align*}

In order to handle the $\ZetaZero$ norm of the second term above we use
\begin{align*}
\left| \left( \pi_2  
\mathcal{A} \begin{pmatrix} v_{n-1} \\ m_{n-1} \end{pmatrix}
 \right) (s,\Ph_{t,s}(x)) - \right.&\left.
 \left( \pi_2  
\mathcal{A} \begin{pmatrix} v_{n-1} \\ m_{n-1} \end{pmatrix}
 \right) (s,\PhB_{t,s}(x)) \right| 
 \\ &
 \leq \left\| \nabla \pi_2  \mathcal{A} \begin{pmatrix} v_{n-1}(s,\cdot) \\ m_{n-1}(s,\cdot) \end{pmatrix} \right\|_{L^{\infty}} \left| 
 \Ph_{t,s}(x) - \PhB_{t,s}(x) \right| 
\\ &
 \leq \left\|  \pi_2  \mathcal{A} \begin{pmatrix} v_{n-1}(s,\cdot) \\ m_{n-1}(s,\cdot) \end{pmatrix} \right\|_{H^{\ell}} \left| 
 \Ph_{t,s}(x) - \PhB_{t,s}(x) \right| 
 \\ & 
 \lesssim \varepsilon^2 (\| v_{n-1}(s,\cdot) \|_{H^{\ell}} + \| m_{n-1}(s,\cdot) \|_{H^{\ell}} ) \left| 
 \Ph_{t,s}(x) - \PhB_{t,s}(x) \right|
\end{align*}
where the second inequality follows by the Sobolev embedding and the third one by \eqref{mathcalAbound}.
Thus, collecting these estimates we arrive at
\begin{align}
\left\| \int_t^\infty \left( \pi_2  
\mathcal{A} \begin{pmatrix} v_{n-1} \\ m_{n-1} \end{pmatrix}
 \right) (s,\Ph_{t,s}) \right.&\left.- 
 \left( \pi_2  
\mathcal{A} \begin{pmatrix} v_{n-1} \\ m_{n-1} \end{pmatrix}
 \right) (s,\PhB_{t,s}) \, \de s \right\|_{L^2}\nonumber
\\ \nonumber
& \lesssim  \varepsilon^2
\int_t^\infty (\| m_{n-1}(s,\cdot) \|_{H^{\ell}} + \|v_{n-1}(s,\cdot)\|_{H^{\ell}} )  \| \Ph_{t,s} - \PhB_{t,s}\|_{L^2}
 \, \de s
\\ \nonumber
& \leq  \varepsilon^2  \| v_n - v_{n-1}\|_{\ZetaZero} (\|m_{n-1}\|_{\Zeta} + \|v_{n-1}\|_{\Zeta} )
\int_t^\infty e^{-  \theta s}   \, \de s
\\ & \nonumber
\overset{\eqref{InductionNBis}}{\leq} \varepsilon^3 \delta^{n-1}
\int_t^\infty e^{-  \theta s}   \, \de s 
\overset{\eqref{IH}}{\leq} D^n \frac{\varepsilon^{n+3}}{\theta} e^{- \theta t},
\end{align}
where we have used Lemma \ref{stime_flussoBisStability} in the second inequality.
In conclusion
\begin{equation}
\sup_{t \geq 0}e^{\theta t} \left\| \int_t^\infty \left( \pi_2  
\mathcal{A} \begin{pmatrix} v_{n-1} \\ m_{n-1} \end{pmatrix}
 \right) (s,\Ph_{t,s}) - 
 \left( \pi_2  
\mathcal{A} \begin{pmatrix} v_{n-1} \\ m_{n-1} \end{pmatrix}
 \right) (s,\PhB_{t,s}) \, \de s \right\|_{L^2} \leq C D^n    \frac{\varepsilon^{n+3}}{\theta}. 
\end{equation} 

\textbf{Bound of \eqref{Uff4}.}
We estimate 
\begin{align*}
\|\eqref{Uff4}\|_{\ZetaZero} &  \lesssim  
\sup_{t \geq 0}e^{\theta t} \left\| \int_t^\infty \left( \left( \pi_2 \Gamma(v_{n},m_{n})
 \right)(s,\Ph_{t,s}) -  \left( \pi_2 \Gamma(v_{n-1},m_{n-1})
 \right)(s,\Ph_{t,s}) \right) \, \de s \right\|_{L^2}
\\
& +  
\sup_{t \geq 0}e^{\theta t} \left\| \int_t^\infty \left( \left( \pi_2 \Gamma(v_{n-1},m_{n-1})
 \right) (s,\Ph_{t,s}) -  \left( \pi_2 \Gamma(v_{n-1},m_{n-1})
 \right)(s,\PhB_{t,s}) \right) \, \de s \right\|_{L^2}
\end{align*}

For the first term on the right hand side above, we use \eqref{fnjwknbdshdsfg533B2} to get
\begin{equation}  
\sup_{t \geq 0}e^{\theta t} \left\| \int_t^\infty \left( \left( \pi_2 \Gamma(v_{n},m_{n})
 \right)(s,\Ph_{t,s}) -  \left( \pi_2 \Gamma(v_{n-1},m_{n-1})
 \right)(s,\Ph_{t,s}) \right) \, \de s \right\|_{L^2}
 \leq C \varepsilon^2 \delta_{n-1} \overset{\eqref{IH}}{\leq} C D^n \varepsilon^{n+2} .
\end{equation}

In order to handle the $\ZetaZero$ norm of the second term, we use 
\begin{align*}
\left| \left( \pi_2 \Gamma(v_{n-1},m_{n-1})
 \right) (s,\Ph_{t,s}(x)) \right.&\left.-  \left( \pi_2 \Gamma(v_{n-1},m_{n-1})
 \right)(s,\PhB_{t,s}(x)) \right| 
 \\ &
 \leq \| \nabla \pi_2 \Gamma(v_{n-1},m_{n-1})
 (s,\cdot) \|_{L^{\infty}} \left| 
 \Ph_{t,s}(x) - \PhB_{t,s}(x) \right| 
\\ &
 \leq \|  \pi_2 \Gamma(v_{n-1},m_{n-1})
 (s,\cdot)  \|_{H^{\ell}} \left| 
 \Ph_{t,s}(x) - \PhB_{t,s}(x) \right| 
 \\ & 
 \lesssim \varepsilon (\| v_{n-1}(s,\cdot) \|^2_{H^{\ell}} + \| m_{n-1}(s,\cdot) \|^2_{H^{\ell}} ) \left| 
 \Ph_{t,s}(x) - \PhB_{t,s}(x) \right|
\end{align*}
where the second inequality follows by the Sobolev embedding and the third one by
 \eqref{GammaEstimate1}.
Thus, it follows that 
\begin{align*}
\Bigg\| \int_t^\infty \left( \left( \pi_2 \Gamma(v_{n-1},m_{n-1})
 \right) \right.&\left.(s,\Ph_{t,s}) -  \left( \pi_2 \Gamma(v_{n-1},m_{n-1})
 \right)(s,\PhB_{t,s}) \right)  \, \de s \Bigg\|_{L^2}
\\ 
& \lesssim  \varepsilon
\int_t^\infty (\| m_{n-1}(s,\cdot) \|^2_{H^{\ell}} + \|v_{n-1}(s,\cdot)\|^2_{H^{\ell}} )  \| \Ph_{t,s} - \PhB_{t,s} \|_{L^2} \, \de s
\\ 
& \leq  \varepsilon  \| v_n - v_{n-1}\|_{\ZetaZero} (\|m_{n-1}\|^2_{\Zeta} + \|v_{n-1}\|^2_{\Zeta} )
\\
&\overset{\eqref{InductionNBis}}{\leq} \varepsilon^3 \delta^{n-1}
\int_t^\infty e^{-  2 \theta s}   \, \de s \\
&\overset{\eqref{IH}}{\leq} D^n \frac{\varepsilon^{n+3}}{2\theta} e^{- 2 \theta t},
\end{align*}
where in the second inequality we have used Lemma \ref{stime_flussoBisStability}.
In conclusion
\begin{align*}
\sup_{t \geq 0}e^{\theta t} \Bigg\| \int_t^\infty \Big( \left( \pi_2 \Gamma(v_{n-1},m_{n-1})\right) & (s,\Ph_{t,s}) -  \left( \pi_2 \Gamma(v_{n-1},m_{n-1})\right)(s,\PhB_{t,s}) \Big) \, \de s \Bigg\|_{L^2} 
\\
& \leq  \sup_{t \geq 0} C D^n
\frac{\varepsilon^{n+3}}{2\theta} e^{-  \theta t} \leq 
C D^n \frac{\varepsilon^{n+3}}{2\theta}. 
\end{align*}

\textbf{Bound of \eqref{Uff5}.}
We estimate 
\begin{align*}
\|\eqref{Uff5}\|_{\ZetaZero} &  \lesssim  
\sup_{t \geq 0}e^{\theta t} \left\| \int_t^\infty \left( \left(    \dive N_2(v_n,m_n) \right) (s,\Ph_{t,s}) - \left(     \dive N_2(v_{n-1},m_{n-1}) \right) (s,\Ph_{t,s})\right) \, \de s \right\|_{L^2}
\\
& +  
\sup_{t \geq 0}e^{\theta t} \left\| \int_t^\infty \left( 
\left(    \dive N_2(v_{n-1},m_{n-1}) \right) (s,\Ph_{t,s}) - \left(    \dive N_2(v_{n-1},m_{n-1}) \right) (s,\PhB_{t,s})
\right) \, \de s \right\|_{L^2}
\end{align*}

For the first term on the right hand side, we use Lemma \ref{Bunch3} to get 
\begin{align*}
\sup_{t \geq 0}e^{\theta t} \left\| \int_t^\infty \left( \left( \left(    \dive N_2(v_n,m_n) \right) (s,\Ph_{t,s}) - \left(     \dive N_2(v_{n-1},m_{n-1}) \right) (s,\Ph_{t,s})\right) \right) \, \de s \right\|_{L^2} &\lesssim \varepsilon^2 \delta_{n-1} \\
&\overset{\eqref{IH}}{\lesssim} D^n  \varepsilon^{n+2}  
\end{align*}
In order to handle the $\ZetaZero$ norm of the second term we use 
\begin{align*}
\left| \left(   \dive N_2(v_{n-1},m_{n-1}) \right) (s,\Ph_{t,s}(x)) \right.&\left.- \left(   \dive N_2(v_{n-1},m_{n-1}) \right) (s,\PhB_{t,s}(x)) \right|
\\  
&\leq   \| \nabla     \dive N_2(v_{n-1},m_{n-1})(s,\cdot)  \|_{L^{\infty}} \left| \Ph_{t,s}(x) - \PhB_{t,s}(x) \right|
\\ 
&\overset{\eqref{ImprovedDIVBound}}{\lesssim} 
 \varepsilon (\| v_{n-1}(s,\cdot) \|^2_{H^{\ell}} + \| m_{n-1}(s,\cdot) \|^2_{H^{\ell}} ) \left| 
 \Ph_{t,s}(x) - \PhB_{t,s}(x) \right|
\end{align*}
which implies
\begin{align} \nonumber 
&
\left\| \int_t^\infty \left( 
\left(    \dive N_2(v_{n-1},m_{n-1}) \right) (s,\Ph_{t,s}) - \left(    \dive N_2(v_{n-1},m_{n-1}) \right) (s,\PhB_{t,s})
\right) \, \de s \right\|_{L^2}
\\ \nonumber
& \lesssim
\varepsilon \int_t^\infty  (\| v_{n-1}(s,\cdot) \|^2_{H^{\ell}} + \| m_{n-1}(s,\cdot) \|^2_{H^{\ell}} ) \|
 \Ph_{t,s} - \PhB_{t,s} \|_{L^2}
  \, \de s
\\ \nonumber
& \lesssim \varepsilon  \| v_n - v_{n-1}\|_{\ZetaZero} (\| v_{n-1} \|^2_{\Zeta} + \| m_{n-1} \|^2_{\Zeta} )
\int_t^\infty e^{- 2 \theta s}   \, \de s
\\ & \nonumber 
\overset{\eqref{InductionNBis}}{\leq} \varepsilon^2 \delta^{n-1}
\int_t^\infty e^{- 2 \theta s}   \, \de s 
\overset{\eqref{IH}}{\lesssim} D^n \frac{\varepsilon^{n+3}}{2\theta} e^{-2 \theta t}
\end{align}
where we have used Lemma \ref{stime_flussoBisStability} in the second inequality.
Thus
\begin{align*}  
\sup_{t \geq 0}e^{\theta t} \left\| \int_t^\infty \left( 
\left( \dive N_2(v_{n-1},m_{n-1}) \right) (s,\Ph_{t,s}) \right.\right.&\left.\left.- \left(    \dive N_2(v_{n-1},m_{n-1}) \right) (s,\PhB_{t,s})
\right) \, \de s \right\|_{L^2}
\\ 
& \leq
 \sup_{t \geq 0} C D^n \frac{\varepsilon^{n+3}}{2\theta} e^{- \theta t} \leq C D^n \frac{\varepsilon^{n+3}}{2\theta}
\end{align*}

\textbf{Conclusion.}
Collecting the previous estimates we arrive at
$$
\| \mn - m_n\|_{\ZetaZero} \leq C \varepsilon \| \mn - m_n\|_{\ZetaZero} + C D^n \varepsilon^{n+1},
$$
which implies \eqref{VeryFinal1BoundIH}, provided $\varepsilon > 0$ is chosen sufficiently small (to absorb the $C\varepsilon$ term into the left-hand side) and $D$ is chosen sufficiently large.\\
\\
\noindent \textbf{Proof of \eqref{VeryFinal3BoundIH}.} From the iterative equation \eqref{iterazione}, we deduce
\begin{align*}
\partial_t(\vn - v_n) = & \Delta(\vn - v_n) + \pi_1 \mathcal{A} \begin{pmatrix} v_n - v_{n-1} \\ m_n - m_{n-1} \end{pmatrix} + \pi_1 (\Gamma(v_n, m_n) - \Gamma(v_{n-1}, m_{n-1})) \\
& + \mathbb{P} \left( m_n\cdot\nabla m_n - m_{n-1}\cdot\nabla m_{n-1} - v_n\cdot\nabla v_n + v_{n-1} \cdot\nabla v_{n-1}\right) \\
&+ \dive ( N_1(v_n, m_n) - N_1(v_{n-1}, m_{n-1})).
\end{align*}
Taking the $L^2$ scalar product with $(\vn - v_n)$, integrating by parts, and using the self-adjointness of the Leray projector $\mathbb{P}$ (noting $\mathbb{P}v_{n+1}=v_{n+1}$), arriving to
\begin{align*}
\frac{1}{2} \frac{\de}{\de t}\|\vn - v_n\|_{L^2}^2 & + \|\nabla(\vn - v_n)\|_{L^2}^2 \\
= & \int_{\T^3} \left( \pi_1 \mathcal{A} \begin{pmatrix} v_n - v_{n-1} \\ m_n - m_{n-1} \end{pmatrix} + \pi_1 \Gamma(v_n, m_n) - \pi_1 \Gamma(v_{n-1}, m_{n-1}) \right) \cdot (\vn - v_n) \, \de x \\
& + \int_{\T^3} (- m_n\otimes m_n + m_{n-1}\otimes m_{n-1} + v_n \otimes v_n - v_{n-1} \otimes v_{n-1}) \cdot \nabla (\vn - v_n) \, \de x \\
& - \int_{\T^3} (N_1(v_n, m_n) - N_1(v_{n-1}, m_{n-1})) \cdot \nabla (\vn - v_n) \, \de x.
\end{align*}
Applying Lemma \ref{FTWOM1IL}, we arrive at
\begin{equation}
\frac{\de}{\de t}\| \vn - v_n\|_{L^2}^2 + \|\nabla (\vn - v_n)\|_{L^2}^2 \leq C \varepsilon^2 (\| v_n - v_{n-1}\|_{L^2}^2 + \| m_n - m_{n-1}\|_{L^2}^2).
\end{equation}
Integrating over $t \in [0, T]$, $T \geq 0$
\begin{align*}
& \| (\vn - v_n)(T,\cdot) \|_{L^2}^2 + \int_{0}^{T} \|\nabla (\vn - v_n)(t,\cdot)\|^2_{L^2} \, \de t \\
& \leq C\varepsilon^2 \int_0^T (\| v_n - v_{n-1}\|_{L^2}^2 + \| m_n - m_{n-1}\|_{L^2}^2) \, \de t \\
& \leq C\varepsilon^2 \int_0^T e^{-2 \theta t} (\| v_n - v_{n-1}\|_{\ZetaZero}^2 + \| m_n - m_{n-1}\|_{\ZetaZero}^2) \, \de t \\
& \leq C \varepsilon^2 \delta_{n-1}^2 \frac{1}{2 \theta} (1 - e^{-2 \theta T}) \overset{\eqref{IH}}{\leq} \frac{C}{\theta} D^{2n} \varepsilon^{2(n+1)} (1 - e^{-2 \theta T}),
\end{align*}
and finally, taking the limit $T \to \infty$, we get
\begin{equation}
\limsup_{T \to \infty} \| (\vn - v_n)(T,\cdot) \|_{L^2}^2 + \int_{0}^{\infty} \|\nabla (\vn - v_n)(t,\cdot)\|^2_{L^2} \, \de t \leq \frac{C}{\theta} D^{2n} \varepsilon^{2(n+1)},
\end{equation}
which implies \eqref{VeryFinal3BoundIH} for $D$ sufficiently large.\\
\\
\noindent \textbf{Proof of \eqref{VeryFinal2BoundIH}.} From the first equation in \eqref{iterazione integrale}, we deduce:
\begin{align*}
(\vn - v_n)(t,\cdot) = & \int_0^t e^{(t-s)\Delta} \mathbb{P} \left( (m_n\cdot\nabla) m_n - (m_{n-1}\cdot\nabla) m_{n-1} - (v_n\cdot\nabla) v_n + (v_{n-1}\cdot\nabla) v_{n-1} \right) \, \de s \\
& + \int_0^t e^{(t-s)\Delta} \mathbb{P} \left( \dive (N_{1}(v_n, m_n) - N_{1}(v_{n-1}, m_{n-1})) \right) \, \de s \\
& + \int_0^t e^{(t-s)\Delta} \left( \pi_1 \mathcal{A} \begin{pmatrix} v_n - v_{n-1} \\ m_n - m_{n-1} \end{pmatrix} + \pi_1 (\Gamma(v_n, m_n) - \Gamma(v_{n-1}, m_{n-1})) \right) \, \de s.
\end{align*}
Using Lemma \ref{gmjdskhkjdsfgjkgsdnj9435782543} and inequality \eqref{ndjkskngdkskhjdf123} (valid replacing $(v_n, m_n)$ with $(v_n-v_{n-1}, m_n-m_{n-1})$), we easily deduce
$$
\| \vn - v_n \|_{\ZetaZero} \leq C \varepsilon \delta_{n-1} \overset{\eqref{IH}}{\leq} C D^{n} \varepsilon^{n+1},
$$
which implies \eqref{VeryFinal2BoundIH} for $D$ sufficiently large.\\
\\
\noindent \textbf{Final step.} To complete the proof, we show \eqref{InductionNBisbdfshkbjgks}. From the uniform bound \eqref{InductionNBis}, the sequence $(v_n, m_n)$ admits a subsequence converging weakly to some $(\tilde v, \tilde m)$ satisfying
\begin{equation}
\| \tilde v \|_{\Zeta} + \| \nabla \tilde v \|_{L^2_t H^{\ell}} + \| \tilde m\|_{\Zeta} \leq \varepsilon.
\end{equation}
By the convergence \eqref{ConvergenceReal}, we must have $(v, m) = (\tilde v, \tilde m)$, and thus \eqref{InductionNBisbdfshkbjgks} follows.
\end{proof}

\section{Proof of Theorems \ref{RealMAINTHM}, \ref{ContMap}}\label{Sec:RealMAINTHM}
We are now ready to prove Theorem \ref{RealMAINTHM}.
\begin{proof}[\underline{Proof of Theorem \ref{RealMAINTHM}}.]
We divide the proof in three steps.\\
\\
\underline{\em Step 1.}\, Existence of the limit solution. \\
\\
From Theorem \ref{ThmSecondIterativeMAIN2}, we know that there exists $\varepsilon_1 > 0$ sufficiently small such that the following holds. Given any $\varepsilon \in (0, \varepsilon_1)$ and any zero-average, divergence-free vector field $y : \T^3 \to \R^3$ satisfying
\begin{equation}\label{yNorm}
\|y\|_{H^{\ell}} \leq \frac{\varepsilon}{3},
\end{equation}
there exists a pair of zero-average, divergence-free vector fields $(v, m): [0, +\infty) \times \T^3 \to \R^3$ bounded by
\begin{equation}\label{fjnsaljndf64678234bdfhyws}
\| v \|_{\Zeta} + \| m\|_{\Zeta} \leq \varepsilon,
\end{equation} 
and a sequence of zero-average, divergence-free vector fields $v_{n}, m_{n} : [0, +\infty) \times \T^3 \to \R^3$ satisfying the integral equations \eqref{iterazione integrale}, such that
\begin{equation*}
\|v_{n} - v\|_{\ZetaB} + \|m_{n} - m\|_{\ZetaB} \overset{n \to \infty}{\longrightarrow} 0,
\end{equation*}
for all $\ell' < \ell$. Note that the sequence $(v_n, m_n)$, and consequently the limit pair $(v,m)$, depends on the initial datum $y$.\\
\\
\underline{\em Step 2.}\, Passing to the limit and reconstructing the MHD solution. \\
\\
It is straightforward to see that, as a byproduct of the arguments from the previous section, we can pass to the limit as $n \to \infty$ in \eqref{iterazione integrale}, thereby proving that $(v, m)$ satisfies the limit integral equations \eqref{iterazione integralePrequel}. Then, invoking Theorem \ref{FinalTHM1dfhgsjdhgfs}, we deduce that the pair $(v, m)$ satisfies the differential equations \eqref{nfdjkshjbffbbfbf}. Thus, by Theorem \ref{fdanjknjfdsjn756734}, we conclude that
\begin{equation}\label{fdmkslnjnkdsf64}
\begin{pmatrix} u \\ b \end{pmatrix} 
:=
\begin{pmatrix} 0 \\ B \end{pmatrix}
+ e^{K} \begin{pmatrix} v \\ m \end{pmatrix},
\end{equation}
is a global solution of the MHD equations \eqref{eq:mhdInLemma}, supplemented with the initial datum
$$
\begin{pmatrix} u(0, x) \\ b(0, x) \end{pmatrix} 
:=
\begin{pmatrix} 0 \\ B(x) \end{pmatrix}
+ e^{K} \begin{pmatrix} y(x) \\ m(0, x) \end{pmatrix}.
$$ 
\underline{\em Step 3.}\, Structural properties and conclusion.\\
\\
Since $\dive B = \dive v = \dive m = 0$, we have $\dive u = \dive b = 0$, which follows as a direct consequence of Proposition \ref{prop:algebraic_properties}. Furthermore, since the perturbations have zero spatial average, i.e., 
$$
\int_{\T^3} v(t,x) \, \de x = \int_{\T^3} m(t,x) \, \de x = 0,
$$ 
we obtain
$$
\int_{\T^3} b(t,x) \, \de x = \int_{\T^3} B(x) \, \de x \quad \text{and} \quad \int_{\T^3} u(t,x) \, \de x = 0,
$$ 
again by Proposition \ref{prop:algebraic_properties}. From the equations \eqref{iterazione integralePrequel} and the bound \eqref{fjnsaljndf64678234bdfhyws}, proceeding exactly as in the proof of Theorems \ref{ThMFirstRecursiveBound}-\ref{ThmSecondIterative} we obtain 
\begin{equation}\label{improvedBoundForv}
\|v\|_{\Zeta} \leq   \|y\|_{H^\ell} + C_{\theta}   \varepsilon  (\| v \|_{\Zeta} + \| m \|_{\Zeta})
\end{equation}
and
\begin{equation}\label{improvedBoundForm}
\|m\|_{\Zeta} \leq    C_{\theta}  \varepsilon  \| m \|_{\Zeta} + 
C_{\theta} \varepsilon^2 \| v \|_{\Zeta}  .
\end{equation}
Summing these estimates we arrive to
\begin{equation}\label{recthisfinal}
\|v\|_{\Zeta} + \|m\|_{\Zeta} \leq   \|y\|_{H^\ell} + 2C_{\theta}   \varepsilon  (\| v \|_{\Zeta} + \| m \|_{\Zeta}).
\end{equation}
Taking $\varepsilon_1 >0$ such that 
\begin{equation}\label{recthisfinal2}
2C_{\theta} \varepsilon_1 \leq 1/2,
\end{equation}
and recalling that 
$\varepsilon \in (0, \varepsilon_1)$ we can absorb the second term on the right hand side into the left hand side and we arrive to
\begin{equation}\label{nmkdofskngfseko54}
\|v\|_{\Zeta} + \|m\|_{\Zeta} \leq 2 \|y\|_{H^{\ell}}.
\end{equation}
Plugging this estimate into \eqref{improvedBoundForv}-\eqref{improvedBoundForm} we arrive to 
\begin{equation}
\|v\|_{\Zeta} \leq   \|y\|_{H^\ell} + 2 C_{\theta}   \varepsilon  \|y\|_{H^{\ell}}
\end{equation}
and
\begin{equation}
\|m\|_{\Zeta} \leq    C_{\theta}  \varepsilon  \| m \|_{\Zeta} + 
2 C_{\theta} \varepsilon^2 \|y\|_{H^{\ell}}.
\end{equation}
Recalling \eqref{recthisfinal2} the latter gives
\begin{equation}\label{fmnjsokngfre77565}
\|m\|_{\Zeta} \leq  
\frac83 C_{\theta} \varepsilon^2 \|y\|_{H^{\ell}}. 
\end{equation}
Noting that 
\begin{equation}\label{fdmkslnjnkdsf64}
\begin{pmatrix} u \\ b \end{pmatrix} 
:=
\begin{pmatrix} 0 \\ B \end{pmatrix} + \begin{pmatrix} v \\ m \end{pmatrix}
+ (e^{K} - \Id) \begin{pmatrix} v \\ m \end{pmatrix},
\end{equation}
and that 
\begin{equation}\label{fdmkslnjnkdsf64Bis}
\| (e^{K} - \Id) \begin{pmatrix} v \\ m \end{pmatrix} \|_{\Zeta} \overset{\eqref{fdnjskbjnfhjskhjdbfgsjk3}}{\leq} C \varepsilon
( \| v \|_{\Zeta} +  \| m \|_{\Zeta}) \overset{\eqref{nmkdofskngfseko54}}{\leq} 2 C \varepsilon 
 \| y \|_{H^{\ell}},
\end{equation}
we can easily deduce the quantitative relaxation bound \eqref{Bound:QuantitativeRelaxation}. 

Theorem \ref{RealMAINTHM} is thus proved by setting the parametrization map to
\begin{equation}
\begin{pmatrix} \phi_1(y) \\  \phi_2(y) \end{pmatrix}
:= e^{K} \begin{pmatrix} y(\cdot) \\ m(0, \cdot) \end{pmatrix},
\end{equation}
where we recall that $m(0, \cdot)$ depends on $y$.

Since the matrix exponential $e^{K}$ is  invertible, the mapping $\Lambda : y \mapsto \Lambda(y):=(\phi_1(y), B + \phi_2(y))$ is automatically injective. \\

Let now $\Psi_{t_1, t_2}$ the flow associated to the velocity field $u$, as defined in \eqref{DiffPsiDef}. 
The induction equation~\eqref{induction} rewrites as
\begin{equation}\label{r42789548123}
b(t_2 , \Psi_{t_1, t_2} x) = (\nabla \Psi_{t_1, t_2}x) b(t_1,x).  
\end{equation}

From Lemma \ref{LastFlowBound} we know that $\Psi_{t_1, t_2}$ converges in $W^{1,\infty}$ to a volume preserving 
diffeomorphism that we denote with
$\Psi_{t_1, \infty}$, as $t_2  \to \infty$, namely
\begin{equation}
\lim_{t_2 \to \infty} \| \Psi_{t_1,t_2} - \Psi_{t_1,\infty}\|_{W^{1,\infty}} = 0.
\end{equation}
Thus, in particular 
\begin{equation}\label{r4278954812}
\lim_{t_2 \to \infty} (\nabla \Psi_{t_1, t_2}) b(t_1) = (\nabla \Psi_{t_1, \infty}) b(t_1) .
\end{equation}
We now claim that 
\begin{equation}\label{r427895481}
\lim_{t_2 \to \infty}  b(t_2 , \Psi_{t_1, t_2} x) = B(\Psi_{t_1, \infty} x)
\end{equation}
To prove this we estimate
\begin{align}
| b(t_2 , \Psi_{t_1, t_2} x) - B(\Psi_{t_1, \infty} x) | \leq 
| b(t_2 , \Psi_{t_1, t_2} x) - B(\Psi_{t_1, t_2} x) | + | B(\Psi_{t_1, t_2} x)  -  B(\Psi_{t_1, \infty} x) |. 
\end{align}
The second term on the right hand side vanishes as $t_2 \to \infty$ since $\Psi_{t_1, t_2} x \overset{t_2 \to \infty}{\to} \Psi_{t_1, \infty} x$ and $B$ is a continuous function. To handle the first term into the right hand side we note that
$$
| b(t_2 , \Psi_{t_1, t_2} x) - B(\Psi_{t_1, t_2} x) | \leq \sup_{y \in \T^3} |b(t_2 , y) - B(y)|
$$
and the right hand side now vanishes as $t_2 \to \infty$ as consequence of \eqref{Bound:QuantitativeRelaxation} (and Sobolev embedding). This proves the claim \eqref{r427895481}. We can use \eqref{r427895481}-\eqref{r4278954812} to pass to the limit
$t_2 \to \infty$ into the identity~\eqref{r42789548123} getting
\begin{equation}
B(\Psi_{t_1, \infty} x) = (\nabla \Psi_{t_1, \infty}x) b(t_1,x),
\end{equation} 
as desired. This completes the proof.


\end{proof}

We now move to the proof of Theorem \ref{ContMap}.
\begin{proof}
[\underline{Proof of Theorem \ref{ContMap}}.]
Taking the difference of equations \eqref{iterazione integralePrequel} for $(v,m)$ and $(v',m')$, respectively, 
and using the bound \eqref{fjnsaljndf64678234bdfhyws} (that now holds for both $(v,m)$ and $(v',m')$) the estimates 
\eqref{fjisuhfd66835985u} can be proved proceeding exactly as in the proof of Theorem \eqref{ThmSecondIterativeMAIN2}
and then coming back to the variables $(u,b)$, $(u',b')$ using \eqref{fdmkslnjnkdsf64} and \eqref{fdmkslnjnkdsf64Bis}.
We omit the full calculation, that is straightforward in light of the proof of Theorem \eqref{ThmSecondIterativeMAIN2}. 
\end{proof}

\section{Proof of Theorem \ref{OrbitsTHMPrequel}}\label{Sec:proofofOrbitsTHMPrequel}

In this section we show that the orbits of the (infinite dimensional family of) solutions constructed in Theorem \ref{RealMAINTHM}
form a large class, that is informally the content of Theorem \ref{OrbitsTHMPrequel}. The proof consists in selecting a representative initial datum lying on a common energy level. Since the energy decreases strictly at positive times, each orbit can intersect the resulting set at most once. Moreover, the selection is continuous and preserves the infinite-dimensional structure of the original parameter set.

\begin{proof}
[\underline{Proof of Theorem \ref{OrbitsTHMPrequel}}.]

We recall that we are assuming
$$
\|B\|_{H^{\ell +2}} \leq \varepsilon, \quad \|B\|_{L^2} =  \sigma \varepsilon.$$

Let $y$ as in Theorem \ref{RealMAINTHM}. In particular,
$$\|y\|_{H^{\ell}} \leq \frac{\varepsilon}{3}.$$ We further restrict to the $y$ that satisfies 
$$
\|y\|_{L^2}   = \frac13 c  \sigma \varepsilon.
$$
The constant $c>0$ must be sufficiently small that the choice of $y$ is indeed 
possible (once this is the case there is indeed an large set of such $y$). For instance we can take 
$c = 2^{-\ell - 1}$ (in particular, the constant $c$ only depends on $\ell$). 

Let $(u,b)$ be the solution, given by Theorem \ref{RealMAINTHM}, corresponding to $\theta y$, where $\theta \in (0,1)$.
Recalling~\eqref{fdmkslnjnkdsf64}, we note that 
$$
u(0, x) =   \pi_1 e^{K}\begin{pmatrix}
\theta y(x) \\ 
m(0, x)   
\end{pmatrix} 
=  \theta y(x) +
 \pi_1 (e^{K} - \Id) \begin{pmatrix}
\theta y(x) \\ 
m(0, x)
\end{pmatrix},
$$
and
$$
b(0, x) = B(x) + m(0,x) +   \pi_2 ( e^{K} - \Id) \begin{pmatrix}
\theta y(x) \\ 
m(0, x)   
\end{pmatrix}.
$$
It is worth recalling that the vector field $m(0, \cdot)$ depends on $\theta$ and $y$, even if we do not display this dependence in the notations. Note that we have 
$$
\| u(0,\cdot) \|_{L^2}^2 =  \frac19 c^2  \theta^2    \varepsilon^2 \sigma^2  + r_1, 
$$
with 
$$
\quad r_1 :=  
\bigg\| \pi_1 (e^{K} - \Id) \begin{pmatrix}
\theta y(\cdot) \\ 
m(0, \cdot) 
\end{pmatrix} \bigg\|_{L^2}^2
+ 2 \int_{\T^3}  \theta y(x) \cdot   \pi_1 (e^{K} - \Id) \begin{pmatrix}
\theta y(x)  \\ 
m(0,x)
\end{pmatrix}\de x,
$$
and 
$$
\| b(0,\cdot)\|_{L^2}^2 =     \varepsilon^2 \sigma^2 + r_2 ,
$$
with
\begin{align*}
r_2  := &\| m(0,\cdot) \|_{L^2}^2 + \bigg\|  \pi_2 ( e^{K} - \Id) \begin{pmatrix}
\theta y \\ 
m(0,\cdot)   
\end{pmatrix} \bigg\|_{L^2}^2 
\\ 
&
+ 2 \int_{\T^3} \left(B(x) \cdot m(0,x) +  B(x) \cdot 
\pi_2 ( e^{K} - \Id) \begin{pmatrix}
\theta y(x) \\ 
m(0,x)   
\end{pmatrix} \right)\de x \\
&+ 2\int_{\T^3} m(0,x) \cdot \pi_2 ( e^{K} - \Id) \begin{pmatrix}
\theta y(x) \\ 
m(0,x)   
\end{pmatrix}\de x.
\end{align*}
By using \eqref{fdnjskbjnfhjskhjdbfgsjk3L2}, \eqref{fmnjsokngfre77565} and the Cauchy-Schwarz inequality we can estimate the remainders as follows
$$
|r_1| \lesssim \theta^2 \varepsilon^3 \sigma^2, \quad |r_2| \lesssim \theta \varepsilon^3 \sigma^2.
$$
Thus, for $\theta =1$, we have
$$
\| u(0,\cdot) \|_{L^2}^2 + \| b(0,\cdot) \|_{L^2}^2 =  \frac{1}{9} c^2 \varepsilon^2 \sigma^2 + \varepsilon^2 \sigma^2 +
O(\varepsilon^3 \sigma^2) > \varepsilon^2 \sigma^2 + \frac{1}{18} c^2 \varepsilon^2 \sigma^2,
$$
for all $\varepsilon >0$ sufficiently small.
This means that the function
$$
E: \theta \in [0,1] \to E(\theta):= \| u(0,\cdot) \|_{L^2}^2 + \| b(0,\cdot) \|_{L^2}^2,
$$
satisfies $E(0) = \varepsilon^2 \sigma^2$ and $E(1) > \varepsilon^2 \sigma^2 + \frac{1}{18} c^2 \varepsilon^2 \sigma^2$.
By Theorem \ref{ContMap} 
 we also see that the function $E$ is continuous. This means that there exists at least a value of 
$\theta$ such that $E(\theta) = \varepsilon^2 \sigma^2 + \frac{1}{18}c^2 \varepsilon^2 \sigma^2$. In fact, we will consider the minimum of such $\theta$. To do so, we introduce the set
$$
P := \left\{ \theta \in [0,1] :  E(\theta) > \varepsilon^2 \sigma^2 + \frac{1}{18} c^2 \varepsilon^2 \sigma^2 \right\}.
$$
This set is non empty, since $1 \in P$ (because $E(1)  > \varepsilon^2 \sigma^2 + \frac{1}{18} c^2 \varepsilon^2 \sigma^2$). We can thus set 
$$
\theta(y):= \inf_{\theta \in [0,1]} P.
$$
Taking advantage of the continuity of $E$ it is easy to check that we must have
\begin{equation}\label{njfiwurehg7890934h789t483frdki}
E(\theta(y)) =  \varepsilon^2 \sigma^2 + \frac{1}{18} c^2 \varepsilon^2 \sigma^2,
\end{equation} 
as long as 
$\theta(y) >0$, that is indeed the case since $0 \notin P$ (because $E(0) =\varepsilon^2 \sigma^2$).
Recalling that we have defined, for $s \geq 0$
$$
\mathbb{B}^s_{\rho}:= \left\{ F \in H^{s}(\T^3) : \quad \dive F=0, \quad \int_{\T^3} F =0, \quad  \|F\|_{H^{s}} \leq \rho \right\}.
$$
We can consider the function
\begin{equation}\label{femiownfer6w6478}
\Theta: y  \in \mathbb{B}^\ell_{\frac{\varepsilon}{3}} \cap \partial \mathbb{B}^0_{\frac{c \sigma \varepsilon}{3}}  \to 
\Theta(y):= \theta(y) y.
\end{equation}
This function is clearly injective. Using Theorem \ref{ContMap} it is not hard to show that this function is also continuous. Moreover, 
the global solution $(u, b)$ to the \eqref{eq:mhdMoreStandard} equations with initial datum 
\begin{equation}
(u^\mathrm{in}, b^\mathrm{in}) = (\phi_1(\Theta(y)), B + \phi_2(\Theta(y))),
\end{equation}
satisfies 
\begin{equation}\label{fdsjnfre668453984i}
\| u^\mathrm{in} \|_{L^2}^2 + \| b^\mathrm{in}\|_{L^2}^2 = \varepsilon^2 \sigma^2 + \frac{1}{18} c^2 \varepsilon^2 \sigma^2, 
\qquad \| \nabla u^\mathrm{in} \|_{L^2} >0. 
\end{equation}
The identity follows by \eqref{njfiwurehg7890934h789t483frdki}, the inequality follows by the Poincar\'e
inequality $\| \nabla u^\mathrm{in} \|_{L^2} \geq  \| u^\mathrm{in} \|_{L^2}$ and by the fact that, for all 
$\varepsilon >0$ 
sufficiently small  
\begin{equation}\label{fndhwuybfe784}
\| u^\mathrm{in} \|_{L^2}^2 = 
\frac19 c^2  \theta^2(y)    \varepsilon^2 \sigma^2 +   
O(\theta^2(y) \varepsilon^3 \sigma^2) > \frac1{18} c^2  \theta^2(y)    \varepsilon^2 \sigma^2   >0,
\end{equation}
the lower bound \eqref{fndhwuybfe784} follows by the previous expansion. Recall that $\theta(y) >0$ since 
$0 \neq P$ and the function $E$ is continuous.

Recalling the energy balance \eqref{hfbEnergy}, this implies that 
$$
\|u(t,\cdot)\|_{L^2}^2 + \|b(t,\cdot)\|_{L^2}^2  < \varepsilon^2 \sigma^2 + \frac{1}{18} c^2\varepsilon^2 \sigma^2, \quad \forall t >0,
$$
thus the function 
$$
z \in \Theta\left(\mathbb{B}^\ell_{\frac{\varepsilon}{3}}  \cap \partial \mathbb{B}^0_{\frac{c \sigma \varepsilon}{3}}\right) \to 
\mathcal{O}(z),
$$
must be injective, as orbits leaving $z$ can not return to their initial energy value \eqref{fdsjnfre668453984i}.
The proof would be complete once we show that the function
\begin{equation}\label{fdmwjn7gr785}
\Theta: y  \in \mathbb{B}^\ell_{\frac{\varepsilon}{3}} \cap \partial \mathbb{B}^0_{\frac{c \sigma \varepsilon}{3}}  \to 
\Theta(y) \in  \Theta\left(\mathbb{B}^\ell_{\frac{\varepsilon}{3}}  \cap \partial \mathbb{B}^0_{\frac{c \sigma \varepsilon}{3}}\right),
\end{equation}
is a homeomorphism. 
We have already observed that this function is injective and continuous (with respect to the $L^2$-topology); moreover, it is surjective since its codomain is defined as its image.
By the Rellich-Kondrachov theorem, the domain of $\Theta$ is compact in the $L^2$ topology. On the other hand, since the codomain is a subset of $L^2$ endowed with the induced
$L^2$ topology, it is metrizable and hence Hausdorff. Therefore, $\Theta$ is a continuous bijection from a compact space onto a Hausdorff space and is thus a homeomorphism. The proof is concluded by setting 
$\Sigma := \Theta\left(\mathbb{B}^\ell_{\frac{\varepsilon}{3}}  \cap \partial \mathbb{B}^0_{\frac{c\sigma \varepsilon}{3}}\right)$.
\end{proof}

\

\

\section{Proof of Corollary \ref{MainCor}}
We now move to the proof of Corollary \ref{MainCor}, namely that the relaxation exhibited in Theorem \ref{RealMAINTHM}
is typically non trivial, according to Definition \ref{Def:triviality}. 

Before doing that, we need a technical lemma.
We introduce a family of Beltrami fields with unit frequency (often referred to as ABC flows) 
that, once evaluated at the point $\bar{x}$, form a base (in fact canonical) of $\R^3$. 
Let $x, \bar{x} \in \T^3$ and
$$
\gamma^{\bar{x}}_1(x) := (\cos(x_3 - \bar{x}_3), \sin(x_3 - \bar{x}_3), 0), 
$$
$$
\gamma^{\bar{x}}_2(x) := (0, \cos(x_1 - \bar{x}_1), \sin(x_1 - \bar{x}_1)),
$$
$$
\gamma^{\bar{x}}_3(x) := (\sin(x_2 - \bar{x}_2), 0, \cos(x_2 - \bar{x}_2)),
$$
Note that for all $j \in \{1,2,3\}$ we have 
$$
\curl (\gamma^{\bar{x}}_j) = - \gamma^{\bar{x}}_j,
$$ 
and thus, in particular
\begin{equation}\label{proprietà gamma}
\dive \gamma^{\bar{x}}_j = 0, \quad \int_{\T^3} \gamma^{\bar{x}}_j = 0.
\end{equation}
Moreover, $|\gamma^{\bar{x}}_j(x)| = 1$ for all $x \in \T^3$ and $j \in \{1,2,3 \}$. Furthermore, for any vector $z \in \R^3$ one has
\begin{equation}\label{4354367236437}
z = \sum_{j=1}^3 z_j \gamma^{\bar{x}}_j(\bar{x}).
\end{equation}
In the next Lemma we can tacitly assume that $\Omega \neq \emptyset$, or equivalently that $B \neq 0$, otherwise the 
statement is trivial (in fact the thesis become empty). 
\begin{lem}\label{BaireLemma}
Let $\ell > \frac52$ and define $\Omega := \{x \in \T^3 : B(x)  \neq 0\}$. There is a dense subset $E$ of 
$\Omega$ such that the following holds. Let
$$
\mathcal{F} := \left\{ F \in H_{{\it div}, {\it mean}}^{\ell}(\T^3) : \quad |F(x)|> \frac12, \quad \forall x \in \T^3  \right\}.
$$
There is a dense subset\footnote{$\bar{\mathcal{F}}$ is equipped with the topology induced by the $H^{\ell}(\T^3)$ norm.} $\mathcal{G}$ of $\bar{\mathcal{F}}$ such that for all $G \in \mathcal{G}$ we have
\begin{equation}\label{mnjeifwn77578934OLD}
|G(x) \wedge B(x)| > 0, \quad  | (B(x) \cdot \nabla) G(x) - (G(x) \cdot \nabla)B(x)| > 0, \quad \forall x \in E.
\end{equation}
Moreover, for all $G \in \mathcal{G}$ there is an open set $E_G$ such that $E \subset E_G$ and 
\begin{equation}\label{mnjeifwn77578934}
|G(x) \wedge B(x)| > 0, \quad  | (B(x) \cdot \nabla) G(x) - (G(x) \cdot \nabla)B(x)| > 0, \quad \forall x \in E_G.
\end{equation}
\end{lem}
\begin{rem}
Note that the set $\mathcal{F}$ is not empty. For instance $\gamma^{\bar{x}}_j \in \mathcal{F}$, for all $j \in \{1,2,3\}$, and for all $\bar{x} \in \T^3$. We also note that 
$$
\bar{\mathcal{F}} \subseteq \left\{ F \in H^{\ell}(\T^3), \quad \dive F=0,  \quad \int_{\T^3} F = 0, \quad |F(x)| 
\geq  \frac12 \quad \forall x \in \T^3   \right\}.
$$
\end{rem}

\begin{proof}
Let $\{ x_{n} \}_{n \in \N}$ be a dense subset of $\Omega$ and define the sets
$$
\Gamma^1_{x_n} := \{ F \in \mathcal{F} :  \quad |F(x_n) \wedge B(x_n)| > 0 \}.
$$
First we prove that, for any $n$, $\Gamma^1_{x_n}$ is an open dense subset of $\bar{\mathcal{F}}$. 
The fact that it is open is immediate. 
To prove that it is dense,
it suffices to prove that 
$\Gamma^1_{x_n}$ is dense in~$\mathcal{F}$.
 
Let $F \in \mathcal{F}$. If $ |F(x_n) \wedge B(x_n)| > 0$ then we also have $F \in \Gamma^1_{x_n}$ and we are done. Otherwise we have 
\begin{equation}\label{ty7t4583y76574h}
|F(x_n) \wedge B(x_n)| = 0.
\end{equation} 
Let $k \in \R^3$ with $|k|=1$ chosen such that
$
B(x_n) \cdot k=0. 
$
For any $\delta >0$ we consider the perturbed field
\begin{equation}\label{fji48398y4t389}
F_n (x)= F (x)+ \delta e_n(x), \qquad e_n (x):=  \sum_{j=1}^3 k_j  \gamma^{x_n}_j(x).
\end{equation}
Since the vector fields $\gamma^{x_n}_j$ satisfy \eqref{proprietà gamma} for all $j\in\{1,2,3\}$, it immediately follows that 
\begin{equation}\label{divergenza e media}
\dive F_n = 0,\qquad \mbox{and }\qquad \int_{\T^3} F_n = 0.
\end{equation}
Moreover, evaluating at $x = x_n$, from \eqref{4354367236437}-\eqref{ty7t4583y76574h} we deduce
$$
F_n(x_n) \wedge B(x_n) = \delta  k \wedge B(x_n) 
$$
which, thanks to the orthogonality of $k$ and $B(x_n)$, leads to
\begin{equation}\label{dsj748232039}
|F_n(x_n) \wedge B(x_n)|  = \delta  | B(x_n)| >0.
\end{equation}

Since $F\in\mathcal{F}$ we have that
$$
\min_{x\in\T^3}|F(x)| = \frac1{2}+\e_F,
$$ 
for some $\e_F>0$. Then, we use that $|\gamma^{x_n}_j(x)|, |k| \leq 1$, for all $x \in \T^3$, to obtain for all $\delta >0$ sufficiently small the bound
\begin{equation}\label{r437827ty4358274t352}
\min_{x\in\T^3}|F_n(x)| \geq  \frac1{2}+\e_F-3\delta>\frac12.
\end{equation}
From \eqref{divergenza e media}-\eqref{dsj748232039}-\eqref{r437827ty4358274t352} we deduce 
that $F_n \in \Gamma^1_{x_n}$ provided that $\delta >0$ is sufficiently small. Finally, since 
$$
\|F_n - F\|_{H^\ell} = \delta \|e_n\|_{H^\ell}\lesssim_{\ell} \delta \to 0,
$$
as $\delta \to 0$, we conclude that $\Gamma^1_{x_n}$ is dense in $\mathcal{F}$ (and thus in $\bar{\mathcal{F}}$).\\
\\
We now consider the sets
$$
\Gamma^2_{x_n} := \{ F \in \mathcal{F} :  | \left((B \cdot \nabla) F - (F \cdot \nabla)B \right)(x_n)| > 0 \}.
$$
Again, it is clear that $\Gamma^2_{x_n}$ is an open subset of $\bar{\mathcal{F}}$, for all $n \in \N$. Moreover, as for the previous case, we prove that $\Gamma^2_{x_n}$ is dense in $\bar{\mathcal{F}}$, for all $n \in \N$, by showing that it is dense in~$\mathcal{F}$. 

Let $F \in \mathcal{F}$. If 
$$
| \left((B \cdot \nabla) F - (F \cdot \nabla)B \right)(x_n)| > 0,
$$ 
then there is nothing to prove. Otherwise we have 
\begin{equation}\label{ty7t4583y76574h2}
((B \cdot \nabla) F)(x_n) = ((F \cdot \nabla)B )(x_n).
\end{equation} 
Let $\delta >0$ and define the perturbed field $F_n$ as in \eqref{fji48398y4t389}. Exactly as before, it is also clear that \eqref{divergenza e media} and \eqref{r437827ty4358274t352} hold for  $\delta >0$ small. We now focus on the non-degenerate condition: from \eqref{ty7t4583y76574h2} we have that
\begin{equation}\label{dsj7482320392}
\left( (B \cdot \nabla) F_n - (F_n \cdot \nabla)B \right)(x_n) =  
\delta \big( (B \cdot \nabla) e_n - (e_n \cdot \nabla)B \big)(x_n).
\end{equation}
Thus, in order to show that $F_n \in \Gamma^2_{x_n}$ we must show that there exists $k \in \R^3$ such that the right-hand side of \eqref{dsj7482320392} is non-zero. By contradiction,  if \eqref{dsj7482320392} is zero for all $k \in \R^3$, a direct computation gives that 
\begin{equation}\label{frequenza 1}
(\nabla B )(x_n) =
\begin{pmatrix}
0&0& B_2(x_n)\\
B_3(x_n)&0& 0\\
0&B_1(x_n)&0
\end{pmatrix}.
\end{equation}



We then consider the ABC flows with frequency $2$ defined as
$$
\gamma^{\bar{x}}_1(x) := (\cos(2(x_3 - \bar{x}_3)), \sin(2(x_3 - \bar{x}_3)), 0), 
$$
$$
\gamma^{\bar{x}}_2(x) := (0, \cos(2(x_1 - \bar{x}_1)), \sin(2(x_1 - \bar{x}_1))),
$$
$$
\gamma^{\bar{x}}_3(x) := (\sin(2(x_2 - \bar{x}_2)), 0, \cos(2(x_2 - \bar{x}_2))),
$$
and, correspondingly, we define the vector fields $e_n^2$ and the perturbed field $F_n^2$. We can repeat the same computations as above leading to
\begin{equation}\label{contraddizione frequenza 2}
\left( (B \cdot \nabla) F_n^2 - (F_n^2 \cdot \nabla)B \right)(x_n) =  
\delta \big( (B \cdot \nabla) e_n^2 - (e_n^2 \cdot \nabla)B \big)(x_n).
\end{equation}
Assume that also for this choice of the perturbation, for any $k\in\R^3$ the right-hand side of \eqref{contraddizione frequenza 2} is zero. Then, from a direct computation we deduce 
\begin{equation}\label{frequenza 2}
(\nabla B )(x_n) = 2
\begin{pmatrix}
0&0& B_2(x_n)\\
B_3(x_n)&0& 0\\
0&B_1(x_n)&0
\end{pmatrix}.
\end{equation}
By subtracting the equations \eqref{frequenza 1} and \eqref{frequenza 2} we thus get 
$$
\begin{pmatrix}
0&0& B_2(x_n)\\
B_3(x_n)&0& 0\\
0&B_1(x_n)&0
\end{pmatrix}=0,
$$
which contradicts the definition of $\Omega$. Thus, we conclude that there exists some $k\in\R^3$ such that the perturbed field $F_n$ belongs to $\Gamma^2_{x_n}$.\\
\\
Invoking the Baire theorem, we see that the set
$$
\mathcal{G} := \cap_{n \in \N} ( \Gamma^1_{x_n} \cap \Gamma^2_{x_n}),
$$
is dense in $\bar{\mathcal{F}}$. In particular, for any vector field $G \in \mathcal{G}$ we have
\begin{equation}\label{nhfedw6446843}
|G(x_n) \wedge B(x_n)| > 0 , \quad 
| \left((B \cdot \nabla) G - (G \cdot \nabla)B \right)(x_n)| > 0, \quad \forall n \in \N. 
\end{equation}
Letting 
$$E := \bigcup_{n \in \mathbb{N}} \{ x_n \},$$
then \eqref{mnjeifwn77578934OLD} follows.

By continuity of the vector fields involved, for each $n \in \mathbb{N}$ there exists a sufficiently small radius $\sigma_n > 0$ such that the strict inequalities \eqref{nhfedw6446843} hold not only at $x_n$, but for all $x$ in the ball $\{ x \in \Omega : \vert{}x - x_n\vert{} < \sigma_n \}$. We can thus define the set
$$
E_G := \bigcup_{n \in \mathbb{N}} \{ x \in \Omega : \vert{}x - x_n\vert{} < \sigma_n \},
$$
and the \eqref{mnjeifwn77578934} follows. 
\end{proof}

We conclude the manuscript stating and proving a stronger version of Corollary \ref{MainCor}. It is easy to check that 
Corollary \ref{MainCorStronger} implies Corollary \ref{MainCor}.
In the following statement $\varepsilon_1>0$ is the small threshold given by Theorem 
\ref{RealMAINTHM}.

\begin{cor}\label{MainCorStronger}
Let $\ell \geq 4$ be an integer. 
Let~$B:\T^3 \to \R^3$ be a stationary solution to the Euler equations with 
$\|B \|_{H^{\ell +2}} = 1$ and define the set $\Omega := \{x \in \T^3 : B(x)  \neq 0\}$. Let $G \in \mathcal{G}$ and let $E_G$ be given by Lemma \ref{BaireLemma}. Let $K$ a compact subset of $E_G$. 
Then, there exist  $\varepsilon_3 \in (0, \varepsilon_1)$ and a small constant $\mathbf{c}>0$ such that
the following holds\footnote{The quantities 
$\varepsilon_3, \mathbf{c}$  
depends on $B, G, K$.}. Let  $\sigma \in (0, \varepsilon_3)$, $\rho\in (0, \mathbf{c} \sigma)$ and set $y := \rho G$. Let $(u,b)$ be the global solution to the \eqref{eq:mhdMoreStandard} equations with initial datum 
\begin{equation}\label{InitialDatumIn MainTHM}
(u^\mathrm{in}, b^\mathrm{in}) = (\phi_1(y), \sigma B + \phi_2(y)) 
\end{equation}
given by Theorem \ref{RealMAINTHM}, that relaxes to $(0, \sigma B)$. 
Then, $(u,b)$ satisfies condition \eqref{ObvCon1}, and also conditions \eqref{NonTRivialityCond}-\eqref{NonTRivialityCond2} are satisfied for all $(\bar{t}, \bar{x}) \in \{ 0\} \times K$. 
Moreover, $\sigma B$ is the pushforward of $b(t_1)$ by the limit flow $\Psi_{t_1 ,\infty}$ defined in \eqref{pushfryuwerungf}, 
for all $t_1 \geq 0$.  
\end{cor}

\begin{proof}

Note that the validity of \eqref{ObvCon1} is an obvious consequence of \eqref{Bound:QuantitativeRelaxation}. Thus, it only remains to verify that the extra conditions \eqref{NonTRivialityCond}-\eqref{NonTRivialityCond2} are satisfied.

Recalling \eqref{fdmkslnjnkdsf64}, we note that 
$$
u(0, x) =   \pi_1 e^{K}\begin{pmatrix}
y(x) \\ 
m(0, x)   
\end{pmatrix} 
=  y(x) +
 \pi_1 (e^{K} - \Id) \begin{pmatrix}
y(x) \\ 
m(0, x)
\end{pmatrix},
$$
and
$$
b(0, x) = \sigma B(x) +  \pi_2 e^{K}\begin{pmatrix}
y(x) \\ 
m(0, x)   
\end{pmatrix}.
$$
We thus have
\begin{equation}\label{fdsji7g784}  
 u(0, x) \wedge b(0, x) 
= y(x) \wedge \sigma B(x) + r(x) =
\rho \sigma G(x) \wedge  B(x) + r(x),
\end{equation}    
where
\begin{align*}
r(x) & := y \wedge \pi_2 e^{K} \begin{pmatrix}
y(x) \\ 
m(0, x)   
\end{pmatrix}
+
 \pi_1 (e^{K} - \Id) \begin{pmatrix}
y(x) \\ 
m(0, x)
\end{pmatrix}  \wedge \sigma B(x) 
\\ 
& + 
 \pi_1 (e^{K} - \Id) \begin{pmatrix}
y(x) \\ 
m(0, x)
\end{pmatrix}  \wedge
\pi_2 e^{K}\begin{pmatrix}
y(x) \\ 
m(0, x)   
\end{pmatrix}.
\end{align*}
Recalling \eqref{fdnjskbjnfhjskhjdbfgsjk2} and using Sobolev embeddings, we have
\begin{align*}
\left|  \pi_j e^{K}\begin{pmatrix}
y(x) \\ 
m(0, x)
\end{pmatrix} \right| & \lesssim   \|y\|_{H^{\ell}} + \|m(0, \cdot)\|_{H^{\ell}}\\
& \overset{\eqref{fmnjsokngfre77565}}{\lesssim}  \|y\|_{H^{\ell}} + \varepsilon_3^2 \|y\|_{H^{\ell}}  \lesssim \rho  .
\end{align*}
Similarly, we use \eqref{fdnjskbjnfhjskhjdbfgsjk3} to deduce
\begin{align*}
\left|  \pi_j \left( e^{K} - \Id \right) \begin{pmatrix}
y(x) \\ 
m(0, x)
\end{pmatrix}  \right| 
 \lesssim
\varepsilon_3 \rho,
\end{align*}
where the implicit constant depends only on $\ell$. Thus, we have that
$$
|r(x)|\leq C( \rho^2+\sigma \varepsilon_3\rho+\varepsilon_3\rho^2)\leq C(\rho^2+\sigma \varepsilon_3\rho).
$$
On the other hand, by Lemma \ref{BaireLemma} and the compactness of $K$ we have
\begin{equation}\label{fdkspjgre84'0123}
|G(x) \wedge  B(x)| > c, \qquad \forall x \in K,
\end{equation}
for some $c >0$, that depends on $B, G, K$. Then, by the triangle inequality we have that 
\begin{align*}
|u(0,x)\wedge b(0,x)|&\geq  \sigma \rho|G(x) \wedge  B(x)|-|r(x)|\\
&\geq c\sigma \rho- C (\rho^2 + \sigma \varepsilon_3\rho).
\end{align*}
The quantity above is strictly positive if
$$
c\sigma > C (\rho + \sigma \varepsilon_3),
$$
which implies
\begin{equation}
C \frac{\rho}{\sigma}+C\varepsilon_3< c.
\end{equation}
Thus, by taking 
\begin{equation}\label{hfeiwbbf664534}
\varepsilon_3<\frac{c}{2C},\qquad \rho<\frac{c}{2C}\sigma,
\end{equation}
the condition \eqref{NonTRivialityCond} is satisfied (by possibly reducing the value of $\varepsilon_3$, i.e. $\varepsilon_3\leq \min\{\e_1, \frac{c}{2C}\}$ and letting $\mathbf{c} := \min\{ \frac{\varepsilon_1}{3}, \frac{c}{2C}\}$).


We now move to the proof of 
$$\partial_t b \big|_{(t,x) = (0,x)} \neq 0, \qquad \forall \, x \in K.$$ 

Since $\dive b = \dive u =0$, we can rewrite the induction equation as 
$$
\partial_t b   = \curl(u \wedge b). 
$$
Recalling \eqref{fdsji7g784} we have 
\begin{align}\label{fnduwi66345787y8543}  
(\partial_t b)(0,x)  
&= \curl( y \wedge \sigma B ) + \curl r  
\\
&
\overset{\dive B = \dive y =0}{=} - \sigma (y \cdot \nabla)B + \sigma (B \cdot \nabla) y + \curl r
\\
&= - \sigma \rho (G \cdot \nabla)B + \sigma \rho (B \cdot \nabla) G  + \curl r
\end{align}    
Proceeding exactly as before, and using the Sobolev embedding $H^\ell \hookrightarrow W^{1,\infty}$ (since $\ell \ge 4$), we have
\begin{equation}\label{fdkspjgre84'012312}
|\curl r(x)| \leq C( \rho^2+\sigma \varepsilon_3\rho+\varepsilon_3\rho^2)\leq C(\rho^2+\sigma \varepsilon_3\rho).
\end{equation}
On the other hand, from \eqref{mnjeifwn77578934} and the compactness of $K$ we deduce 
\begin{equation}\label{u478937y78t439874t}
|  (G \cdot \nabla)B -  (B \cdot \nabla) G| > c,
\end{equation}
for some $c >0$ that depends also on $B, G, K$. Again, the statement follows from \eqref{fnduwi66345787y8543}-\eqref{fdkspjgre84'012312}-\eqref{u478937y78t439874t}, once we restrict to \eqref{hfeiwbbf664534}.
\end{proof}

\appendix

\section{}
Here we formalise the fact, stated in the introduction, that the non-collinearity condition \eqref{NonTRivialityCond} prevents a trivial evolution of the magnetic lines, namely that the image of each magnetic line remains unchanged at any later time. 
 
Before stating and proving it we need to define the image of a magnetic line.
We recall that we have defined a magnetic line $\gamma$ at time $t$ as an integral curve of the magnetic field $b(t, \cdot)$, namely a global solution of
the ODE
$$
\frac{\de}{\de s} \gamma(s) = b(t, \gamma(s)).
$$
We will denote by $\Gamma$ the image of $\gamma$, which is the following subset of $\T^3$
\begin{equation}\label{DefGammaMaiuscolo}
\Gamma := \{ \gamma(s) : s \in \R \}.
\end{equation}

\begin{prop}\label{NonCollinearity} 
Let $u,b : [0, +\infty) \times \T^3\to \R^3$ divergence-free  vector fields of class $C^1$ such that 
$$
\partial_t b = - (u \cdot  \nabla ) b + (b \cdot  \nabla ) u. 
$$
Suppose that for any magnetic line $\gamma$ at time $t= 0$ we have 
\begin{equation}\label{fmdjnksbfye7664374}
\Psi_{0, t} \Gamma = \Gamma, \quad  \mbox{for all $t \geq 0$},
\end{equation}
where $\Gamma$ is the image of $\gamma$ according to definition \eqref{DefGammaMaiuscolo}. 
Then
\begin{equation}\label{otherwise}
b(t, x) \wedge u(t,x) =0, \quad \mbox{for all $(t, x) \in [0, +\infty) \times \T^3$.}
\end{equation} 
\end{prop}

\begin{proof}
Note that the assumption \eqref{fmdjnksbfye7664374} is equivalent to assuming that for all $t , \tau \geq 0$ and for any magnetic line $\gamma$ at time $t$ we have 
$$
\Psi_{t, \tau} \Gamma := \Gamma,
$$
where $\Gamma$ is the image of $\gamma$.\\
  
Let $t \geq 0$ and $x \in \T^3$. 
We can assume 
\begin{equation}\label{Triv1}
b(t,x) \neq 0,
\end{equation} otherwise the identity \eqref{otherwise} is trivially satisfied.

Let $\gamma$ be the magnetic line at time $t$ that passes through $x$, namely 
$$
\gamma : \R \to \T^3,
$$ 
is the 
(unique) global solution of
\begin{equation}\label{Triv2}
\frac{\de}{\de s} \gamma(s) = b(t, \gamma(s)), \qquad \gamma(0) = x.
\end{equation}
%
%
%

Since $\Psi_{t, \tau} \Gamma = \Gamma$ for all $\tau \geq 0$, there exists 
a function 
$$
\kappa : \tau \in  [0, + \infty) \to \kappa(\tau) \in \R,
$$ 
such that 
\begin{equation}\label{implicitdefinitionofgamma}
\Psi_{t, \tau} x = \gamma(\kappa(\tau)).
\end{equation}
The function $\kappa$ depends on $x$ and $\gamma$, we will not emphasize this in the notations.
Note that $$\kappa(t) = 0.$$ 

Note that by \eqref{Triv1}-\eqref{Triv2} we have
$$
\left(\frac{\de}{\de s} \gamma\right)(0) = b(t, x) \neq 0.
$$
Thus, there exists (at least one) $j \in \{ 1,2,3\}$ such that 
$$
\left(\frac{\de}{\de s} \gamma_j\right)(0) = b_j(t, x) \neq 0.
$$
This means that the function 
$$
s \in \R \to \gamma_j(s),
$$ 
is invertible in a neighborhood of $s =0$ with $C^1$ inverse $\gamma_j^{-1}$. Thus we see from 
\eqref{implicitdefinitionofgamma} that 
$$
\kappa(\tau) = \gamma_j^{-1} \circ(  \Psi_{t, \tau} x)_j ,
$$
from which we deduce that $\kappa(\tau)$ is of class $C^1$ in a neighborhood of $\tau = t$.
In this neighborhood of $\tau = t$ the following computation is  thus allowed  
\begin{equation}\label{fnjwi7459234}
u(\tau, \Psi_{t, \tau} x) = \frac{\de}{\de \tau}  \Psi_{t, \tau} x  = \frac{\de}{\de\tau} \gamma(\kappa(\tau)) =   
b(t, \gamma(\kappa(\tau))) \frac{\de}{\de \tau}  \kappa(\tau). 
\end{equation}
Letting $\tau = t$ in this identity and  $\alpha := (\frac{\de}{\de \tau}  \kappa)\big|_{\tau = t}$ we arrive to
$$
u(t, x) = u(t, \Psi_{t, t} x) \overset{\eqref{fnjwi7459234}}{=}
\alpha  b(t, \gamma(\kappa(t)))  = \alpha  b(t, \gamma(0)) = \alpha b(t, x)  ,
$$
from which the statement follows.
\end{proof}

\subsection*{Acknowledgements}
G.C. is supported by INdAM-GNAMPA and by the project PRIN2022 ``Classical equations of compressible fluid mechanics: existence and properties of non-classical solutions''.


\end{document}